\documentclass{article}
 
\usepackage{amsmath,amssymb,amsfonts,amsthm,stmaryrd} 
\usepackage[colorlinks=true]{hyperref}
\usepackage[T1]{fontenc}
\usepackage[utf8]{inputenc} 
\usepackage{enumitem}
\usepackage{xcolor} 
\usepackage{authblk}
\usepackage{bbold} 
\usepackage[a4paper]{geometry} 
\newtheorem{theorem}{Theorem}[section]

\newtheorem{lemma}[theorem]{Lemma}
\newtheorem{remark}[theorem]{Remark}

\newtheorem{proposition}[theorem]{Proposition}
\newtheorem{definition}[theorem]{Definition}  
\newtheorem{corollary}[theorem]{Corollary}

\newcommand{\vertiii}[1]{{\left\vert\kern-0.25ex\left\vert\kern-0.25ex\left\vert #1 
    \right\vert\kern-0.25ex\right\vert\kern-0.25ex\right\vert}}
\newcommand\numberthis{\addtocounter{equation}{1}\tag{\theequation}}
 
\hypersetup{
    colorlinks,
    linkcolor={blue!50!blue},
    citecolor={red}
}

\newtheorem{assumptionF}{Assumption} 
\newtheorem{assumptionO}{Assumption}

\title{A probabilistic study of the kinetic Fokker-Planck equation in cylindrical domains}

\author[1,2]{Tony Lelièvre\thanks{E-mail: tony.lelievre@enpc.fr}} 
\author[1,2]{Mouad Ramil\thanks{E-mail: mouad.ramil@enpc.fr}}
\author[1]{Julien Reygner\thanks{E-mail: julien.reygner@enpc.fr}}
\affil[1]{CERMICS, Ecole des Ponts, Marne-la-Vallée, France}
\affil[2]{MATHERIALS, Inria, Paris, France} 

\date{\today} 

\begin{document}
\maketitle

\begin{abstract}
  We consider classical solutions to the kinetic Fokker-Planck
  equation on a bounded domain~$\mathcal O \subset~\mathbb{R}^d$ in 
position, and we obtain a probabilistic representation of the solutions
  using the Langevin diffusion process with absorbing boundary conditions on
  the boundary of the phase-space cylindrical domain $D = \mathcal O \times \mathbb{R}^d$. Furthermore, a Harnack inequality, as well as a maximum principle, are provided on $D$ for solutions to this kinetic Fokker-Planck equation, together with the existence of a smooth transition 
density for the associated absorbed Langevin process. This transition
density is shown to satisfy an explicit Gaussian upper-bound. Finally,
the continuity and positivity of this transition density at the
boundary of $D$ are also studied. All these results are in particular crucial to
study the behavior of the Langevin diffusion process when it is
trapped in a metastable state defined in terms of positions.


\end{abstract}

  \section{Introduction and motivation}

In statistical physics, the evolution of a molecular system at a given temperature is
typically modeled by the Langevin process:
\begin{equation}\label{eq:Langevin_intro}
  \left\{
    \begin{aligned}
        &\mathrm{d}q_t=M^{-1} p_t \mathrm{d}t , \\
        &\mathrm{d}p_t=F(q_t) \mathrm{d}t -\gamma M^{-1} p_t \mathrm{d}t +\sqrt{2\gamma\beta^{-1}} \mathrm{d}B_t ,
    \end{aligned}
\right.  
\end{equation}
where $d=3N$ for a number $N$ of particles, $(q_t,p_t) \in \mathbb{R}^d \times \mathbb{R}^d$ denotes the respective vectors of positions and momenta of the
particles, $M \in \mathbb{R}^{d \times d}$ is the mass matrix,
$F:\mathbb{R}^d \to \mathbb{R}^d$ is the force acting on the particles, $\gamma >0$ is
the friction parameter, and $\beta^{-1}= k_B T$ with $k_B$
the Boltzmann constant and $T$
the temperature of the system. Such dynamics are used in particular to
compute thermodynamic and dynamic quantities, with numerous
applications in biology, chemistry and materials science. In practice, the system remains trapped for very long times
in subsets of the phase space, called metastable states, see for
example~\cite[Sections 6.3 and 6.4]{LelSto16}. Typically, these states
are defined in terms of positions only, and are thus cylinders of the
form $D=\mathcal{O} \times \mathbb{R}^d$ with positions living in an open set $\mathcal{O}$ of $\mathbb{R}^d$, and momenta in $\mathbb{R}^d$.

In order to understand
the behavior of the stochastic process in such a metastable state, it
is important to study the Langevin diffusion with absorbing boundary
conditions when leaving $D$. This is for example useful to
define the quasi-stationary distribution which can be seen as a ``local
stationary distribution'' within the metastable state, see~\cite{LelRamRey2},~\cite{Ram} and~\cite[Chapter 4]{RamPhD}. This distribution is the
cornerstone of the so-called accelerated dynamics algorithms to sample
metastable processes over long times, see for
example~\cite{perez-uberuaga-voter-15,lelievre-18}. Studying this process is
also important to identify the stationary distribution of the entry and exit points of the process in $D$, see~\cite[Chapter 5]{RamPhD}, which can then be employed to build unbiased estimators of the mean
transition time between metastable states~\cite{baudel-guyader-lelievre-20}. 

However, if the Langevin process is very much used in practice, a complete theory for the related kinetic Fokker-Planck equation, with boundary conditions, has yet to be established. In the literature, weak solutions in a domain have been studied in~\cite{Carillo,Mourrat,nier-18,Har}, as well as classical solutions in~\cite{Vel} for the case $d=1$ with $F=0$ and $\gamma=0$, later extended for $d=2,3$ in~\cite{Hwang}. The main issue for the study of solutions with boundary conditions lies in the fact that, unlike the elliptic case, solutions exhibit a loss of regularity close to a subset of the boundary called singular set, see~\cite{Vel,Hwang} for more details. This work consists in an extension of the framework of classical solutions for a domain in the multi-dimensional case for the kinetic Fokker-Planck operator related to the Langevin process.

The objective of this work is to provide an ensemble of crucial
properties on the absorbed Langevin process and the related kinetic Fokker-Planck equation. In particular, we
will obtain: 
\begin{enumerate}[label=(\roman*)]
  \item a Feynman-Kac type formula to represent
probabilistically the classical solution to a partial differential equation
associated with the Langevin process on a cylindrical domain, which is usually called the kinetic Fokker-Planck equation in the
partial differential equation literature; 
  \item a Harnack inequality as
well as a maximum principle for this partial differential equation;
  \item the existence of a smooth transition density for the absorbed
process, continuous up to the boundary, which satisfies an explicit Gaussian upper bound.
\end{enumerate}
 As will be explained
below, such results are standard for elliptic diffusions (overdamped
Langevin process), but were not proven for the Langevin process 
(which is not elliptic but only hypoelliptic). The
non-ellipticity requires in particular a careful treatment of the
boundary conditions (determining precisely the set of exit
points). The proofs rely on a combination of tools from stochastic
analysis (in particular a parametrix method, inspired by~\cite{Menozzi})
and analysis of partial differential equations (in particular a generalization
of~\cite{Har}). 

\medskip
\textbf{Outline.} In Section~\ref{sec:main_res}, we give the main results, which are then proven
in the subsequent sections. More precisely, Section~\ref{section 2} is
devoted to the proof of the existence of a classical solution to the
kinetic Fokker-Planck equation, as well as its probabilistic
representation. Section~\ref{Section 3} gives the proof of the Harnack
inequality and the maximum principle. In Section~\ref{Section 4}, we provide the proofs of the
existence of a smooth transition density of the absorbed Langevin process as well as
Gaussian upper bounds on the latter. Finally, we prove in Section~\ref{Section Adjoint process
  and compactness} the continuity of this transition density up to the boundary of~$D$, using a so-called adjoint process and time-reversibility arguments. The proofs of intermediate or technical results are postponed to several Appendix sections.

\medskip
\textbf{Notation.} Let us conclude this introductory section with some notation that will
be used in the following. We denote by $x=(q,p)$ generic elements of $\mathbb{R}^{2d}$. The Euclidean norm is denoted by $|\cdot|$, indifferently on $\mathbb{R}^d$ and on $\mathbb{R}^{2d}$, and the scalar product between vectors $\xi$ and $\zeta$ of $\mathbb{R}^d$ or $\mathbb{R}^{2d}$ is denoted by $\xi \cdot \zeta$. The open ball centered at $\xi$ with radius $\rho$ is denoted by $\mathrm{B}(\xi,\rho)$. The distance between a point $\xi$ (resp. a subset $A$) and a subset $B$ is denoted, and defined, by $\mathrm{d}(\xi,B) := \inf_{\zeta \in B} |\xi-\zeta|$ (resp. $\mathrm{d}(A,B) := \inf_{\xi \in A, \zeta \in B} |\xi-\zeta|$).

For a subset $A$ of $\mathbb{R}^d$ or $\mathbb{R}^{2d}$, we denote by:
\begin{enumerate}[label=(\roman*),ref=\roman*]
  \item $\overline{A}$ the closure of $A$, $\partial A$ its boundary and $A^c$ its complement,
  \item $\mathcal{B}(A)$ the Borel $\sigma$-algebra on $A$,
  \item $|A|$ the Lebesgue measure of $A$ (if $A$ is measurable).
\end{enumerate}

For a subset $A$ of $\mathbb{R}^d$, $\mathbb{R}^{2d}$, $\mathbb{R}_+^* \times \mathbb{R}^{2d}$ or $\mathbb{R}_+^* \times \mathbb{R}^{2d} \times \mathbb{R}^{2d}$, we denote by:
\begin{enumerate}[label=(\roman*),ref=\roman*]
    \item for $1 \le p\le \infty$, $\mathrm{L}^p(A)$ the set of
      $\mathrm{L}^p$ scalar-valued functions on $A$ and $\Vert\cdot\Vert_{\mathrm{L}^p(A)}$ the associated norm, 
    \item $\mathcal{C}(A)$ (resp. $\mathcal{C}^b(A)$) the set of
      scalar-valued continuous (resp. continuous and bounded)
      functions on $A$,
    \item for $1 \le k\le \infty$, $\mathcal{C}^k(A)$ (resp.  $\mathcal{C}_c^k(A)$)
      the set of scalar-valued $\mathcal{C}^k$ (resp. $\mathcal{C}^k$ with
      compact support) functions on $A$,
    \item if $A \subset \mathbb{R}^d \text{ or } \mathbb{R}^{2d}$,  $\mathcal{C}^{1,2}(\mathbb{R}_+^{*}\times A)$ the set of
      scalar-valued functions $u(t,\xi)$ on $\mathbb{R}_+^{*}\times A$ such that $u$,
      $\partial_t u$, $\nabla_{\xi} u$ and $\nabla^2_{\xi} u$ exist and are continuous on $\mathbb{R}_+^{*}\times A$.
\end{enumerate}
When we work with vector-valued functions, we use such notations as $\mathcal{C}^\infty(\mathbb{R}^d,\mathbb{R}^d)$ or $\mathcal{C}([0,T],\mathbb{R}^{2d})$. For bounded functions $\phi$, we shall also use the notation $\|\phi\|_\infty$ as a shorthand for the $\mathrm{L}^\infty$ norm.

For $a,b \in \mathbb{R}$, we use the notation $a \wedge b = \min(a,b)$ and $a \vee b = \max(a,b)$. We write $\mathbb{N}=\{0,1,2,\ldots\}$ and $\mathbb{N}^*=\{1,2,\ldots\}$. Integer intervals are denoted by $\llbracket a, b\rrbracket$.

\section{Main results}\label{sec:main_res}

This section presents the main results we obtained. As a motivation, we
first recall in
Section~\ref{subsection 1} some well-known results for
parabolic equations and overdamped Langevin processes, which we then extend to our hypoelliptic and degenerate
framework: the existence of a classical solution to the
kinetic Fokker-Planck equation, as well as its probabilistic
representation using the absorbed Langevin process in Section~\ref{sec:kFP_Lang}; the existence of a
transition density and Gaussian upper bounds for the Langevin process (without absorption) in Section~\ref{sec:gauss_lang}; the existence of a
transition density which is smooth in the domain and continuous up to the boundary, as well as Gaussian upper bounds for the absorbed Langevin process in Section~\ref{sec:gauss_lang_abs}.

\subsection{Parabolic equations and the overdamped Langevin process}
\label{subsection 1}

As an introduction to our results, we briefly review standard material on the probabilistic interpretation, and a few properties of the associated diffusion process, of Initial-Boundary Value Problems for parabolic equations on bounded domains. The prototypical example of such a problem writes
\begin{equation}\label{FIBVP}
  \left\{\begin{aligned}
    \partial_t \overline{u}(t,q) &= \overline{\mathcal{L}}\overline{u}(t,q), && t > 0, \quad q \in \mathcal{O},\\
    \overline{u}(0,q) &= \overline{f}(q), && q \in \mathcal{O},\\
    \overline{u}(t,q) &= \overline{g}(q), && t > 0, \quad q \in \partial \mathcal{O},
  \end{aligned}\right.
\end{equation}
where $\overline{\mathcal{L}}$ is the second-order differential operator
\begin{equation}\label{OL L}
  \overline{\mathcal{L}} = F \cdot \nabla + \frac{\sigma^2}{2} \Delta
\end{equation}
for some vector field $F : \mathbb{R}^d \to \mathbb{R}^d$ and
$\sigma>0$; $\mathcal{O}$ is an open, regular and bounded subset of
$\mathbb{R}^d$; $\overline{f} : \mathcal{O} \to \mathbb{R}$,
$\overline{g}: \partial\mathcal{O} \to \mathbb{R}$ are given initial
and boundary conditions. Under smoothness as well as compatibility
assumptions on the data $\overline f$ and $\overline g$, the following
standard reasoning can be followed: 
\begin{enumerate}
  \item weak solutions $\overline{u}$ can be constructed by variational approach;
  \item parabolic regularization implies that weak solutions are actually smooth;
  \item the smoothness of $\overline{u}$ allows to apply the Itô formula to obtain the probabilistic representation 
  \begin{equation}\label{eq:baru}
    \overline{u}(t,q) = \mathbb{E}\left[ \mathbb{1}_{\overline{\tau}^q_{\partial}>t}
  \overline{f}(\overline{q}^q_t)+\mathbb{1}_{\overline{\tau}^q_{\partial}\leq t}\overline{g}(\overline{q}^q_{\overline{\tau}^q_{\partial}}) \right],
  \end{equation}
  where $(\overline{q}^q_t)_{t \geq 0}$ is the so-called overdamped Langevin process, defined by the stochastic differential equation
  \begin{equation}\label{eq:OL}
    \left\{\begin{aligned}
      &\mathrm{d}\overline{q}^q_t=F(\overline{q}^q_t) \mathrm{d}t+\sigma \mathrm{d}B_t , \\
      &\overline{q}^q_0=q ,
    \end{aligned}\right. 
  \end{equation}
  and
  \begin{equation*}
    \overline{\tau}^q_{\partial} := \inf\{t > 0: \overline{q}^q_t \not\in \mathcal{O}\}.
  \end{equation*}
  This representation implies in particular the uniqueness of classical solutions to~\eqref{FIBVP}.
\end{enumerate}
We refer for example to Evans~\cite[Section~7.1]{Evans} for the first two results, and Friedman~\cite{F,FF,FFF} for the last result. These references also present a Harnack inequality and a maximum principle for~\eqref{FIBVP}. 
In addition, the following facts are closely related with the probabilistic representation formula~\eqref{eq:baru}: 
\begin{enumerate}[label=(\roman*),ref=\roman*]
  \item for any $q \in \mathcal{O}$, the nonnegative measure $\overline{\mathrm{P}}^\mathcal{O}_t(q,\cdot) := \mathbb{P}(\overline{q}^q_t \in \cdot, \overline{\tau}^q_\partial > t)$ has a smooth density $\overline{\mathrm{p}}^\mathcal{O}_t(q,q')$ with respect to the Lebesgue measure on $\mathcal{O}$;
  \item this transition density satisfies the backward and forward Kolmogorov equations
  \begin{equation*}
    \partial_t \overline{\mathrm{p}}^\mathcal{O}_t(q,q') = \mathcal{L}_q \overline{\mathrm{p}}^\mathcal{O}_t(q,q'), \qquad \partial_t \overline{\mathrm{p}}^\mathcal{O}_t(q,q') = \mathcal{L}^*_{q'} \overline{\mathrm{p}}^\mathcal{O}_t(q,q'),
  \end{equation*}
  where $\mathcal{L}^*$ is the formal $\mathrm{L}^2(\mathrm{d} x)$ adjoint of $\mathcal{L}$ and the subscripts $q$, $q'$ in the notation $\mathcal{L}_q$, $\mathcal{L}^*_{q'}$ indicate the variable on which the operator acts;
  \item for all $t > 0$, the function $\overline{\mathrm{p}}^\mathcal{O}_t$ is positive on $\mathcal{O} \times \mathcal{O}$ and has a continuous extension to $\overline{\mathcal{O}} \times \overline{\mathcal{O}}$ which vanishes on $\partial(\mathcal{O} \times \mathcal{O})$.
\end{enumerate}

The aim of this work is to obtain similar results for the Langevin
process~\eqref{eq:Langevin_intro} rather than the overdamped Langevin
process~\eqref{eq:OL}. A technical tool on which several of our results crucially relies is the fact that the transition density of the Langevin process is bounded from above by an explicit Gaussian transition density (see Theorem~\ref{borne densite thm} below). This fact is a natural extension of Baldi's results~\cite[Théorème~4.2]{Baldi} 
for the overdamped Langevin process, based on the so-called parametrix method.

\subsection{Kinetic Fokker-Planck equation and Langevin process}
\label{sec:kFP_Lang}

\subsubsection{The kinetic Fokker-Planck equation}

From now on, we fix $\gamma \in \mathbb{R}$ and $\sigma>0$, and let $F : \mathbb{R}^d \to \mathbb{R}^d$ be a vector field satisfying the following 
\begin{assumptionF}\label{hyp F1}
  $F\in\mathcal{C}^\infty(\mathbb{R}^{d},\mathbb{R}^{d})$. 
\end{assumptionF}

 The kinetic Fokker-Planck operator $\mathcal{L}_{F,\gamma,\sigma}$, simply denoted by $\mathcal{L}$ when there is no ambiguity, writes for $(q,p)\in\mathbb{R}^d\times\mathbb{R}^d$,
\begin{equation}\label{generateur Langevin}
    \mathcal{L} = \mathcal{L}_{F,\gamma,\sigma}=  p\cdot\nabla_q +F(q)\cdot\nabla_p -\gamma  p\cdot\nabla_p+\frac{\sigma^2}{2}\Delta_p . 
\end{equation}
  The operator $\mathcal L$ is the infinitesimal generator of the
  Langevin process~\eqref{eq:Langevin_intro}, in which we
consider the mass to be identity without loss of generality (see
the change of variables in~\cite[Equation~(3.117)]{LelRouSto10}).
As explained in the introduction, in the case $\gamma>0$ and
$\sigma^2=2 \gamma \beta^{-1}$ with $\beta^{-1}=k_BT$ this process is used to describe the
behavior of particles moving in a thermal bath at temperature $T$ and a rate $\gamma$ and subject to the force field $F$. Let us emphasize that in
the following, we consider the general case $\gamma\in\mathbb{R}$ and
$\sigma>0$ not necessarily related to $\gamma$.


Let $\mathcal{O} \subset \mathbb{R}^d$ satisfy
\begin{assumptionO}\label{hyp O}
  $\mathcal{O}$ is open, $\mathcal{C}^2$ and bounded,  
\end{assumptionO}
and consider the following cylindrical domain of $\mathbb{R}^{2d}$:
 $$D:=\mathcal{O}\times\mathbb{R}^d.$$ This
is the natural phase space domain of the Langevin process absorbed when
leaving the set of positions in $\mathcal{O}$. For $q\in\partial\mathcal{O}$, let
$n(q)\in\mathbb{R}^d$ be the unitary outward normal vector to
$\mathcal{O}$ at $q \in \partial\mathcal{O}$. Let us introduce the following partition of $\partial D$:
$$ \Gamma^0=\{ (q,p)\in\partial\mathcal{O}\times\mathbb{R}^d : p\cdot n(q)=0 \} ,$$
$$ \Gamma^+=\{ (q,p)\in\partial\mathcal{O}\times\mathbb{R}^d : p\cdot n(q)>0 \} ,$$
$$ \Gamma^-=\{ (q,p)\in\partial\mathcal{O}\times\mathbb{R}^d : p\cdot n(q)<0 \} .$$

The kinetic Fokker-Planck equation on the domain $D$ with initial condition~$f$ and boundary condition~$g$ is the Initial-Boundary Value Problem
\begin{equation}\label{kFP pb}
  \left\{\begin{aligned}
    \partial_t u(t,x) & =\mathcal{L}u(t,x) && t>0, \quad x\in D ,\\
    u(0,x) &=f(x) && x\in D ,\\
    u(t,x) &=g(x)  && t>0, \quad x\in\Gamma^+ .
  \end{aligned}\right. 
\end{equation}
Notice that, in contrast with the Initial-Boundary Value Problem~\eqref{FIBVP} associated with the overdamped Langevin process~\eqref{eq:OL}, the boundary condition only applies on the subset
$\Gamma^+$ of the boundary $\partial D$. We refer the reader
to~\cite[Chapter 11]{FF} for a study of boundary conditions
for diffusions degenerating at the boundary of the domain $D$. Even if
these results do not apply in our case, since the diffusion is
degenerated everywhere, they still provide useful intuition on the
behavior of the stochastic process at the boundary. 

\begin{remark}\label{rk:kFPname}
  Our convention to call the Initial-Boundary Value
  Problem~\eqref{FIBVP} `kinetic Fokker-Planck equation' is the same
  as that in the work by Armstrong and Mourrat~\cite{Mourrat}. As we shall see in Theorem~\ref{Solution PDE} below, in the homogeneous case $g=0$ the solution $u$ describes the evolution of the semigroup of the Langevin process absorbed at $\partial D$.
  
  In the literature, it seems more standard to reserve this denomination for the dual equation describing the evolution of the law of the Langevin process absorbed at $\partial D$~\cite{Vel,Hwang,Carillo,Har}. This equation writes
  \begin{equation}\label{eq:kFPstar}
        \left\{\begin{aligned}
            \partial_t v(t,x) & =\mathcal{L}^*v(t,x) && t>0, \quad x\in D ,\\
            v(0,x) &=f(x) && x\in D ,\\
            v(t,x) &=0  && t>0, \quad x\in\Gamma^-,
        \end{aligned}\right. 
  \end{equation}
  where $f$ now denotes the initial distribution of the process and $\mathcal{L}^*$ is the formal adjoint of $\mathcal{L}$ in $\mathrm{L}^2(\mathrm{d}x)$, given by 
\begin{equation}\label{generateur adjoint}
    \mathcal{L}^*=-p\cdot\nabla_q-F(q)\cdot\nabla_p +\gamma \mathrm{div}_p(p  \cdot )+\frac{\sigma^2}{2}\Delta_p .
\end{equation} 
  The boundary condition is set on $\Gamma^-$ because, in this probabilistic interpretation, $v$ denotes the density of a system of particles which are allowed to escape the domain $D$ but cannot re-enter it.
  
  We also note that, in the case where $F=-\nabla V$, $\gamma>0$ and $\sigma=\sqrt{2\gamma\beta^{-1}}$, the denomination `kinetic Fokker-Planck equation' may refer to the variant of~\eqref{eq:kFPstar} in which the operator $\mathcal{L}^*$ in~\eqref{eq:kFPstar} is replaced by 
  \begin{equation*}
    \mathcal{L}^\dagger:=-p\cdot\nabla_q+\nabla V(q)\cdot\nabla_p -\gamma p . \nabla_p+\gamma \beta^{-1}\Delta_p,
  \end{equation*}
  and which describes the evolution of the density of the Langevin process with respect to the invariant measure $\exp(-\beta(V(q)+\frac{|p|^2}{2}))$ rather than the Lebesgue measure. See for instance~\cite{DMS} and the discussion in~\cite[Section~7]{Villani}.
  
  In any case, there is a natural link between the solutions to~\eqref{kFP pb} and~\eqref{eq:kFPstar}, as it is easily checked that if $u$ solves~\eqref{kFP pb} with operator $\mathcal{L}=\mathcal{L}_{F,\gamma,\sigma}$ and homogeneous boundary condition $g=0$, then the function $v$ defined by 
  \begin{equation*}
      v(t,(q,p)) := \mathrm{e}^{-d\gamma t} u(t,(q,-p))
  \end{equation*}
  solves~\eqref{eq:kFPstar} with operator $\mathcal{L}^*=\mathcal{L}^*_{F,-\gamma,\sigma}$. Therefore, our results for the equation~\eqref{kFP pb} also apply to~\eqref{eq:kFPstar}.
\end{remark} 

For the sake of clarity, let us make precise the standard notions of
solutions we will need in the following.
\begin{definition}[Classical solutions]\label{def:class}
A function $u$ is a classical solution to~\eqref{kFP pb} if $u \in
   \mathcal{C}^{1,2}(\mathbb{R}_+^*\times D) \cap \mathcal{C}((\mathbb{R}_+\times(
    D\cup\Gamma^+))\setminus(\{0\}\times\Gamma^+)) $ and $u$ satisfies~\eqref{kFP pb}.
  \end{definition}
  Notice that the regularity on $u$ in this definition
is required for the boundary value and initial
condition in~\eqref{kFP pb} to hold in a classical sense. We will also
use the notion of distributional solutions to
\begin{equation}\label{eq:kfp_dist}
  \partial_t u=\mathcal{L}u \qquad \text{on $\mathbb{R}_+^*\times D$},
\end{equation}
without neither initial condition nor boundary value.

\begin{definition}[Distributional solutions]\label{def:dist}
A distribution $u$ on $\mathbb{R}_+^*\times D$  is a distributional
solution of~\eqref{eq:kfp_dist} if for all
$\Phi\in\mathcal{C}_c^\infty(\mathbb{R}_+^*\times D)$,
$$\iint_{\mathbb{R}^{*}_+\times
D} u(t,x)\left(\partial_t\Phi(t,x)+\mathcal{L}^*\Phi(t,x)
\right)\mathrm{d}t \, \mathrm{d}x=0, $$
where the operator $\mathcal{L}^*$ is defined in~\eqref{generateur adjoint}.  
\end{definition}
A distributional solution to~\eqref{eq:kfp_dist} differs from a
classical solution to~\eqref{kFP pb} in two ways: interior regularity,
and boundary regularity (required to define boundary and initial
conditions in~\eqref{kFP pb}). On the one hand, additional regularity is necessary to properly define the initial and boundary values of distributional solutions, see for example the works~\cite{nier-18,Mourrat,Carillo,Vel}. On the other hand, concerning interior regularity,
it is actually known that distributional solutions
of~\eqref{eq:kfp_dist} are $\mathcal{C}^{\infty}(\mathbb{R}_+^* \times
D)$ by hypoellipticity. Let us recall these standard results, see~\cite{HM}.
\begin{definition}[Hypoellipticity]\label{def:hypoell}
A differential operator $\mathcal{G}$ is said to be hypoelliptic on
an open set $A\subset\text{$\mathbb{R}^d$, $\mathbb{R}^{2d}$ or $\mathbb{R}_+^* \times \mathbb{R}^{2d}$}$ if for all
$f\in\mathcal{C}^\infty(A)$ and $u$ a distributional solution to
$\mathcal{G}u=f$ on $A$  then $u\in\mathcal{C}^\infty(A)$. 
\end{definition}
It is well known that under Assumption \ref{hyp F1} the operators
$\mathcal{L}$ and  $\mathcal{L}^*$ (resp. $\partial_t-\mathcal{L}$ and
$\partial_t-\mathcal{L}^*$)  are hypoelliptic on $D$ (resp. on
$\mathbb{R}_+^*\times D$), see for example~\cite[Section
2.3.1]{LelSto16} and references therein.  

\subsubsection{Probabilistic representation of classical solution}

In this work, we are interested in the existence and uniqueness of classical solutions of~\eqref{kFP pb}, see Theorem~\ref{Solution PDE} below. Besides we will show that this solution admits a
probabilistic representation in terms of the Langevin process $(X^x_t=(q^x_t,p^x_t))_{t\geq0}$, described by its
position $q^x_t$ and velocity $p^x_t$ at time $t$ and defined by the
following SDE: 
\begin{equation}\label{Langevin}
  \left\{
    \begin{aligned}
        &\mathrm{d}q^x_t=p^x_t \mathrm{d}t , \\
        &\mathrm{d}p^x_t=F(q^x_t) \mathrm{d}t -\gamma  p^x_t \mathrm{d}t+\sigma \mathrm{d}B_t ,\\
        &(q^x_0,p^x_0)=x .
    \end{aligned}
\right.
\end{equation}
Let $\tau^x_{\partial}$ be the first exit time from $D$ of the process
$(X^x_t)_{t\geq0}$, i.e.  $$\tau^x_{\partial}=\inf \{t>0: X^x_t\notin
D\}= \inf \{t>0: q^x_t\notin \mathcal O\}.$$

Under Assumption~\ref{hyp F1}, the drift coefficient $(q,p)\mapsto
(p,F(q)-\gamma p)$ in \eqref{Langevin} is locally Lipschitz continuous on $\mathbb{R}^{2d}$, therefore~\eqref{Langevin} admits a unique strong solution $(X^x_t)_{t \geq 0}$, which is {\em a priori} only defined
up to some explosion time $\tau^x_\infty$ by~\cite[Theorem IV.3.1]{Ikeda}. Under the additional Assumption~\ref{hyp O}, this drift coefficient is globally Lipschitz continuous on $D$, and thus the process exits the set $D$ before the explosion time almost surely, so that the solution $(X^x_t)_{t \geq 0}$ is at least well-defined until $\tau^x_{\partial}$. Since observing the process only up to the time $\tau^x_\partial$ amounts to imposing an absorbing boundary condition on $\partial D$, this justifies the following definition.

\begin{definition}[Absorbed Langevin process]\label{defi:absorbed}
  Under Assumptions~\ref{hyp F1} and~\ref{hyp O}, the process $(X^x_t)_{0 \leq t \leq \tau^x_\partial}$ is called the absorbed Langevin process.
\end{definition}

Since $(X^x_t)_{0 \leq t \leq \tau^x_\partial}$ is a solution to the SDE~\eqref{Langevin}, it is a continuous-time Markov process with almost surely continuous sample paths. Besides, since the coefficients in~\eqref{Langevin} are locally bounded on $\mathbb{R}^{2d}$, then $(X^x_t)_{0 \leq t \leq \tau^x_\partial}$ satisfies the strong Markov property, see~\cite[Theorem 4.20 p. 322]{Karatzas}. 

\begin{remark}\label{solution dans D}
  Friedman's uniqueness result~\cite[Theorem 5.2.1.]{F}
ensures that the trajectories $(X^x_t)_{0 \le t \le \tau^x_\partial}$ do not depend on the values of $F$ outside of $\mathcal{O}$. Therefore, whenever we are interested in quantities which only depend on the absorbed Langevin process, there is no loss of generality in modifying $F$ outside of $\mathcal{O}$ so that it satisfies the following strengthening of Assumption~\ref{hyp F1}:
\begin{assumptionF}\label{hyp F}
$F\in\mathcal{C}^\infty(\mathbb{R}^{d},\mathbb{R}^{d})$ and $F$ is bounded and globally Lipschitz continuous on $\mathbb{R}^d$.
\end{assumptionF} 
\end{remark}

Under Assumption~\ref{hyp F}, the drift coefficient $(q,p) \mapsto (p, F(q)-\gamma p)$ in \eqref{Langevin} is globally Lipschitz continuous, with a Lipschitz constant which we shall denote by $C_\mathrm{Lip}$, and therefore the strong solution $(X^x_t)_{t \geq 0}$ (without absorbing boundary condition) is defined globally in time.

In order to describe the probabilistic representation of the classical
solution to~\eqref{kFP pb}, we first state a trajectorial result on
the solution $(q_t^x,p_t^x)_{t \geq 0}$ of~\eqref{Langevin} for $x\in
\mathbb{R}^{2d}\setminus \Gamma^0$. We prove that, almost surely,
the process $(q_t^x,p_t^x)_{t \geq 0}$ does not reach the set $\Gamma^0$ in finite time. In other words, the set $\Gamma^0$ is non attainable in the sense of~\cite[Chapter~11.1]{FF}.

\begin{proposition}[Non-attainability of $\Gamma^0$]\label{non attainability}
  Under Assumptions~\ref{hyp O} and~\ref{hyp F}, for all $x\in\mathbb{R}^{2d}\setminus \Gamma^0$, 
  \begin{equation*}
    \mathbb{P}\left(\exists t > 0: X^x_t \in \Gamma^0\right) = 0.
  \end{equation*}
\end{proposition} 

Using this non-attainability result we are able to characterize more precisely the exit event from~$D$ in the following proposition. 

\begin{proposition}[Attainability of the boundary]\label{prop:tau}
  Let Assumptions~\ref{hyp O} and~\ref{hyp F1} hold.
  \begin{enumerate}[label=(\roman*),ref=\roman*]
    \item\label{it:tau:reg} If $x\in\Gamma^+\cup\Gamma^0$, then for all $\epsilon>0$, $(q^x_t,p^x_t)_{t\geq0}$ visits $\overline{D}^c$ almost surely on $[0,\epsilon]$, i.e.
    \begin{equation}\label{sortie immédiate}
     \mathbb{P}\left(\exists t \in [0,\epsilon] : q^x_t \in \overline{\mathcal{O}}^c\right)=1 .
      \end{equation}
    In particular, $\tau^x_\partial=0$ almost surely.
    \item\label{it:tau:sing} If $x\in D\cup\Gamma^-$, then
      $\tau^x_\partial>0$ almost surely and one has
  \begin{equation}\label{singular set}
    \mathbb{P}\left(p^x_{\tau^x_\partial}\cdot n(q^x_{\tau^x_\partial}) \leq0 , \tau^x_\partial<\infty \right)=0.
    \end{equation}
  \end{enumerate}
\end{proposition} 
\begin{remark}  One can actually prove that for all $x\in D\cup \Gamma^-$, $\tau^x_\partial<\infty$ almost surely (since it admits exponential moments see~\cite[Remark~2.20]{LelRamRey2}). Therefore, the equality \eqref{singular set} can be written equivalently: for all $x\in D\cup \Gamma^-$,
  $\mathbb{P}(X^x_{\tau^x_\partial} \in
    \Gamma^-)=1$. 
\end{remark} 

Proposition~\ref{prop:tau} implies that for all $t \geq 0$, almost
surely, if $\tau^x_{\partial}>t$ then $X^x_t \in D\cup\Gamma^-$, and
if $\tau^x_{\partial}\leq t$ then $X^x_{\tau^x_{\partial}} \in
\Gamma^+\cup\Gamma^0$. This ensures that the definition of the
function $u$ in Equation~\eqref{u} below is legitimate.

We are now in position to state the main result of this section, namely the existence
of a unique classical solution to the kinetic Fokker-Planck equation~\eqref{kFP pb}, and its probabilistic representation.

\begin{theorem}[Classical solution and probabilistic representation for the kinetic Fokker-Planck equation~\eqref{kFP pb}]\label{Solution PDE} Under Assumptions~\ref{hyp O} and~\ref{hyp F1}, let $f\in\mathcal{C}^b(D\cup\Gamma^-)$ and $g\in\mathcal{C}^b(\Gamma^+\cup\Gamma^0)$, and define the function $u$ on $\mathbb{R}_+\times \overline{D}$ by
\begin{equation}\label{u}
    u:(t,x)\mapsto\mathbb{E}\left[ \mathbb{1}_{\tau^x_{\partial}>t} f(X^x_t) +\mathbb{1}_{\tau^x_{\partial}\leq t} g(X^x_{\tau^x_{\partial}}) \right].
    \end{equation}
  Then we have the following results:
\begin{enumerate}[label={\rm(\roman*)},ref=\roman*]
  \item\label{it:ibvp:val} Initial and boundary values: the function $u$ satisfies
 \begin{equation*}
u(0,x)=\left\{
\begin{aligned}
    f(x) &\quad \text{if $x\in D\cup\Gamma^-$,}\\
    g(x) &\quad \text{if $x\in \Gamma^+\cup\Gamma^0$},
\end{aligned}
\right. 
\end{equation*}
and 
\begin{equation*}
  \forall t>0, \quad \forall x\in\Gamma^+\cup\Gamma^0, \qquad u(t,x)=g(x).
\end{equation*}
  \item\label{it:ibvp:cont} Continuity: $u \in \mathcal{C}^b((\mathbb{R}_+\times\overline{D}) \setminus (\{0\}\times(\Gamma^+\cup\Gamma^0)) )$, and if $f$ and $g$ satisfy the compatibility condition
\begin{equation}\label{compatibility cond}
    x\in\overline{D}\mapsto\mathbb{1}_{x\in
      D\cup\Gamma^-}f(x)+\mathbb{1}_{x \in \Gamma^+\cup\Gamma^0}g(x)\in\mathcal{C}^b(\overline{D}),
      \end{equation}
then $u\in\mathcal{C}^b(\mathbb{R}_+\times\overline{D})$.
  \item\label{it:ibvp:reg} Interior regularity: $u \in \mathcal{C}^\infty(\mathbb{R}_+^*\times D)$ and, for all $t>0$, $x \in D$,
  \begin{equation}\label{Langevin PDE}
    \partial_tu(t,x)=\mathcal{L}u(t,x).
    \end{equation}
\end{enumerate}  
The three items~\eqref{it:ibvp:val}, \eqref{it:ibvp:cont}
and~\eqref{it:ibvp:reg} show that $u$ defined by~\eqref{u} is a classical solution to~\eqref{kFP pb} in the sense of
    Definition~\ref{def:class}. We also have the following uniqueness
    result for classical solutions to~\eqref{kFP pb}:
\begin{enumerate}[label={\rm (\roman*)},ref=\roman*]
  \setcounter{enumi}{3} 
  \item\label{it:ibvp:uniq} Let $v$ be a
    classical solution to~\eqref{kFP pb} in the sense of
    Definition~\ref{def:class}. If, for all $T>0$, $v$ is bounded on the set  $[0,T] \times D$, then $v(t,x)=u(t,x)$ for all $(t,x)\in (\mathbb{R}_+\times(D\cup\Gamma^+))\setminus(\{0\}\times \Gamma^+)$.
\end{enumerate}  
\end{theorem}
Proposition~\ref{prop:tau} and Theorem~\ref{Solution PDE} are proven
in Section~\ref{section 2}. The proof essentially follows the same three-step structure as for the probabilistic representation formula~\eqref{eq:baru} of solutions to the Initial-Boundary Value Problem~\eqref{FIBVP}: we construct weak (actually, distributional) solutions to~\eqref{Langevin PDE} by parabolic approximation, use the hypoellipticity of $\partial_t - \mathcal{L}$ to obtain the smoothness of such solutions and apply the It\^o formula to identify the solution with $u$ defined by~\eqref{u}. We mention here that, regarding the first step, a variational approach to~\eqref{kFP pb}, closer to the spirit of the proof outlined in Section~\ref{subsection 1} than our parabolic approximation argument, was recently developed by Armstrong and Mourrat~\cite{Mourrat}.

\begin{remark}[Extension to bounded and measurable functions]\label{rk:extLinf}
Let $f : D \cup \Gamma^- \to \mathbb{R}$ be measurable and bounded, and take $g \equiv 0$. Using an elementary regularization argument, which can be rigorously justified with the help of Theorem~\ref{thm density intro} and Corollary~\ref{Rq densite estimation} stated below, it is easy to check that the function $u$ defined by~\eqref{u} remains a distributional solution of~\eqref{Langevin PDE} on $\mathbb{R}_+^* \times D$ and therefore, by hypoellipticity, still satisfies Assertion~\eqref{it:ibvp:reg}.
\end{remark}

\begin{remark}
  It is easy to check that,  using the same proofs, these results also hold for a time-dependent boundary
condition $g(t,x)\in\mathcal{C}^b(\mathbb{R}_+ \times
(\Gamma^+\cup\Gamma^0))$, replacing~\eqref{u} by \linebreak  $
u(t,x)=\mathbb{E}[ \mathbb{1}_{\tau^x_{\partial}>t} f(X^x_t)
  +\mathbb{1}_{\tau^x_{\partial}\leq t} g(t-\tau^x_\partial, X^x_{\tau^x_{\partial}}) ]$. We stick to a
time-homogeneous boundary conditions for the ease of notation.
\end{remark}

\begin{remark}
Note that the compatibility condition~\eqref{compatibility cond} is necessary to ensure the continuity of the solution at $\{0\}\times \overline{D}$. Furthermore, it is known that even for smooth and compatible boundary and initial conditions, one cannot expect the solution to be smooth at the boundary $\partial D$: it has indeed been shown in the one-dimensional case ($d=1$) that the solution is only expected to be Hölder-continuous on the singular set $\Gamma^0=\{ (q,p)\in\partial\mathcal{O}\times\mathbb{R}^d : p\cdot n(q)=0 \} ,$ and not differentiable, see~\cite{Vel}.
\end{remark}

\subsubsection{Maximum principle and Harnack inequality}

As an immediate consequence of Theorem~\ref{Solution PDE}, under Assumptions~\ref{hyp O} and~\ref{hyp F1}, if $f\geq0$ on $D$ and $g\geq0$ on $\Gamma^+$ then it follows that any solution $v$
of~\eqref{kFP pb} which satisfies the conditions
of item~\eqref{it:ibvp:uniq} in Theorem~\ref{Solution PDE} is such that $v\geq0$ on
$\mathbb{R}_+\times\overline{D}$. We now state stronger forms of this
maximum principle, as well as a Harnack inequality, under the following supplementary assumption on the domain $\mathcal{O}$. 
\begin{assumptionO}\label{hyp:conn}
  The set $\mathcal{O}$ is connected.
\end{assumptionO}

\begin{theorem}[Maximum principle]\label{maximum principle}
Let Assumptions~\ref{hyp F1}, \ref{hyp O} and~\ref{hyp:conn} hold. Let $u\in\mathcal{C}^{1,2}(\mathbb{R}_+^*\times D)$ such that $\partial_tu-\mathcal{L}u\leq0$ on $\mathbb{R}_+^*\times D$.
\begin{enumerate}[label=(\roman*),ref=\roman*]
    \item\label{weak principle} Assume that $u\in \mathcal{C}^b((\mathbb{R}_+\times( D\cup\Gamma^+))\setminus(\{0\}\times\Gamma^+))$, then 
\begin{equation}\label{sup de u}
    \sup_{\mathbb{R}_+^*\times D}u(t,x)=\sup_{(\{t=0\}\times D)\cup(\mathbb{R}_+^*\times \Gamma^+)}u(t,x).
    \end{equation} 
    \item\label{strong principle} Assume that $u$ reaches a maximum at a point $(t_0,x_0)\in\mathbb{R}_+^*\times D$, then
      $$\forall t\leq t_0,\quad \forall x\in D,\qquad u(t,x)=u(t_0,x_0).$$
\end{enumerate}
\end{theorem}
Theorem~\ref{maximum principle} is proven
in Section~\ref{ss:maximum}.
Let us conclude this section by stating a Harnack inequality. In the literature, a variant of the Harnack inequality was obtained in the stationary case for hypoelliptic operators, see~\cite{Bon69}. In~\cite{Har}, the authors prove a Harnack inequality
satisfied by distributional solutions to $\partial_tu=\mathcal{L}u$ in
sufficiently small domains. Here, we extend their result on a general compact set of $D$ in the
following theorem.  The proof uses in particular the concept of Harnack chains from~\cite{Polidoro}.  

\begin{theorem}[Harnack inequality]\label{Harnack} Let Assumptions~\ref{hyp F1}, \ref{hyp O} and~\ref{hyp:conn} hold. For any compact set $K\subset
  D$, $\epsilon>0$ and $T>0$, there exists a constant
  $C_{K,\epsilon,T}>0$ such that for any non-negative distributional
  solution $u$ of $\partial_tu=\mathcal{L}u$  on $\mathbb{R}_+^*\times D$ (in the sense of Definition~\ref{def:dist}), for all $t\geq \epsilon$,
\begin{equation}\label{Harnack inequality}
    \sup_{x\in K}u(t,x)\leq C_{K,\epsilon,T}\inf_{x\in K}u(t+T,x).
\end{equation}
\end{theorem} 
Theorem~\ref{Harnack} is proven in Section~\ref{ss:harnack}.

\begin{remark}\label{rmk Harnack global}
For a given compact set $K\subset
  \mathbb{R}^{2d}$, one can find an open set $\mathcal{O}$ satisfying Assumptions~\ref{hyp O} and~\ref{hyp:conn} such that $K\subset\mathcal{O}\times\mathbb{R}^d$. Therefore, Theorem \ref{Harnack} implies the following statement: under Assumption~\ref{hyp F1}, for any compact set $K\subset
  \mathbb{R}^{2d}$, $\epsilon>0$ and $T>0$, there exists a constant
  $C_{K,\epsilon,T}>0$ such that, for all $t\geq \epsilon$, the Harnack inequality~\eqref{Harnack inequality} holds for any non-negative distributional
  solution $u$ of $\partial_tu=\mathcal{L}u$ on the whole space $\mathbb{R}_+^*\times \mathbb{R}^{2d}$. 
\end{remark}

\subsection{Kolmogorov equations and Gaussian bounds for the Langevin process}\label{sec:gauss_lang}

In this section, we consider the Langevin process~\eqref{Langevin}
without absorbing boundary condition. We recall that under Assumption \ref{hyp F}, for all
$x\in\mathbb{R}^{2d}$ the equation \eqref{Langevin} admits a unique
strong global solution $(X^x_t)_{t\geq0}$ on $\mathbb{R}^{2d}$. Let us
introduce the associated transition kernel $\mathrm{P}_t$:
$$\forall t>0, \quad \forall x\in \mathbb{R}^{2d}, \quad \forall A\in\mathcal{B}(\mathbb{R}^{2d}), \qquad \mathrm{P}_t(x,A):=\mathbb{P}(X^x_t\in A).$$

The following standard properties of $\mathrm{P}_t(x,\cdot)$ are for
example proven in~\cite[Corollary 7.2]{RB} (see also Equations~(153)
and~(155) there).
\begin{proposition}[Kolmogorov equations for the Langevin process]\label{prop:kolmo-langevin}
  Under Assumption \ref{hyp F}, there exists a function 
\begin{equation}\label{global density}
    (t,x,y)\mapsto \mathrm{p}_t(x,y) \in\mathcal{C}^\infty(\mathbb{R}_+^{*}\times \mathbb{R}^{2d}\times \mathbb{R}^{2d})
    \end{equation}
such that for all $t>0$, $x\in\mathbb{R}^{2d}$
and $A\in\mathcal{B}(\mathbb{R}^{2d})$, $$\mathrm{P}_t(x,A)=\int_A
\mathrm{p}_t(x,y) \mathrm{d}y .$$
Moreover, this transition density satisfies the backward and forward Kolmogorov equations: 
\begin{enumerate}[label=(\roman*),ref=\roman*]
    \item $(t,x)\mapsto \mathrm{p}_t(x,y)$ satisfies $\partial_t \mathrm{p}=\mathcal{L}_x\mathrm{p}$ on $\mathbb{R}_+^{*}\times\mathbb{R}^{2d}$,
    \item $(t,y)\mapsto \mathrm{p}_t(x,y)$ satisfies $\partial_t \mathrm{p}=\mathcal{L}_y^*\mathrm{p}$ on $\mathbb{R}_+^{*}\times\mathbb{R}^{2d}$.
\end{enumerate}
  \end{proposition}
The subscripts in $\mathcal{L}_x$ and $\mathcal{L}^*_y$ indicate on
  which variables the differential operators apply. We will also need the following immediate corollary of Proposition~\ref{prop:kolmo-langevin}.
\begin{corollary}[Atoms of $\tau^x_{\partial}$]\label{rq densité}
Under the assumptions of  Proposition~\ref{prop:kolmo-langevin}, for
all $x\in \overline{D}$, for all~$t>0$,
$$\mathbb{P}(\tau^x_{\partial}= t)\leq
\mathbb{P}(q^x_t\in\partial\mathcal{O})=0.$$
\end{corollary}

Theorem \ref{borne densite thm} below states that the
transition density $\mathrm{p}_t(x,y)$ admits an explicit Gaussian
upper bound. This has already been proven in \cite{Menozzi} in the case $\gamma=0$, using the
parametrix method, if
Assumption \ref{hyp F} holds. In this case, the drift
coefficient $(q,p)\mapsto F(q)$ for the velocity-related SDE in~\eqref{Langevin} is indeed globally Lipschitz continuous and bounded,
and thus satisfies the hypothesis required in \cite{Menozzi}. If
$\gamma\neq0$ the drift coefficient $(q,p)\mapsto F(q)-\gamma p$ is
not bounded in $\mathbb{R}^{2d}$, but we adapt the idea of the
parametrix method  to obtain a Gaussian upper bound in this case, see Section~\ref{subsection Gaussian}.
 
\begin{theorem}[Gaussian upper bound on $\mathrm{p}_t$]\label{borne densite thm}
Under Assumption \ref{hyp F}, the transition density $\mathrm{p}_t(x,y)$ of the Langevin process $(X^x_t)_{t\geq0}$ satisfying \eqref{Langevin} is such that for all $\alpha\in (0,1)$, there exists $c_\alpha>0$, depending only on $\alpha$, such that for all $T>0$ and $t\in(0,T]$, for all $x,y\in\mathbb{R}^{2d}$,
\begin{equation} \label{borne densité}
    \mathrm{p}_t(x,y)\leq\frac{1}{\alpha^d} \sum_{j=0}^\infty \frac{\left( \Vert F\Vert_\infty c_\alpha(1+\sqrt{\gamma_- T}) \sqrt{\pi t}  \right)^j}{\sigma^j \Gamma\left(\frac{j+1}{2}\right)} \widehat{\mathrm{p}}^{(\alpha)}_{t}(x,y),
\end{equation}
where $\gamma_-=\max(-\gamma,0)$ is the negative part of $\gamma\in\mathbb{R}$,
$\Gamma$ is the Gamma function and
$\widehat{\mathrm{p}}^{(\alpha)}_{t}(x,y)$ is the transition density
of the Gaussian process with infinitesimal generator $\mathcal{L}_{0,\gamma,\sigma/\sqrt{\alpha}}$ defined in~\eqref{generateur Langevin}, see also Equations~\eqref{expr densite}-\eqref{densite p^alpha} below for explicit formulas.
\end{theorem}   

\subsection{Kolmogorov equations for the absorbed Langevin process}\label{sec:gauss_lang_abs}

Let us define the transition kernel $\mathrm{P}_t^D$ for the absorbed Langevin process $(X^x_t)_{0 \leq t \leq \tau^x_\partial}$: 
\begin{equation}    \label{kernel D}
    \forall t \geq 0, \quad \forall x \in \overline{D}, \quad \forall A \in \mathcal{B}(D), \qquad \mathrm{P}_t^D(x,A):=\mathbb{P}(X^x_t\in A,\tau^x_\partial>t) .
\end{equation}
It is easy to see that for any $t \geq 0$, $x \in \overline{D}$ and $A \in \mathcal{B}(D)$,
\begin{equation}\label{eq:PtDPt}
  \mathrm{P}^D_t(x,A) \leq \mathrm{P}_t(x,A).
\end{equation}
The next theorem is the equivalent of Proposition~\ref{prop:kolmo-langevin} for the transition kernel $\mathrm{P}^D_t(x,\cdot)$.
\begin{theorem}[Kolmogorov equations for the absorbed Langevin process]\label{thm density intro}
    Under Assumptions \ref{hyp O} and \ref{hyp F1}, 
there exists a function
    $$(t,x,y)\mapsto \mathrm{p}_t^D(x,y)
    \in\mathcal{C}^\infty(\mathbb{R}_+^{*}\times D\times
    D)\cap\mathcal{C}(\mathbb{R}_+^*\times\overline{D}\times\overline{D})$$
    such that for all $t>0$, $x\in \overline{D}$ and
    $A\in\mathcal{B}(D)$, 
    $$\mathrm{P}_t^D(x,A)=\int_A \mathrm{p}_t^D(x,y) \mathrm{d}y.$$
Moreover, this transition density $\mathrm{p}^D_t$ satisfies the backward and forward Kolmogorov  equations:
    \begin{enumerate}[label=(\roman*),ref=\roman*]
    \item $(t,x)\mapsto \mathrm{p}_t^D(x,y)$ satisfies $\partial_t \mathrm{p}^D=\mathcal{L}_x\mathrm{p}^D$ on $\mathbb{R}_+^{*}\times D$,
    \item $(t,y)\mapsto \mathrm{p}_t^D(x,y)$ satisfies $\partial_t \mathrm{p}^D=\mathcal{L}_y^*\mathrm{p}^D$ on $\mathbb{R}_+^{*}\times D$.
\end{enumerate} 
Finally, for all $t>0$,  $(x,y) \in \overline{D} \times \overline{D}$,
\begin{enumerate}[label=(\roman*),ref=\roman*]
    \item $\mathrm{p}^D_t(x,y)=0$ if $x\in \Gamma^+\cup\Gamma^0$ or $y\in\Gamma^-\cup\Gamma^0$,
    \item under Assumption~\ref{hyp:conn}, $\mathrm{p}^D_t(x,y)>0$ if $x\notin \Gamma^+\cup\Gamma^0$ and $y\notin\Gamma^-\cup\Gamma^0$.
\end{enumerate}
\end{theorem}
\begin{remark}
   The positivity of the transition density $\mathrm{p}_t^D$ is
   obtained using the Harnack inequality from
   Theorem~\ref{Harnack}. In addition, as explained above, if the unabsorbed Langevin
   process is defined globally in time, the existence of a smooth
   transition density $\mathrm{p}_t$ satisfying the Kolmogorov
   equations follows from~\cite[Corollary 7.2]{RB}. Using
   Remark~\ref{rmk Harnack global}, one can then apply the Harnack
   inequality to ensure the positivity of $\mathrm{p}_t$ on
   $\mathbb{R}^{2d}\times \mathbb{R}^{2d}$. This extends the
   positivity result from~\cite[Corollary 3.3]{PositiveDensity} to non-polynomial coefficients.
\end{remark}
The proof of the existence of a transition density
$\mathrm{p}_t^D(x,y)$ which is smooth on $\mathbb{R}_+^{*}\times D
\times D$ and satisfies the backward and forward
Kolmogorov equations will be done in Propositions \ref{existence
  densite} and \ref{regularité densité}, in Section~\ref{smooth
  density subsection}. 
The results on the continuity at the boundary as well as on the positivity of
$\mathrm{p}^D_t$ are stated in Theorem~\ref{boundary property
  density}. They crucially rely on the introduction and study of a so-called adjoint process to~\eqref{Langevin}, which is carried out in Section~\ref{Section Adjoint process and compactness}.
  
Under Assumptions \ref{hyp O} and \ref{hyp F1}, the transition density $\mathrm{p}^D_t$ only depends on the values of $F$ inside $\mathcal{O}$, see Remark~\ref{solution dans D}. Therefore, up to a modification of $F$ outside of $\mathcal{O}$ so that $F$ satisfies~\ref{hyp F}, the following useful corollary of  Theorem~\ref{thm density intro}, the inequality~\eqref{eq:PtDPt} and Theorem~\ref{borne densite thm} can be deduced.
\begin{corollary}[Gaussian upper bound on $\mathrm{p}^D_t$]\label{Rq densite estimation}
  Under Assumptions \ref{hyp O} and \ref{hyp F1}, $\mathrm{p}^D_t(x,y)$ satisfies the Gaussian upper bound of Theorem~\ref{borne densite thm}, in which the quantity $\|F\|_\infty$ is replaced by~$\|F\|_{\mathrm{L}^\infty(D)}$.  
\end{corollary}
 
\section{The Initial-Boundary Value problem}
\label{section 2}
This section is devoted to the proof of Theorem \ref{Solution PDE}. In
this theorem, Assertion~\eqref{it:ibvp:val} is an immediate
consequence of Proposition~\ref{prop:tau} which is proven in Section
\ref{section proof lemmata}. Assertions~\eqref{it:ibvp:uniq},
\eqref{it:ibvp:cont} and~\eqref{it:ibvp:reg} are respectively proven
in Sections~\ref{ss:ibvp:uniq}, \ref{ss:ibvp:cont}
and~\ref{ss:ibvp:reg}. Finally, Section~\ref{section proof lemmata} is
devoted to the proofs of Propositions~\ref{non attainability} and~\ref{prop:tau}.

All the results in this section are proven under Assumptions~\ref{hyp O} and~\ref{hyp
  F}. Since the final statement of Theorem~\ref{Solution PDE} only
depends on the values of $F$ in $\mathcal{O}$ (see
Remark~\ref{solution dans D}), this statement remains valid if Assumption~\ref{hyp F} is replaced by Assumption~\ref{hyp F1}. 

Before proceeding, we introduce the following notation: we let $\mathrm{d}_\partial$ be the Euclidean distance function to the boundary $\partial\mathcal{O}$ from a point in $\mathcal{O}$, i.e.
\begin{equation}\label{distance boundary}
\mathrm{d}_\partial:q\in\mathbb{R}^d\mapsto
\begin{cases}
  \mathrm{d}(q,\partial\mathcal{O})
&\text{if $q\in\mathcal{O}$,}\\
0 &\text{if $q \not\in\mathcal{O}$,} 
\end{cases}
\end{equation}
and $\mathrm{d}_{\overline{\mathcal{O}}}$ be the Euclidean distance to
the compact set $\overline{\mathcal{O}}$, 
\begin{equation}\label{distance compact}
\mathrm{d}_{\overline{\mathcal{O}}}:q\in\mathbb{R}^d\mapsto
\mathrm{d}(q,\overline{\mathcal{O}}).
\end{equation} These distance functions are $1$-Lipschitz continuous.

\subsection{\texorpdfstring{Uniqueness: proof of Assertion~\eqref{it:ibvp:uniq} in Theorem~\ref{Solution PDE}}{}}\label{ss:ibvp:uniq}

Assertion~\eqref{it:ibvp:uniq} in Theorem~\ref{Solution PDE} follows from the application of the Itô formula. Although the general argument is well-known, we detail its application here in order to emphasize two specificities of our framework: the fact that the domain $D$ is unbounded, and the fact that the kinetic nature of the Langevin process $(X^x_t)_{t \geq 0}$ makes boundary conditions on $v$ only necessary on the subset $\Gamma^+$ of $\partial D$.
 
\begin{proof}[Proof of Assertion~\eqref{it:ibvp:uniq} in Theorem~\ref{Solution PDE}]
Let $g\in\mathcal{C}^b(\Gamma^+\cup\Gamma^0)$ and $f\in\mathcal{C}^b(D\cup\Gamma^-)$. Let $v$ be a solution of \eqref{kFP pb}, satisfying the conditions of Assertion~\eqref{it:ibvp:uniq} in Theorem~\ref{Solution PDE}.

Let $x\in D$. Let $(X^x_t=(q^x_t,p^x_t))_{t\geq0}$ be the strong solution of \eqref{Langevin} on $\mathbb{R}^{2d}$. For $k>0$, let $V_k$ be the following open and bounded subset of $D$  
\begin{equation}\label{Vk}
    V_k:=\left\{(q,p)\in D : \vert p\vert< k, \mathrm{d}_\partial(q)>\frac{1}{k} \right\}.
\end{equation}
Let us choose $k$ large enough so that $x\in V_k$. Let $\tau^{x}_{V_k^c}$ be the following stopping time:
\[\begin{aligned}
\tau^x_{V_k^c}&=\inf \{t>0: X^x_t\notin V_k\} .
\end{aligned}\]
Let $t>0$ and $s\in[0,t)$. Since
$v\in\mathcal{C}^{1,2}(\mathbb{R}_+^{*}\times D)$, Itô's formula
applied to the process $(v(t-r,X^x_{r}))_{0\leq r\leq s}$ between $0$
and $s\land \tau^x_{V_k^c}$ yields: almost surely, for $s\in[0,t)$
\[\begin{aligned}
v(t-s\land \tau^x_{V_k^c},X^x_{s\land \tau^x_{V_k^c}})&=v(t,x)+\sigma \int_0^{s\land \tau^x_{V_k^c}}\nabla_p v(t-r,X^x_{r})\cdot \mathrm{d}B_r,
\end{aligned}\]
since $\partial_t v-\mathcal{L}v=0$ on $\mathbb{R}_+^{*}\times D$.
Besides, since $\nabla_p v$ is continuous on the compact set $[t-s,t]\times \overline{V_k}$, hence bounded on $(t-s,t)\times V_k$, the stochastic integral in the right-hand side is a martingale and its expectation vanishes. Therefore
\begin{equation}\label{ito v}
\begin{aligned}
v(t,x)&=\mathbb{E}\left[ v(t-s\land \tau^x_{V_k^c},X^x_{s\land \tau^x_{V_k^c}}) \right]\\
&=\mathbb{E}\left[ \mathbb{1}_{\tau^x_{V_k^c}>s} v(t-s,X^x_s) + \mathbb{1}_{\tau^x_{V_k^c}\leq s} v(t-\tau^x_{V_k^c},X^x_{\tau^x_{V_k^c}})  \right] .
\end{aligned}
\end{equation}
Now we would like to let $k\rightarrow\infty$ and then $s\rightarrow
t$ in \eqref{ito v}.

First let us prove the following
limit $$\lim_{k\rightarrow\infty}\tau^x_{V_k^c}=\tau^x_{\partial}\quad\text{almost
  surely.}$$
The sequence $(\tau^x_{V_k^c})_{k\geq1}$ is an increasing sequence of random variables, therefore it converges almost surely to $\sup_{k\geq1}\tau^x_{V_k^c}$. Besides, using the continuity of the trajectories of $(X^x_t)_{t\geq0}$, one gets, for all $r>0$,
\[\begin{aligned}
\left\{\sup_{k\geq1}\tau^x_{V_k^c}>r\right\}&=\left\{\exists  k\geq1: \tau^x_{V_k^c}>r\right\}\\
&=\left\{\exists  k\geq1: \sup_{u\in[0,r]}\vert p^x_u\vert< k, \inf_{u\in[0,r]} \mathrm{d}_\partial(q^x_u)>\frac{1}{k}\right\}\\
&=\left\{\sup_{u\in[0,r]}\vert p^x_u\vert<\infty, \inf_{u\in[0,r]} \mathrm{d}_\partial(q^x_u)>0\right\}\\
&=\left\{\sup_{u\in[0,r]}\vert p^x_u\vert<\infty, \tau^x_\partial>r\right\} .
\end{aligned}\]
 For all $r>0$, we have that, almost surely, $\sup_{u\in[0,r]}\vert p^x_u\vert<\infty$. Therefore, $\sup_{k\geq1}\tau^x_{V_k^c}>r$ if and only if $\tau^x_\partial>r$, that is to say $\sup_{k\geq1}\tau^x_{V_k^c}=\tau^x_\partial$ almost surely.
As a result, one gets $\lim_{k\rightarrow\infty}\tau^x_{V_k^c}=\tau^x_{\partial}$ almost surely. Consequently, since $k\mapsto \tau^x_{V_k^c}$ is increasing and $t\mapsto\mathbb{1}_{t>s}$ is left-continuous, one has  almost surely that for all $s>0$,
\begin{equation}\label{cv indicatrice k}
    \mathbb{1}_{\tau^x_{V_k^c}>s}\underset{k\rightarrow\infty}{\longrightarrow}\mathbb{1}_{\tau^x_{\partial}>s}.
\end{equation}

Second, notice that $X^x_{\tau^x_{\partial}}\in \Gamma^+$ almost surely on the event $\{\tau^x_{\partial}\leq s\}$ by Proposition~\ref{prop:tau} since $x\in D$. Consequently, since $v\in\mathcal{C}(\mathbb{R}^{*}_+\times( D\cup\Gamma^+))$ and $v=g$ on $\mathbb{R}^{*}_+\times\Gamma^+$
\begin{equation}\label{cv tau k}
    \mathbb{1}_{\tau^x_{V_k^c}\leq s} v(t-\tau^x_{V_k^c},X^x_{\tau^x_{V_k^c}})\underset{k\rightarrow\infty}{\longrightarrow} \mathbb{1}_{\tau^x_{\partial}\leq s} g(X^x_{\tau^x_{\partial}})\quad\text{ almost surely.}
    \end{equation}
We now use \eqref{cv indicatrice k} and~\eqref{cv tau k} to apply the
dominated convergence theorem to \eqref{ito v} when $k$ goes to
infinity, using the fact that $v$  is assumed to be bounded on $[0,t]\times D$. Therefore, one gets for $x\in D$ and  $s\in[0,t)$,
\begin{equation}\label{ito v 2}
\begin{aligned}
v(t,x)&=\mathbb{E}\bigg[ \mathbb{1}_{\tau^x_{\partial}>s} v(t-s,X^x_s) \bigg]+\mathbb{E}\bigg[ \mathbb{1}_{\tau^x_{\partial}\leq s} g(X^x_{\tau^x_{\partial}}) \bigg].
\end{aligned}
\end{equation}

Finally, let us consider the limit $s\rightarrow t$ in \eqref{ito v
  2}. Notice
that $$\mathbb{1}_{\tau^x_{\partial}>s}\underset{s\rightarrow
  t}{\longrightarrow}\mathbb{1}_{\tau^x_{\partial}>t}
\quad\text{almost surely,} $$ using Corollary~\ref{rq densité} (which
follows from Proposition~\ref{prop:kolmo-langevin}, which holds independently from
the results proven in this section, as will become clear in Section~\ref{section proof lemmata}). Therefore, the second term in the right-hand side of the equality~\eqref{ito v 2} satisfies by dominated convergence,
$$\mathbb{E}\bigg[ \mathbb{1}_{\tau^x_{\partial}\leq s} g(X^x_{\tau^x_{\partial}}) \bigg]\underset{s\rightarrow t}{\longrightarrow} \mathbb{E}\bigg[ \mathbb{1}_{\tau^x_{\partial}\leq t} g(X^x_{\tau^x_{\partial}}) \bigg] .$$
Moreover the continuity of the trajectories of $(X^x_s)_{s\geq0}$ and the continuity of $v$ on $\mathbb{R}_+\times D$ ensure that $$\mathbb{1}_{\tau^x_{\partial}>s} v(t-s,X^x_s)\underset{s\rightarrow t}{\longrightarrow}\mathbb{1}_{\tau^x_{\partial}>t} v(0,X^x_t)=\mathbb{1}_{\tau^x_{\partial}>t} f(X^x_t)\quad\text{almost surely.}$$
Finally, taking the limit $s\rightarrow t$ in \eqref{ito v 2}, the dominated convergence theorem ensures that 
$$v(t,x)=\mathbb{E}\big[ \mathbb{1}_{\tau^x_{\partial}>t} f(X^x_t)+ \mathbb{1}_{\tau^x_{\partial}\leq t} g(X^x_{\tau^x_{\partial}}) \big] $$
for all $t>0$, $x\in D$. This concludes the proof of Assertion~\eqref{it:ibvp:uniq} in Theorem~\ref{Solution PDE} using the continuity of $v$ in $(\mathbb{R}_+\times(D\cup\Gamma^+))\setminus(\{0\}\times \Gamma^+)$.
\end{proof}
 
\subsection{\texorpdfstring{Continuity: proof of Assertion~\eqref{it:ibvp:cont} in Theorem~\ref{Solution PDE}\label{ss:ibvp:cont}}{}}

The proof of Assertion~\eqref{it:ibvp:cont} in Theorem~\ref{Solution PDE} relies on the the following lemmas. We recall that under Assumption~\ref{hyp F}, we denote by $C_\mathrm{Lip}$ the Lipschitz constant of the drift of~\eqref{Langevin}.

\begin{lemma}[Gronwall Lemma]\label{couplage Lemma} Under Assumption~\ref{hyp F}, for all $t\geq0$, for all $x,y \in \mathbb{R}^{2d}$, one has
\begin{equation}\label{couplage}
    \sup_{s\in[0,t]}\vert X^x_s- X^y_s\vert \leq \vert x- y \vert \mathrm{e}^{C_\mathrm{Lip}t}\quad\text{ almost surely}.\\
\end{equation}
Besides, for $(t_0,x_0)\in\mathbb{R}_+\times\mathbb{R}^{2d}$ we have
\begin{equation}\label{cv jointe t,x}
    X^{x}_{t}\underset{(t,x)\rightarrow (t_0,x_0)}{{\longrightarrow}}X^{x_0}_{t_0}\quad\text{almost surely.}
\end{equation} 
\end{lemma}
The estimate~\eqref{couplage} follows from a standard application of the Gronwall Lemma which we do not detail here. The joint continuity statement~\eqref{cv jointe t,x} is then a straightforward consequence of the continuity of the trajectories of $(X^{x}_{t})_{t\geq0}$.

\begin{lemma}[Continuity of the exit event indicator]\label{cv indicatrices lemma}
Under Assumptions~\ref{hyp O} and~\ref{hyp F}, let $(t,x)\in(\mathbb{R}_+\times\overline{D})\setminus (\{0\}\times(\Gamma^+\cup\Gamma^0))$. Let $(t_n,x_n)_{n\geq1}$ be a sequence of $\mathbb{R}_+\times\overline{D}$ converging towards $(t,x)$. Then one has
\begin{equation}\label{cv indicatrices 1}
    \mathbb{1}_{\tau^{x_n}_{\partial}> t_n}\underset{n\rightarrow \infty}{{\longrightarrow}}\mathbb{1}_{\tau^{x}_{\partial}> t}\quad\text{almost surely.}
    \end{equation}
    \end{lemma}
 
\begin{proof} We prove the convergence \eqref{cv indicatrices 1} on the events $\{\tau^{x}_{\partial}<t\}$, $\{\tau^{x}_{\partial}=t\}$ and $\{\tau^{x}_{\partial}>t\}$, separately.

\medskip \noindent\textbf{Step 1}. Let us start by proving \eqref{cv
  indicatrices 1} on the event $\{\tau^{x}_{\partial}< t\}$
(necessarily $t>0$). Let $\epsilon\in(0,t-\tau^x_\partial)$. If $x\in\Gamma^+\cup\Gamma^0$, then by Proposition~\ref{prop:tau}, $\tau^x_\partial=0$ and there exists $s\in(0,\epsilon]$ such that $\mathrm{d}_{\overline{\mathcal{O}}}(q_s^x)>0$. If $x\in D\cup\Gamma^-$, then $X^x_{\tau^{x}_{\partial}}\in\Gamma^+$ almost surely  by Proposition~\ref{prop:tau}. By the strong Markov property and Proposition~\ref{prop:tau}, there exists again $s\in (\tau^{x}_{\partial},\tau_\partial^x+\epsilon]$ such that $\mathrm{d}_{\overline{\mathcal{O}}}(q_s^x)>0$.

Since $x_n\underset{n\longrightarrow \infty}{{\longrightarrow}}x$, there exists $N_1\geq1$ such that for all $n\geq N_1$, $\vert x_n- x \vert \leq \frac{\mathrm{d}_{\overline{\mathcal{O}}}(q_s^x)}{2} \mathrm{e}^{-C_{\mathrm{Lip}}(\tau_\partial^x+\epsilon)}$. As a result using Lemma \ref{couplage Lemma} along with the fact that the distance function $\mathrm{d}_{\overline{\mathcal{O}}}$ is $1$-Lipschitz continuous, it follows that for all $n\geq N_1$ $$\mathrm{d}_{\overline{\mathcal{O}}}(q_s^{x_n})\geq \frac{\mathrm{d}_{\overline{\mathcal{O}}}(q_s^x)}{2}>0 . $$
Therefore, we have $\tau_\partial^{x_n} < s \leq \tau^x_\partial + \epsilon$. In addition, the convergence $t_n\underset{n\longrightarrow \infty}{{\longrightarrow}}t$ implies that there exists $N_2\geq1$ such that for all $n\geq N_2$, one has $t_n\geq \tau_\partial^x+\epsilon$, since $\tau_\partial^x+\epsilon<t$. As a result for $n\geq \max(N_1,N_2)$,
$$\tau^{x_n}_{\partial}<\tau^{x}_{\partial}+\epsilon\leq t_n . $$
Hence the convergence \eqref{cv indicatrices 1} on the event $\{\tau^{x}_{\partial}< t\}$.

\medskip \noindent \textbf{Step 2}. Let us consider the event
$\{\tau^{x}_{\partial}= t\}$. For $t=0$ and $x\in D\cup\Gamma^-$,
$\mathbb{P}(\tau^{x}_{\partial}= 0)=0$ by
Proposition~\ref{prop:tau}. Moreover for $t>0$,
$\mathbb{P}(\tau^{x}_{\partial}= t)=0$ by Corollary~\ref{rq densité}. As a result, it is not necessary to prove the convergence \eqref{cv indicatrices 1} on the event $\{\tau^{x}_{\partial}= t\}$ as the latter is negligible.

\medskip \noindent \textbf{Step 3}. Finally, it only remains to prove
the convergence \eqref{cv indicatrices 1} on the event
$\{\tau^{x}_{\partial}> t\}$. Let $t^{'}:=\frac{1}{2}(\tau^{x}_{\partial}+t)$. Since $t_n\underset{n\longrightarrow \infty}{{\longrightarrow}}t$, there exists $N_1\geq1$ such that for $n\geq N_1$, $t_n\leq t^{'}$.

On the one hand, if $x\in D$, then by the continuity of the trajectories of $(q^x_s)_{s
  \in [0,t']}$, one has $\inf_{s\in[0,t^{'}]}\mathrm{d}_\partial(q_s^x)>0$. By Lemma \ref{couplage Lemma} and the fact that the distance function $\mathrm{d}_\partial$ is $1$-Lipschitz continuous, there exists $N_2\geq1$ such that for $n\geq N_2$, $$\inf_{s\in[0,t^{'}]} \mathrm{d}_\partial(q_s^{x_n})>0,$$ which yields $\tau^{x_n}_{\partial}> t^{'}$. As a result, for $n\geq\max(N_1,N_2)$, $$\tau^{x_n}_{\partial}> t^{'}\geq t_n . $$

On the other hand, if $x\in\partial D$, then necessarily $x\in\Gamma^-$, otherwise $\tau^{x}_{\partial}=0$ by Proposition~\ref{prop:tau}. Then for all $k\geq1$, $$\inf_{s\in\left[\frac{1}{k},t'\right]}\mathrm{d}_\partial(q_s^x)>0 . $$
Using Lemma~\ref{couplage Lemma} again, we get that there exists $M_k\geq1$ such that for $n\geq M_k$, $\tau^{x_n}_{\partial}> t^{'}$ or $\tau^{x_n}_{\partial}\leq\frac{1}{k}$. Assume that there exists an unbounded sequence $(n_k)_{k\geq1}$ such that $\tau^{x_{n_k}}_{\partial}\leq\frac{1}{k}$. Then $x\in \Gamma^+\cup\Gamma^0$ since $X^{x_{n_k}}_{\tau^{x_{n_k}}_{\partial}}\in\Gamma^+$ and $X^{x_{n_k}}_{\tau^{x_{n_k}}_{\partial}}\underset{n\longrightarrow \infty}{{\longrightarrow}}x$ by Lemma \ref{couplage Lemma}, which is in contradiction with the fact that $x\in \Gamma^-$. As a result there exists $N_2\geq1$ such that for all $n\geq N_2$, $\tau^{x_n}_{\partial}> t^{'}$. Hence, for $n\geq\max(N_1,N_2)$, $\tau^{x_n}_{\partial}> t_n$. This concludes the proof of the convergence \eqref{cv indicatrices 1} on the event $\{\tau^{x}_{\partial}> t\}$.
\end{proof}

\begin{remark}
Notice that the convergence \eqref{cv indicatrices 1} cannot be satisfied for $(t,x)\in\{0\}\times(\Gamma^+\cup\Gamma^0)$. Indeed, if $x \in \Gamma^+ \cup \Gamma^0$ and $(x_n)_{n \geq 1}$ is a sequence of elements of $D$ which converges towards $x$, then by Proposition~\ref{prop:tau} we have $\mathbb{1}_{\tau^{x_n}_{\partial}> 0}=1$ almost surely while $\mathbb{1}_{\tau^x_{\partial}> 0}=0$ almost surely. 
\end{remark}

\begin{remark}\label{rk:cvtauxn}
  Let us take $t_n=t>0$ for all $n \geq 1$ in Lemma~\ref{cv
    indicatrices lemma}. Then we get that for any sequence $(x_n)_{n
    \geq 1}$ of elements of $\overline{D}$ converging to some $x \in
  \overline{D}$, $\mathbb{1}_{\tau^{x_n}_{\partial}> t}$ converges
  almost surely to $\mathbb{1}_{\tau^x_{\partial}> t}$. Using the monotonicity of the functions $t \mapsto
  \mathbb{1}_{\tau^{x_n}_{\partial}> t}$ and $t \mapsto
  \mathbb{1}_{\tau^x_{\partial}> t}$, we deduce that almost surely, for any $t>0$ such that $t \not=\tau^x_\partial$, $\mathbb{1}_{\tau^{x_n}_{\partial}> t}$ converges
  almost surely to $\mathbb{1}_{\tau^x_{\partial}> t}$. Integrating in time, we conclude that $\tau^{x_n}_{\partial}$ converges almost surely to
  $\tau^x_{\partial}$.
\end{remark}

We are now in position to prove Assertion~\eqref{it:ibvp:cont} in Theorem~\ref{Solution PDE}.

\begin{proof}[Proof of Assertion~\eqref{it:ibvp:cont} in Theorem~\ref{Solution PDE}]The proof is divided into two steps. In the first step we show that $u$ is continuous on $(\mathbb{R}_+\times\overline{D})\setminus(\{0\}\times(\Gamma^+\cup\Gamma^0))$, and in the second step we show that if $f$ and $g$ satisfy the compatibility condition~\eqref{compatibility cond} then $u$ is continuous on $\mathbb{R}_+ \times \overline{D}$.

\medskip \noindent\textbf{Step 1}.
 Let $(t,x)\in (\mathbb{R}_+\times\overline{D})\setminus(\{0\}\times(\Gamma^+\cup\Gamma^0))$. Let $(t_n,x_n)_{n\geq1}$ be a sequence in $\mathbb{R}_+\times\overline{D}$ converging to $(t,x)$. Let us prove that
 \begin{equation}\label{convergence u_n to u}
     u(t_n,x_n)\underset{n\longrightarrow \infty}{{\longrightarrow}}u(t,x).
 \end{equation}
To this aim, let us study the difference $\vert  u(t_n,x_n)-u(t,x)\vert$. It follows from the expression \eqref{u} of $u$ and the triangle inequality that
\begin{equation}\label{majoration u_n-u}
    \vert  u(t_n,x_n)-u(t,x)\vert\leq\mathbb{E}\left[\left\vert f(X^{x_n}_{t_n})\mathbb{1}_{\tau^{x_n}_{\partial}>t_n}-f(X^{x}_{t})\mathbb{1}_{\tau^{x}_{\partial}>t}\right\vert \right]+\mathbb{E}\left[\left\vert g(X^{x_n}_{\tau^{x_n}_{\partial}})\mathbb{1}_{\tau^{x_n}_{\partial}\leq t_n}-g(X^{x}_{\tau^{x}_{\partial}})\mathbb{1}_{\tau^{x}_{\partial}\leq t}\right\vert \right].    
  \end{equation}
  
Let us start with the first term in the right-hand side of the inequality above. We have that
\begin{align*}
    &\left\vert f(X^{x_n}_{t_n})\mathbb{1}_{\tau^{x_n}_{\partial}>t_n}-f(X^{x}_{t})\mathbb{1}_{\tau^{x}_{\partial}>t}\right\vert\\
    &=\left\vert\mathbb{1}_{\tau^{x_n}_{\partial}>t_n,\tau^{x}_{\partial}>t}\left[f(X^{x_n}_{t_n})-f(X^{x}_{t})\right]+f(X^{x_n}_{t_n})\mathbb{1}_{\tau^{x_n}_{\partial}>t_n,\tau^{x}_{\partial}\leq t}-f(X^{x}_{t})\mathbb{1}_{\tau^{x}_{\partial}>t,\tau^{x_n}_{\partial}\leq t_n}\right\vert\\
    &\leq\mathbb{1}_{\tau^{x_n}_{\partial}>t_n,\tau^{x}_{\partial}>t}\left\vert f(X^{x_n}_{t_n})-f(X^{x}_{t})\right\vert+\Vert f\Vert_\infty \left|\mathbb{1}_{\tau^{x_n}_{\partial}>t_n}-\mathbb{1}_{\tau^{x}_{\partial}>t}\right|,
\end{align*}
since $\mathbb{1}_{\tau^{x_n}_{\partial}>t_n,\tau^{x}_{\partial}\leq t}+\mathbb{1}_{\tau^{x}_{\partial}>t,\tau^{x_n}_{\partial}\leq t_n}=|\mathbb{1}_{\tau^{x_n}_{\partial}>t_n}-\mathbb{1}_{\tau^{x}_{\partial}>t}|$. By Lemmas~\ref{couplage Lemma} and \ref{cv indicatrices lemma}, since $f\in\mathcal{C}^b(D\cup\Gamma^-)$, it follows by the dominated convergence theorem that $\mathbb{E}[\vert f(X^{x_n}_{t_n})\mathbb{1}_{\tau^{x_n}_{\partial}>t_n}-f(X^{x}_{t})\mathbb{1}_{\tau^{x}_{\partial}>t}\vert]\underset{n\rightarrow \infty}{{\longrightarrow}}0$.

Let us now consider the second term in the right-hand side of the inequality \eqref{majoration u_n-u}. We have that
\begin{align*}
    &\left\vert g(X^{x_n}_{\tau^{x_n}_{\partial}})\mathbb{1}_{\tau^{x_n}_{\partial}\leq t_n}-g(X^{x}_{\tau^{x}_{\partial}})\mathbb{1}_{\tau^{x}_{\partial}\leq t}\right\vert\\
    &=\left\vert \mathbb{1}_{\tau^{x_n}_{\partial}\leq t_n, \tau^x_{\partial}\leq t} \left[g(X^{x_n}_{\tau^{x_n}_{\partial}})-g(X^{x}_{\tau^{x}_{\partial}})\right] + g(X^{x_n}_{\tau^{x_n}_{\partial}})\mathbb{1}_{\tau^{x_n}_{\partial}\leq t_n, \tau^x_{\partial} > t} - g(X^{x}_{\tau^{x}_{\partial}})\mathbb{1}_{\tau^{x_n}_{\partial}> t_n, \tau^x_{\partial}\leq t}\right\vert\\
    &\leq \mathbb{1}_{\tau^{x_n}_{\partial}\leq t_n, \tau^x_{\partial}\leq t}\left\vert g(X^{x_n}_{\tau^{x_n}_{\partial}})-g(X^{x}_{\tau^{x}_{\partial}})\right\vert + \Vert g\Vert_\infty \left\vert \mathbb{1}_{\tau^{x_n}_{\partial}> t_n}-\mathbb{1}_{\tau^{x}_{\partial}>t} \right\vert,
\end{align*}
so that we deduce again from Lemmas~\ref{couplage Lemma} and \ref{cv indicatrices lemma} along with Remark~\ref{rk:cvtauxn} and the dominated convergence theorem, since $g \in\mathcal{C}^b(\Gamma^+\cup\Gamma^0)$,  that
\begin{equation*}
  \mathbb{E}\left[\left\vert g(X^{x_n}_{\tau^{x_n}_{\partial}})\mathbb{1}_{\tau^{x_n}_{\partial}\leq t_n}-g(X^{x}_{\tau^{x}_{\partial}})\mathbb{1}_{\tau^{x}_{\partial}\leq t}\right\vert \right] \underset{n\rightarrow \infty}{{\longrightarrow}}0,
\end{equation*}
which completes the proof of~\eqref{convergence u_n to u}.

\medskip \noindent\textbf{Step 2}. Assume now that $f$ and $g$ satisfy the compatibility condition \eqref{compatibility cond}.  Let $x\in\Gamma^+\cup\Gamma^0$ and $(t_n,x_n)_{n\geq1}$ be a sequence in $\mathbb{R}_+\times\overline{D}$ converging to $(0,x)$. Let us prove that $$u(t_n,x_n)\underset{n\longrightarrow \infty}{{\longrightarrow}}u(0,x)=g(x).$$
We have that 
\[\begin{aligned}
\vert   u(t_n,x_n)-g(x) \vert&= \left\vert  \mathbb{E}\left[ \left(f(X^{x_n}_{t_n})-g(x)\right) \mathbb{1}_{\tau^{x_n}_{\partial}>t_n} \right]+\mathbb{E}\left[ \left(g(X^{x_n}_{\tau^{x_n}_{\partial}})-g(x)\right) \mathbb{1}_{\tau^{x_n}_{\partial}\leq t_n} \right]  \right\vert\\
&\leq  \mathbb{E}\big[\big\vert f(X^{x_n}_{t_n})-g(x)\big\vert\mathbb{1}_{\tau^{x_n}_{\partial}>t_n}\big]+\mathbb{E}\big[ \big\vert g(X^{x_n}_{\tau^{x_n}_{\partial}})-g(x)\big\vert \mathbb{1}_{\tau^{x_n}_{\partial}\leq t_n} \big] .\\
\end{aligned}\]
It follows from the compatibility condition \eqref{compatibility cond}
and Lemma \ref{couplage Lemma} that
$\mathbb{1}_{\tau^{x_n}_{\partial}>t_n}\big\vert
f(X^{x_n}_{t_n})-g(x)\big\vert \underset{n\rightarrow
  \infty}{{\longrightarrow}}0$
almost surely. Therefore, using the dominated convergence theorem, the first term in the right-hand side of the  inequality above converges to $0$.
Furthermore, on the event $\{\tau^{x_n}_{\partial}\leq t_n\}$, it follows  by Lemma \ref{couplage Lemma} that
\begin{align*}
    \left\vert X^{x_n}_{\tau^{x_n}_{\partial}}-x\right\vert&\leq\left \vert X^{x_n}_{\tau^{x_n}_{\partial}}-X^{x}_{\tau^{x_n}_{\partial}}\right\vert+\left\vert X^{x}_{\tau^{x_n}_{\partial}}-x\right\vert\\
    &\leq \underbrace{\vert x_n-x\vert \mathrm{e}^{C_\mathrm{Lip} t_n}}_{\underset{n\rightarrow \infty}{{\longrightarrow}}0}+\underbrace{\left\vert X^{x}_{\tau^{x_n}_{\partial}}-x\right\vert}_{\underset{n\rightarrow \infty}{{\longrightarrow}}0\text{ a.s.}},
\end{align*}
since $\tau^{x_n}_{\partial}\leq t_n\underset{n\rightarrow
  \infty}{{\longrightarrow}}0$. As a result, $\mathbb{E}\big[
\big\vert g(X^{x_n}_{\tau^{x_n}_{\partial}})-g(x)\big\vert
\mathbb{1}_{\tau^{x_n}_{\partial}\leq t_n} \big]\underset{n\rightarrow
  \infty}{{\longrightarrow}}0$ since
$g\in\mathcal{C}^b(\Gamma^+\cup\Gamma^0)$. Hence
$u(t_n,x_n)\underset{n\longrightarrow
  \infty}{{\longrightarrow}}g(x)$. This concludes the proof of Assertion~\eqref{it:ibvp:cont} in Theorem~\ref{Solution PDE}.
\end{proof}
 
\subsection{\texorpdfstring{Interior regularity: proof of Assertion~\eqref{it:ibvp:reg} in Theorem~\ref{Solution PDE}}{}}\label{ss:ibvp:reg}
The link between functions of the form of $u$ defined by~\eqref{u} and parabolic problems of the form~\eqref{Langevin PDE} is standard for uniformly elliptic operators in bounded domains with compatible initial and boundary conditions, see for instance~\cite[Chapter~6]{F}. In order to extend this link to the degenerate operator $\mathcal{L}$, we proceed by approximating~\eqref{Langevin PDE} by the uniformly elliptic problem
\begin{equation}\label{eq:edpueps}
  \partial_t u_\epsilon = \mathcal{L} u_\epsilon + \epsilon \Delta_q u_\epsilon.
\end{equation} 

Let $\epsilon>0$ and $(\widetilde{B}_t)_{t\geq0}$ be a $d-$dimensional Brownian motion independent of $(B_t)_{t\geq0}$. Under Assumption \ref{hyp F}, for all $x \in D$ we denote by $(X_t^{x,\epsilon}=(q^{x,\epsilon}_t,p^{x,\epsilon}_t))_{t\geq0}$ the strong solution of 
\begin{equation}\label{perturbed Langevin}
  \left\{
    \begin{aligned}
&        \mathrm{d}q^{x,\epsilon}_t=p^{x,\epsilon}_t \mathrm{d}t+\sqrt{2\epsilon} \mathrm{d}\widetilde{B}_t, \\
 &       \mathrm{d}p^{x,\epsilon}_t=F(q^{x,\epsilon}_t) \mathrm{d}t-\gamma p^{x,\epsilon}_t \mathrm{d}t+\sigma \mathrm{d}B_t,\\
  &      (q^{x,\epsilon}_0,p^{x,\epsilon}_0)= x .
    \end{aligned}
\right.
\end{equation}
Let $\tau^{x,\epsilon}_{\partial}=\inf \{t>0: X^{x,\epsilon}_t\notin D\}$ be the first exit time from $D$ of the process $(X_t^{x,\epsilon})_{t\geq0}$. 

We first assume that the functions $f$ and $g$ satisfy the compatibility condition~\eqref{compatibility cond}, define the function $h \in\mathcal{C}^b(\overline{D})$ by
\begin{equation}\label{eq:h-fg}
  h(x) = \mathbb{1}_{x\in D\cup\Gamma^-}f(x)+\mathbb{1}_{\Gamma^+\cup\Gamma^0}g(x),
\end{equation}
and state the following two lemmas.

\begin{lemma}[Perturbed problem]\label{v_eps equation}
Under Assumptions~\ref{hyp O} and~\ref{hyp F}, let $\epsilon>0$ and let $f\in\mathcal{C}^b(D\cup~\Gamma^-)$, $g\in\mathcal{C}^b(\Gamma^+\cup\Gamma^0)$ satisfy \eqref{compatibility cond}. Let $h \in \mathcal{C}^b(\overline{D})$ be defined by~\eqref{eq:h-fg}. The function $u_\epsilon$ on $\mathbb{R}_+^*\times D$ defined by
\begin{equation}\label{v_epsilon}
    u_\epsilon:(t,x)\mapsto\mathbb{E}\left[ \mathbb{1}_{\tau^{x,\epsilon}_{\partial}>t} h_{|D}(X^{x,\epsilon}_t) +\mathbb{1}_{\tau^{x,\epsilon}_{\partial}\leq t} h_{|\partial D}\Big(X^{x,\epsilon}_{\tau^{x,\epsilon}_{\partial}}\Big) \right]
\end{equation}
satisfies~\eqref{eq:edpueps} in the sense of distributions on $\mathbb{R}_+^* \times D$.
\end{lemma}

\begin{lemma}[Convergence]\label{v_eps convergence}
Under the assumptions of Lemma~\ref{v_eps equation}, for all $t>0$ and $x\in D$, $$u_\epsilon(t,x)\underset{\epsilon\rightarrow 0}{{\longrightarrow}} u(t,x).$$
\end{lemma}

Before proving Lemmas~\ref{v_eps equation} and~\ref{v_eps
  convergence}, let us conclude the proof of
Assertion~\eqref{it:ibvp:reg} in Theorem~\ref{Solution PDE} using
these results. Under the assumption that $f$ and $g$ satisfy the
compatibility condition~\eqref{compatibility cond}, it is immediate,
using the result of Lemma \ref{v_eps convergence}, to obtain that $u$
solves~\eqref{Langevin PDE} in the sense of distributions, by passing
to the limit $\epsilon \to 0$ in the weak formulation of the partial
differential equation and using the fact that
  $\|u_\epsilon\|_{L^\infty(\overline D)} \le
  \|h\|_{L^\infty(\overline D)}$.

If $f$ and $g$ do not satisfy the compatibility
condition~\eqref{compatibility cond}, one can use the following
approximation argument to conclude. First, we note that since $g$ is
continuous on the closed set $\Gamma^+ \cup \Gamma^0$, there exists a
function $\widetilde{g} \in \mathcal{C}^b(\overline{D})$ which coincides
with $g$ on $\Gamma^+ \cup \Gamma^0$ by Tietze-Urysohn's extension
  theorem~\cite[Theorem 4.5.1]{Dieudonne}. For any $k \geq 1$ and $x \in D \cup \Gamma^-$, let us now set
\begin{equation*}
  \widetilde{f}_k(x) = (1-\psi_k(x)) f(x) + \psi_k(x)\widetilde{g}(x),
\end{equation*}
where $\psi_k : \overline{D} \to [0,1]$ is a continuous function such that
\begin{equation*}
  \psi_k(x) = \begin{cases}
    1 & \text{if $x \in \Gamma^+ \cup \Gamma^0$,}\\
    0 & \text{if $\mathrm{d}(x, \Gamma^+ \cup \Gamma^0) \geq 1/k$.}
  \end{cases}
\end{equation*}
Then $\widetilde{f}_k$ and $g$ satisfy the compatibility condition~\eqref{compatibility cond}, so that the argument above shows that the function $\widetilde{u}_k$ defined by
\begin{equation*}
  \widetilde{u}_k(t,x) := \mathbb{E}\left[ \mathbb{1}_{\tau^x_{\partial}>t} \widetilde{f}_k(X^x_t) +\mathbb{1}_{\tau^x_{\partial}\leq t} g(X^x_{\tau^x_{\partial}}) \right]
\end{equation*}
solves~\eqref{Langevin PDE} in the distributional sense. On the other
hand, $\widetilde{f}_k(x)$ converges to $f(x)$ for all $x \in D \cup
\Gamma^-$ when $k \to +\infty$, which by the dominated convergence
theorem implies that $\widetilde{u}_k(t,x)$ converges to $u(t,x)$ and
therefore shows that $u$ is a distributional solution
to~\eqref{Langevin PDE}, also in the case when $f$ and $g$ do not satisfy the compatibility
condition~\eqref{compatibility cond}.

It finally follows from the hypoellipticity of the operator $\partial_t - \mathcal{L}$ that $u$ is actually in $\mathcal{C}^\infty(\mathbb{R}_+^* \times D)$, which completes the proof of Assertion~\eqref{it:ibvp:reg} in Theorem~\ref{Solution PDE}.

Let us now conclude this section by proving the two Lemmas~\ref{v_eps equation} and~\ref{v_eps convergence}.
\begin{proof}[Proof of Lemma~\ref{v_eps equation}] 
The result is standard for bounded domains, but $D$ is not bounded. We
thus use an approximation argument. Let $(\widetilde{V}_k)_{k\geq1}$ be a sequence of $\mathcal{C}^2$ bounded open subsets of $D$ such that:
\begin{enumerate}[label=(\roman*),ref=\roman*]
    \item for all $k\geq1$, $\widetilde{V}_k\subset D\cap\{(q,p)\in\mathbb{R}^{2d}: \vert p\vert\leq k\}$,
    \item for all $k\geq1$, $\widetilde{V}_k\subset\widetilde{V}_{k+1}$,
    \item $\bigcup_{k\geq1}\widetilde{V}_k=D$.
\end{enumerate}
For $\epsilon>0$, let $\tau^{x,\epsilon}_{\widetilde{V}_k^c}$ be the following stopping time:
\[\begin{aligned}
\tau^{x,\epsilon}_{\widetilde{V}_k^c}&=\inf \{t>0: X^{x,\epsilon}_t\notin \widetilde{V}_k\} .
\end{aligned}\]
Let $T>0$. Consider the following Initial-Boundary Value Problem,
\begin{equation}\label{perturbed Langevin PDE}
  \left\{\begin{aligned}
    \partial_t v_{k,\epsilon}(t,x) &=\mathcal{L}v_{k,\epsilon}(t,x) + \epsilon \Delta_q v_{k,\epsilon}(t,x), && t\in(0,T], \quad x \in \widetilde{V}_k,\\
    v_{k,\epsilon}(0,x) &= h_{|\widetilde{V}_k}(x), && x\in \widetilde{V}_k,\\
    v_{k,\epsilon}(t,x) &= h_{|\partial \widetilde{V}_k}(x), && t\in(0,T], \quad x \in \partial\widetilde{V}_k.
  \end{aligned}\right.
\end{equation} 

\noindent By~\cite[Chapter 6, Theorem 5.2]{F} there exists a unique
classical solution $v_{k,\epsilon}$ in
$\mathcal{C}^2((0,T]\times\widetilde{V}_k)\cap\mathcal{C}^b([0,T]\times\overline{\widetilde{V}_k})$
of \eqref{perturbed Langevin PDE}. Furthermore, the solution can be written as
follows: for all $t>0$ and $x\in
D$  $$v_{k,\epsilon}(t,x)=\mathbb{E}\left[
  \mathbb{1}_{\tau^{x,\epsilon}_{\widetilde{V}_k^c}>t} h_{|
    \widetilde{V}_k}(X^{x,\epsilon}_t)
  +\mathbb{1}_{\tau^{x,\epsilon}_{\widetilde{V}_k^c}\leq t} h_{|\partial
    \widetilde{V}_k}\Big(X^{x,\epsilon}_{\tau^{x,\epsilon}_{\widetilde{V}_k^c}}\Big)
\right].$$ Moreover when $k$ goes to infinity one has (following the
proof of Assertion~\eqref{it:ibvp:uniq} in Theorem~\ref{Solution PDE},
see Section~\ref{ss:ibvp:uniq}):
\begin{equation}\label{cv en k}
    v_{k,\epsilon}(t,x)\underset{k\longrightarrow \infty}{{\longrightarrow}}u_\epsilon(t,x).
\end{equation}
Therefore, since $v_{k,\epsilon}$ is a classical solution of
\eqref{perturbed Langevin PDE} it is also a solution in the sense of
distributions of
$\partial_tv_{k,\epsilon}=\mathcal{L}v_{k,\epsilon}+\epsilon
\Delta_qv_{k,\epsilon}$ on $(0,T)\times \widetilde{V_k}$. But then, since
$T$ is arbitrary, $u_\epsilon$ is also a solution in the sense of distributions of $\partial_t u_{\epsilon}=\mathcal{L}u_{\epsilon}$ on $\mathbb{R}_+
^*\times D$. Indeed, for $\Phi\in\mathcal{C}^\infty_c(\mathbb{R}_+^*
\times D)$,  there exists $k_0>0$ and $T_0>0$ such that $\mathrm{supp}(\Phi)\subset(0,T_0]\times\widetilde{V}_{k_0}$. As a result, for all $k>k_0$ and $T>T_0$,
$$\iint_{\mathbb{R}^{*}_+\times D}v_{k,\epsilon}(t,x)\left(\partial_t
  \Phi(t,x)+\mathcal{L}^*\Phi(t,x)+\epsilon \Delta_q \Phi(t,x)
\right)\mathrm{d}t \mathrm{d}x=0 .$$ The proof is then easily
completed, using~\eqref{cv en k} and the dominated convergence theorem.
\end{proof}
\begin{proof}[Proof of Lemma~\ref{v_eps convergence}] 
 An application of Gronwall's Lemma, as in the proof of Lemma~\ref{couplage Lemma}, shows that, almost surely, 
 \begin{equation}\label{difference processes}
\sup_{s\in[0,t]}\vert X_s^{x,\epsilon}- X^x_s\vert\leq \sqrt{2\epsilon} \sup_{s\in[0,t]}\vert \widetilde{B}_{s} \vert   \mathrm{e}^{C_\mathrm{Lip} t} 
\end{equation}
where $C_\mathrm{Lip}$ is the Lipschitz constant of the drift of~\eqref{Langevin}. In particular, for all $t\geq0$,   $X_t^{x,\epsilon}\underset{\epsilon\longrightarrow 0}{{\longrightarrow}}X^x_t$ almost surely.

Let us now consider the difference between $u_\epsilon(t,x)$ and
$u(t,x)$ for $t>0$, $x\in D$. Using the same triangle inequality as in
the proof of Assertion~\eqref{it:ibvp:cont} of Theorem~\ref{Solution
  PDE} (see Section~\ref{ss:ibvp:cont}), one has 
\[\begin{aligned}
\left\vert  u_\epsilon(t,x)-u(t,x)\right\vert&\leq \mathbb{E}\left[\mathbb{1}_{\tau^{x,\epsilon}_{\partial}>t,\tau^{x}_{\partial}>t}\left\vert h_{| D}(X^{x,\epsilon}_{t})-h_{| D}(X^{x}_{t}) \right\vert \right]+2\Vert h\Vert_{\infty}\mathbb{E}\left[ \left\vert \mathbb{1}_{\tau^{x,\epsilon}_{\partial}>t}-\mathbb{1}_{\tau^{x}_{\partial}>t}\right\vert \right]\\
&\quad +\mathbb{E}\left[\mathbb{1}_{\tau^{x,\epsilon}_{\partial}\leq t, \tau^x_{\partial}\leq t}\left\vert h_{| \partial D}(X^{x,\epsilon}_{\tau^{x,\epsilon}_{\partial}})-h_{| \partial D}(X^{x}_{\tau^{x}_{\partial}})\right\vert \right] .
\end{aligned}\] 
Using \eqref{difference processes} and the fact that $h_{|
  D}\in\mathcal{C}^b(D)$, it follows from the dominated convergence
theorem that the first term in the right-hand side of the inequality
converges to $0$ as $\epsilon$ goes to $0$. Besides, remember that
$\mathbb{P}(\tau^x_{\partial}=t)=0$ for $x\in D$ and
$t>0$ by Corollary~\ref{rq densité}. As a result if one can prove that for all $x\in D$, $t>0$,
\begin{equation}\label{cv indicatrices 3}
\mathbb{1}_{\tau^{x,\epsilon}_{\partial}>t}\underset{\epsilon\rightarrow 0}{{\longrightarrow}}   \mathbb{1}_{\tau^x_{\partial}>t}\quad\text{almost surely on the events $\{\tau^x_\partial < t\text\}$ and $\{\tau^x_\partial > t\}$,}
\end{equation}
and
\begin{equation}\label{cv indicatrices 3 bis}
\tau^{x,\epsilon}_{\partial}\underset{\epsilon\longrightarrow
  0}{{\longrightarrow}}\tau^{x}_{\partial}\quad\text{almost surely on
  the event }\{\tau^{x}_{\partial} < t\},
\end{equation}
then using \eqref{difference processes}, the fact that $h_{|\partial
  D}\in\mathcal{C}^b(\Gamma^+\cup\Gamma^0)$ and the continuity of the
trajectories of $(X_t^x)_{t\geq0}$, the convergence of $
u_\epsilon(t,x)$ towards $u(t,x)$ follows from
the dominated convergence theorem, and the proof is complete.

 Let us now prove the two convergences \eqref{cv indicatrices 3} and \eqref{cv indicatrices 3 bis}.

\medskip \noindent \textbf{Step 1}. Consider first the convergence
\eqref{cv indicatrices 3} on the event $\{\tau^x_{\partial}>t\}$.  By the continuity of the trajectories of $(q^x_s)_{s\geq0}$,
$$ \epsilon_0:=\inf_{0\leq s\leq t}\mathrm{d}_\partial(q^x_s)>0 .$$

\noindent Let $S_t:=\sup_{0\leq s\leq t}\vert \widetilde{B}_{s}
\vert$. For $\epsilon\leq \frac{\epsilon^2_0}{8S^2_t}
\mathrm{e}^{-2C_\mathrm{Lip}t}$ $(\text{which is positive since
}S_t<\infty\text{ almost surely})$, one has by \eqref{difference processes}:
$$\sup_{0\leq s\leq t}\vert q^{x,\epsilon}_s- q^x_s\vert \leq \sup_{0\leq s\leq t}\vert X_s^{x,\epsilon}- X^x_s\vert
\leq \frac{\epsilon_0}{2} .$$

\noindent Hence, since $\mathrm{d}_\partial$ is $1$-Lipschitz continuous, for $\epsilon\leq \frac{\epsilon^2_0}{8S^2_t} \mathrm{e}^{-2C_\mathrm{Lip}t}$ $$\inf_{0\leq s\leq t}\mathrm{d}_\partial(q^{x,\epsilon}_s)\geq \frac{\epsilon_0}{2}>0,$$
which implies
$ \mathbb{1}_{\tau^{x,\epsilon}_{\partial}>t}=1$ and~\eqref{cv
  indicatrices 3} thus holds on the event $\{\tau^x_{\partial}>t\}$.

\medskip \noindent \textbf{Step 2}. Let us now prove the convergences
\eqref{cv indicatrices 3} and \eqref{cv indicatrices 3 bis} on the
event $\{\tau^x_{\partial}< t\}$. Since $x\in D$, by Proposition \ref{prop:tau} one has $(q^x_{\tau^x_\partial},p^x_{\tau^x_\partial})\in\Gamma^+$ almost surely. Let $0<\eta<(t-\tau^x_\partial)\land\tau^x_\partial$. The strong Markov property along with Proposition \ref{prop:tau} ensure that there exists almost surely $t_0\in(\tau^x_\partial,\tau^x_\partial+\eta)$ such that 

$$\epsilon_1:=\mathrm{d}_{\overline{\mathcal{O}}}(q^x_{t_0}) >0 .$$
Besides, the continuity of the trajectories of $(q_s^x)_{s\geq0}$ ensures that 
$$\epsilon_2:=\inf_{0\leq s\leq \tau^x_\partial-\eta}\mathrm{d}_\partial(q^{x}_s)>0 . $$

\noindent As a result, for $\epsilon\leq \frac{\epsilon^2_1\land\epsilon^2_2}{8S^2_t} \mathrm{e}^{-2C_\mathrm{Lip}t}$,
$$\sup_{0\leq s\leq t}\vert q^{x,\epsilon}_s- q^x_s\vert \leq \sup_{0\leq s\leq t}\vert X_s^{x,\epsilon}- X^x_s\vert
\leq \frac{\epsilon_1\land\epsilon_2}{2}.$$
Hence, since $\mathrm{d}_{\overline{\mathcal{O}}}$ is 1-Lipschitz continuous, $$\mathrm{d}_{\overline{\mathcal{O}}}(q^{x,\epsilon}_{t_0})\geq\frac{\epsilon_1}{2}>0 ,$$
and since $\mathrm{d}_\partial$ is 1-Lipschitz continuous as well, one has
$$\inf_{0\leq s\leq \tau^x_\partial-\eta}\mathrm{d}_\partial(q^{x,\epsilon}_s)\geq \frac{\epsilon_2}{2}>0 . $$
Therefore, for $\epsilon$ small enough,
$$ \vert \tau^{x,\epsilon}_\partial-\tau^{x}_\partial \vert\leq
\eta\quad\text{and in particular}\quad\tau^{x,\epsilon}_\partial\leq\tau^{x}_\partial+\eta<t .$$ Consequently, the convergences \eqref{cv indicatrices 3} and \eqref{cv indicatrices 3 bis} hold on the event $\{\tau^x_{\partial}< t\}$.
\end{proof}
 
\subsection{Proof of Propositions \ref{non attainability} and \ref{prop:tau}}\label{section proof lemmata}
We conclude Section~\ref{section 2} with the proofs of
Propositions~\ref{non attainability} and~\ref{prop:tau}, which are the cornerstones of all the
previous results. In Section~\ref{sss:sphere}, we deduce from a simple geometric argument that Assertion~\eqref{it:tau:reg} in Proposition~\ref{prop:tau} holds for $x \in \Gamma^+$, and that, taking Proposition~\ref{non attainability} for granted, Assertion~\eqref{it:tau:sing} in Proposition~\ref{prop:tau} holds. The proof of the remaining statements, namely Proposition~\ref{non attainability} and Assertion~\eqref{it:tau:reg} in Proposition~\ref{prop:tau} for $x \in \Gamma^0$, both rely on a preliminary reduction to a Gaussian process, thanks to the Girsanov theorem, which is detailed in Section~\ref{sss:girsanov}. Last, we complete the proof of Assertion~\eqref{it:tau:reg} in Proposition~\ref{prop:tau} in Section~\ref{subsec:proof assertion i}
and we provide the proof of Proposition~\ref{non attainability} in Section~\ref{subsec:proof proposition}. 

\subsubsection{The interior and exterior sphere conditions}\label{sss:sphere}

The first part of the proof of Proposition~\ref{prop:tau} relies on the following geometric property of the set $\mathcal{O}$ which is standard.

\begin{proposition}[Uniform interior and exterior sphere conditions]\label{prop:sphere}
  Under Assumption~\ref{hyp O}, there exists $\rho > 0$ such that for any $q \in \partial \mathcal{O}$, there exist two points $q_\mathrm{int} \in \mathcal{O}$ and $q_\mathrm{ext} \in \overline{\mathcal{O}}^c$ such that the open Euclidean balls $\mathrm{B}(q_\mathrm{int},\rho)$ and $\mathrm{B}(q_\mathrm{ext},\rho)$ satisfy
\begin{equation*}
  \mathrm{B}(q_\mathrm{int},\rho) \subset \mathcal{O}, \qquad \mathrm{B}(q_\mathrm{ext},\rho) \subset \overline{\mathcal{O}}^c, \qquad  \overline{\mathrm{B}(q_\mathrm{int},\rho)} \cap \mathcal{O}^c = \overline{\mathrm{B}(q_\mathrm{ext},\rho)} \cap \overline{\mathcal{O}} = \{q\}.
\end{equation*}
\end{proposition}

Let us now detail the application of Proposition~\ref{prop:sphere} to the proof of Proposition~\ref{prop:tau}. 

For $x=(q,p)\in\Gamma^+$, let $q_\mathrm{ext} \in
\overline{\mathcal{O}}^c$ be given by the exterior sphere
condition. Necessarily, the vectors $q_\mathrm{ext}-q$ and
$n(q)$ are colinear. On the other hand, for $t \to 0$, $(q^x_t - q) \cdot n(q) \sim t p \cdot n(q) > 0$, which then implies that $|q^x_t-q_\mathrm{ext}|^2 = \rho^2 - 2 \rho t p \cdot n(q) + o(t)$ so that $q^x_t \in \mathrm{B}(q_\mathrm{ext},\rho) \subset \overline{\mathcal{O}}^c$ for $t$ small enough, and therefore~\eqref{sortie
  immédiate} holds.

With similar arguments, the interior sphere condition shows that if $x
\in \Gamma^-$, then $\tau_\partial^x > 0$ almost surely. Moreover, it
is obvious that if $x
\in D$, then $\tau_\partial^x > 0$ almost surely. Finally, if $x \in
D \cup \Gamma^-$, then on the event $\tau^x_\partial \leq T$ one necessarily has
$X^x_{\tau^x_\partial} \in \Gamma^+ \cup \Gamma^0$ almost surely,
which rewrites:
\begin{equation}\label{eq:p.n<0}
 \forall T >0,\quad \forall x \in D \cup \Gamma^-,\qquad  \mathbb{P}\left( p^x_{\tau^x_\partial}\cdot n(q^x_{\tau^x_\partial}) < 0 ,
  \tau^x_\partial \le T \right)=0 .
\end{equation}
Therefore, taking Proposition~\ref{non attainability} for granted, we obtain Assertion \eqref{it:tau:sing} in Proposition~\ref{prop:tau}.

\subsubsection{Reduction to a Gaussian process}\label{sss:girsanov}

Proposition~\ref{non attainability} and Assertion \eqref{it:tau:reg} in Proposition~\ref{prop:tau} rely on the following preliminary result.

\begin{lemma}[Girsanov Theorem]\label{Girsanov}
Let Assumption~\ref{hyp F} hold. Let $x\in \mathbb{R}^{2d}$ and let $(\Check{q}^x_t,\Check{p}^x_t)_{t\geq0}$ be the strong solution on $\mathbb{R}^{2d}$ of
\begin{equation}\label{Langevin2}
  \left\{
    \begin{aligned}
&        \mathrm{d}\Check{q}^x_t=\Check{p}^x_t\mathrm{d}t , \\
 &       \mathrm{d}\Check{p}^x_t= \sigma \mathrm{d}B_t ,\\
  &      (\Check{q}_0^x,\Check{p}_0^x)=x.
    \end{aligned}
\right.
\end{equation} 
For $T\geq0$, the laws of $(\Check{q}^x_t,\Check{p}^x_t)_{t\in[0,T]}$
and $(q^x_t,p^x_t)_{t\in[0,T]}$ are equivalent in the space of sample
paths $\mathcal{C}([0,T],\mathbb{R}^{2d})$, i.e. for all Borel sets
$A$ of $\mathcal{C}([0,T],\mathbb{R}^{2d})$,
$$\mathbb{P}((\Check{q}^x_t,\Check{p}^x_t)_{t\in[0,T]}\in A)=0\quad\text{if and only if}\quad\mathbb{P}((q^x_t,p^x_t)_{t\in[0,T]}\in A)=0 . $$
\end{lemma}
\begin{proof}
Let $x=(q,p)\in\mathbb{R}^{2d}$. Equation~\eqref{Langevin2} admits
a unique global in time strong solution
$(\Check{q}^x_t,\Check{p}^x_t)_{t\geq0}$ on~$\mathbb{R}^{2d}$ since
its coefficients are globally Lipschitz continuous. For $T\geq0$, let us define, 
$$ \mathcal{Z}^x_T=F(\Check{q}^x_T)-\gamma \Check{p}^x_T ,$$ and 
$$ \mathcal{E}^x_T=\exp\left( \int_0^T \mathcal{Z}^x_s\cdot\mathrm{d}B_s -\frac{1}{2}\int_0^T\vert\mathcal{Z}^x_s\vert^2 \mathrm{d}s   \right) .$$
It is clear that $\mathcal{E}^x_T$ is $\mathcal{F}_T$-measurable. Let us show that for all $T\geq0$, $\mathbb{E}[\mathcal{E}^x_T] =1 $.
According to \cite[Theorem 1.1 p. 152]{F}, this equality  is satisfied
if there exists $\mu>0$ such
that  $$\sup_{s\in[0,T]}\mathbb{E}[\mathrm{exp}(\mu
\vert\mathcal{Z}^x_s\vert^2)]<\infty,$$ which we now prove. Since $F$ satisfies Assumption \ref{hyp F}, it follows that for $s\in[0,T]$,
$$\vert\mathcal{Z}^x_s\vert^2=\vert F(\Check{q}^x_s) -\gamma \Check{p}^x_s\vert^2\leq 2\Vert F\Vert_\infty^2 +2\gamma^2\vert\Check{p}^x_s\vert^2. $$
In addition, $\Check{p}^x_s\sim \mathcal{N}_{d} (p,\sigma^2 s I_d) $. Let $G\sim\mathcal{N}_{d} (0, I_d)$, we get for $s\in[0,T]$,
\begin{align*}
    \mathbb{E}[\mathrm{exp}(\mu \vert\mathcal{Z}^x_s\vert^2)]&\leq\mathrm{exp}(2\mu\Vert F\Vert_\infty^2) \mathbb{E}(\mathrm{exp}(2\mu\gamma^2\vert\Check{p}^x_s\vert^2)) \\
    &=\mathrm{exp}(2\mu\Vert F\Vert_\infty^2) \mathbb{E}(\mathrm{exp}(2\mu\gamma^2\vert p+\sigma\sqrt{s} G\vert^2))\\
    &\leq \mathrm{exp}(2\mu\Vert F\Vert_\infty^2+4\mu\gamma^2\vert p\vert^2) \mathbb{E}(\mathrm{exp}(4\mu\gamma^2\sigma^2T\vert G\vert^2)) .
\end{align*}
 Moreover, $\mathbb{E}(\mathrm{exp}(4\mu\gamma^2\sigma^2T\vert
   G\vert^2)< \infty$ for sufficiently small $\mu$.  

This result allows us to define the probability measure $\mathbb{Q}_T$
on $\mathcal{F}_T$ by $\mathrm{d}\mathbb{Q}_T=\mathcal{E}^x_T
\mathrm{d}\mathbb{P}\vert_{\mathcal{F}_T}$. Since $\mathcal{E}^x_T>0$, $\mathbb{P}\vert_{\mathcal{F}_T}$-a.s., the measures $\mathbb{P}\vert_{\mathcal{F}_T}$ and $\mathbb{Q}_T$ are equivalent. Besides, by the Girsanov Theorem \cite[Theorem 1.1 p. 152]{F}
the process $$ \left(\Check{B}_s:=B_s-\int_0^s\mathcal{Z}^x_r
  \mathrm{d}r\right)_{0\leq s\leq T}$$ is a
$(\mathcal{F}_s)_{s\in[0,T]}$-Brownian motion under the probability
$\mathbb{Q}_T$. As a result, the process
$(\Check{X}^x_s,\Check{B}_s)_{s\in[0,T]}$ satisfies \eqref{Langevin}
on the probability space $(\Omega,\mathcal{F}
,(\mathcal{F}_s)_{s\in[0,T]},\mathbb{Q}_T)$. On the other hand, the
pathwise uniqueness for~\eqref{Langevin} implies the uniqueness in
distribution by Yamada Watanabe's theorem, so that the law of
$(\Check{q}^x_s,\Check{p}^x_s)_{s\in[0,T]}$ under $\mathbb{Q}_T$ is
the law of $(q^x_s,p^x_s)_{s\in[0,T]}$ under $\mathbb{P}$, whence the final result.
\end{proof}

We are now in position to complete the proof of Proposition~\ref{prop:tau} and to detail the proof of Proposition~\ref{non attainability}. By Lemma \ref{Girsanov} it is
sufficient to prove both statements for the process $(\Check
{X}^x_t)_{t\geq0}$ defined in~\eqref{Langevin2}, and for which we
introduce the notation $\Check{\tau}^x_\partial := \inf\{t>0:
\Check{X}^x_t \not\in D\}$.  

\subsubsection{Proof of Assertion~\eqref{it:tau:reg} in Proposition~\ref{prop:tau}}\label{subsec:proof assertion i}

\begin{proof}[Proof of Assertion~\eqref{it:tau:reg} in Proposition~\ref{prop:tau}]
The case $x \in \Gamma^+$ has been addressed in Section~\ref{sss:sphere}. It remains
now to prove \eqref{sortie immédiate} for $x=(q,p)\in\Gamma^0$. One has from \eqref{Langevin2} that for all $t\geq0$,
$$\Check{p}^x_t=p+\sigma B_t \text{ and thus }
\Check{q}^x_t=q+pt + \sigma \int_0^t B_s \,  \mathrm{d}s.$$
The idea of the proof is to reduce the problem to the case of a flat boundary by
a change of variable, and then to use known results for $1$-d
integrated Brownian motion, see~\cite{Lach}.

Let
$(e_1,\ldots, e_d)$ be the canonical basis of $\mathbb{R}^d$. Since
$\mathcal{O}$ is a bounded $\mathcal{C}^2$ set of $\mathbb{R}^d$,
by~\cite[Theorem~2.1.2]{BG} there exists an open neighborhood $U$ of $q$ and a $\mathcal{C}^2$-diffeomorphism $\phi:(-1,1)^d\rightarrow U$ satisfying $\phi(0)=q$ and
$$ \mathcal{O}\cap U=\phi(\{y\in(-1,1)^d: y\cdot e_d<0\})\quad\text{and}\quad\partial\mathcal{O}\cap U=\phi(\{y\in(-1,1)^d: y\cdot e_d=0\}) .$$
Moreover, $n(q)\in\mathbb{R}^d$ is the unique vector such that
\begin{equation}\label{proprietes vecteur normal}
    n(q)\in\mathrm{Span}(d_0\phi(e_1),\ldots,d_0\phi(e_{d-1}))^\perp,  \vert n(q)\vert=1,  \text{ and } d_0\phi(e_d)\cdot n(q)>0,
\end{equation}
where $d_0\phi$ is the differential at $0\in\mathbb{R}^d$ of $\phi$.

Now let $K$ be a compact set included in $U$ such that $q\in\mathring{K}$. Let $\Check{\tau}^x_{K^c}:=\inf\{t>0:\Check{q}^x_t\notin K\}$ be the first exit time of $K$ for $(\Check{q}^x_t)_{t\geq0}$, then $\Check{\tau}^x_{K^c}>0$ almost surely by continuity of the trajectories of~$(\Check{q}^x_t)_{t\geq0}$.

\noindent For $t\leq\Check{\tau}^x_{K^c}$ we have 
\begin{equation}\label{eq:phi1}
\phi^{-1}(\Check{q}^x_t)= \underbrace{\phi^{-1}(q)}_{=0}+\int_0^t d_{\Check {q}^x_s}(\phi^{-1})\left(p +\sigma B_s\right) \mathrm{d}s.
\end{equation}
Since $\phi^{-1}$ is a $\mathcal{C}^2$-diffeomorphism from $U$ to $(-1,1)^d$, then $y\in K\subset U\mapsto d_y(\phi^{-1})$ is $\mathcal{C}^1$ on the compact set $K$. In particular it is Lipschitz continuous with some Lipschitz constant $k$. As a result, since for $t\geq0$, $\Check{q}^x_t=q+\int_0^t\Check{p}^x_s \mathrm{d}s$, then for all $t\in[0,\Check{\tau}^x_{K^c}]$ and $z\in\mathbb{R}^d$,
\begin{align}
\left\vert d_{\Check {q}^x_t}(\phi^{-1})(z)-d_q(\phi^{-1})(z)
  \right\vert
  &\leq k \vert  \Check{q}^x_t-q \vert \vert z\vert =k \left\vert  t
    p+\sigma \int_0^t B_s  \mathrm{d}s \right\vert \vert z\vert  \nonumber\\
&\leq k t\left(  \vert p\vert+\sigma \sup_{s\in[0,t]}\vert B_s\vert  \right) \vert z\vert.\label{eq:phi2}
\end{align}
Hence we have  from~\eqref{eq:phi1} and~\eqref{eq:phi2}
$$\left\vert\phi^{-1}(\Check{q}^x_t)-\int_0^t d_{q}(\phi^{-1})(p+\sigma B_s) \mathrm{d}s\right\vert\leq kt^2\left(\vert  p\vert+ \sigma \sup_{s\in[0,t]}\vert B_s\vert  \right)^2.$$
Therefore 
\begin{equation}\label{prod scalaire 3} 
    \left\vert\phi^{-1}(\Check {q}^x_t)\cdot e_d-  td_q(\phi^{-1})(p)\cdot e_d- \sigma d_q(\phi^{-1})\left(\int_0^t B_s \mathrm{d}s\right)\cdot e_d\right\vert\leq k t^2 \left(\vert  p\vert+ \sigma \sup_{s\in[0,t]}\vert B_s\vert  \right)^2.
\end{equation}

Let us now prove that, since $x \in \Gamma^0$,
\begin{equation}\label{prod scalaire}
    d_q(\phi^{-1})(p)\cdot e_d=0.
\end{equation}
Since $\phi$ is a $\mathcal{C}^1$-diffeomorphism from $(-1,1)^d$ to
$U$ with $U$ a neighborhood of $q$ and $\phi(0)=q$, then $d_0(\phi)$
is invertible with inverse
satisfying $$(d_0(\phi))^{-1}=d_q(\phi^{-1})$$ In particular, the
family $(d_0(\phi)(e_1),\ldots,d_0(\phi)(e_d))$ is a basis of
$\mathbb{R}^d$. Let us now decompose the vector~$p$ in this basis: $$p=\sum_{j=1}^d p_j  d_0(\phi)(e_j) .$$
Using \eqref{proprietes vecteur normal} and the fact that
$p\cdot n(q)=0$ since $x\in\Gamma^0$, we get
$p_d=0$. As a result,
\[\begin{aligned}
d_q(\phi^{-1})(p)\cdot e_d&=d_q(\phi^{-1})\left(\sum_{j=1}^d p_j  d_0(\phi)(e_j)\right)\cdot e_d\\
&=\sum_{j=1}^d p_j(d_0(\phi))^{-1}d_0(\phi)(e_j)\cdot e_d=p_d=0 .
\end{aligned}\]
This concludes the proof of \eqref{prod scalaire}.

Now notice that
\begin{equation}\label{prod scalaire 2}
    d_q(\phi^{-1})\left(\int_0^t B_s \mathrm{d}s\right)\cdot e_d=\int_0^tB_s\cdot d_0(\phi)^{-T}(e_d) \mathrm{d}s,
\end{equation} where $d_0(\phi)^{-T}$ is the transpose matrix of $d_0(\phi)^{-1}$.  Moreover, $\vert d_0(\phi)^{-T}(e_d)\vert >0$, since $d_0(\phi)^{-T}$ is also invertible. Using \eqref{prod scalaire} and \eqref{prod scalaire 2} in \eqref{prod scalaire 3}, one gets
\begin{equation}\label{eq11}
    \left\vert\phi^{-1}(\Check {q}^x_t)\cdot e_d- \sigma \int_0^tB_s\cdot d_0(\phi)^{-T}(e_d) \mathrm{d}s\right\vert\leq k t^2 \left(\vert  p\vert+ \sigma \sup_{s\in[0,t]}\vert B_s\vert  \right)^2.
\end{equation}
Let us define the process $(\widehat{B}_s)_{s \in [0,t]}$ by $$\forall s\in[0,t],\qquad\widehat {B}_s:=B_s\cdot \frac{d_0(\phi)^{-T}(e_d)}{\vert d_0(\phi)^{-T}(e_d)\vert}.$$It is clearly a one-dimensional Brownian motion on $[0,t]$. Then \eqref{eq11} rewrites 
$$\left\vert\phi^{-1}(\Check {q}^x_t)\cdot e_d- \sigma \vert d_0(\phi)^{-T}(e_d)\vert  \int_0^t\widehat {B}_s \mathrm{d}s\right\vert\leq k t^2 \left(\vert  p\vert+ \sigma \sup_{s\in[0,t]}\vert B_s\vert  \right)^2.$$
 
The law of the iterated logarithm for the integrated Brownian motion
(see~\cite[Theorem 1]{Lach}) provides us with the following asymptotic
result: $$\limsup_{t\rightarrow0}\frac{\int_0^t\widehat {B}_s \mathrm{d}s}{\sqrt{\frac{2}{3}} t^{\frac{3}{2}}\sqrt{\log \log(1/t)}}=1\quad\text{almost surely.}$$
For $t>0$, let $\Psi(t)=\sqrt{\frac{2}{3}} t^{\frac{3}{2}}\sqrt{\log \log(1/t)}$, then 
\[\begin{aligned}
\left\vert\frac{\phi^{-1}(\Check {q}^x_t)\cdot e_d}{\Psi(t)}
-\sigma \vert d_0(\phi)^{-T}(e_d)\vert  \frac{\int_0^t\widehat {B}_s \mathrm{d}s}{\Psi(t)} \right\vert\leq\underbrace{\frac{k t^2}{\Psi(t)}}_{\underset{t\rightarrow0}{\longrightarrow}0}\left(\vert  p\vert+ \sigma \sup_{s\in[0,t]}\vert B_s\vert  \right)^2  . \\
\end{aligned}\]
Therefore, almost surely, $$\limsup_{t\rightarrow0}\frac{\phi^{-1}(\Check {q}^x_t)\cdot e_d}{\Psi(t)}= \sigma \vert d_0(\phi)^{-T}(e_d)\vert>0  .$$
As a result, the process $(\Check {q}^x_t)_{t\geq0}$ visits $U\cap\overline{\mathcal{O}}^c$ infinitely often for times close to $0$. This implies in particular that $\Check{\tau}^x_\partial=0$ almost surely.  
\end{proof}

\subsubsection{Proof of Proposition~\ref{non attainability}}\label{subsec:proof proposition}
 
We now address the proof of Proposition~\ref{non attainability}. For $x\in\mathbb{R}^{2d}$, let 
$\Check{\tau}^x_{0}:=\inf\{t>0:(\Check{q}^x_t,\Check{p}^x_t)\in\Gamma^0\}$ and let us show here that for all $x\in\mathbb{R}^{2d}\setminus \Gamma^0$,
$$\mathbb{P}(\Check{\tau}^x_{0}<\infty)=0,$$
which is equivalent to
\begin{equation}\label{eqn2}
    \forall T>0,\qquad\mathbb{P}(\Check{\tau}^x_{0}\leq T)=0.
\end{equation}  
The idea of the proof is the following. If one replaces the random time
$\Check{\tau}^x_0$ by a deterministic time $t \leq T$, and
denote by $\Check{n}$ some continuous extension of the normal vector
$n$ in a neighborhood of $\partial\mathcal{O}$, then using the fact
that $\Check{p}^x_t$ has a nondegenerate Gaussian conditional
distribution given $\Check{q}^x_t$ allows us to write $$\mathbb{P}\left(\Check{p}^x_t\cdot \Check{n}(\Check{q}^x_t)= 0\right) = \mathbb{E}\left[\mathbb{P}\left(\Check{p}^x_t\cdot\Check{n}(\Check{q}^x_t)= 0|\Check{q}^x_t\right)\right]=0.$$ Our proof therefore relies on the approximation of $\Check{\tau}^x_0$ by a grid of deterministic times and exploits the fact that while assuming that such a time $t$ is close to $\Check{\tau}^x_0$ makes the distribution of $\Check{q}^x_t$ quite singular, it leaves `enough randomness' in the distribution of $\Check{p}^x_t$ for quantities of the form $\mathbb{P}(\Check{p}^x_{\Check{\tau}^x_0}\cdot n(\Check{q}^x_{\Check{\tau}^x_0})= 0, \Check{\tau}^x_0 \simeq t)$ to be sufficiently small.
\begin{proof}[Proof of Proposition~\ref{non attainability}] 
Let $x=(q,p) \in \mathbb{R}^{2d}\setminus \Gamma^0$. As explained above, the objective is to prove~\eqref{eqn2}.

Let $\alpha\in(0,1/2)$. Since $(\Check{p}^x_t)_{0\leq t\leq T}$ is a Brownian motion, one has that
$$\sup_{0\leq t\leq T}\vert \Check {p}^x_t\vert<\infty,\quad
\sup_{0\leq s<t\leq T}\frac{\vert \Check {p}^x_t-\Check
  {p}^x_s\vert}{\vert t-s\vert^{\alpha}}<\infty\quad\text{almost
  surely.}$$ Let $\epsilon>0$ and let us choose $M$ large enough so that 
\begin{equation}\label{ineq2}   
\mathbb{P}\left(\sup_{0\leq t\leq T}\vert \Check {p}^x_t\vert>M \right)\leq\epsilon,\quad \mathbb{P}\left(\sup_{0\leq s<t\leq T}\frac{\vert \Check {p}^x_t-\Check {p}^x_s\vert}{\vert t-s\vert^{\alpha}}>M \right)\leq\epsilon.
\end{equation} 
Therefore,
\[\begin{aligned}
&\mathbb{P}\left(\Check{\tau}^x_0\leq T \right)\\
&=\mathbb{P}\left(\Check {p}^x_{\Check{\tau}^x_0}\cdot n(\Check{q}^x_{\Check{\tau}^x_0})=0 , \Check{\tau}^x_0\leq T \right)\\&\leq\mathbb{P}\left(\Check {p}^x_{\Check{\tau}^x_0}\cdot n(\Check{q}^x_{\Check{\tau}^x_0})=0 , \Check{\tau}^x_0\leq T ,\sup_{0\leq t\leq T}\vert \Check {p}^x_t\vert\leq M ,\sup_{0\leq s<t\leq T}\frac{\vert \Check{p}^x_t-\Check{p}^x_s\vert}{\vert t-s\vert^{\alpha}}\leq M \right)+2 \epsilon.
\end{aligned}\]
Let us now consider the first term in the right-hand side of the inequality above .

\medskip \noindent \textbf{Step 1}.  Let $N\in\mathbb{N}^*$. We divide the interval $(0,T]$ into $N$ intervals $(t_k,t_{k+1}]$ with $t_k:=k \eta_N$ and $\eta_N:=\frac{T}{N}$. As a result, since $\Check{\tau}^x_0>0$ almost surely, because $x$ belongs to the open set $\mathbb{R}^{2d}\setminus \Gamma^0$, 
\[\begin{aligned}
&\mathbb{P}\left(\Check{p}^x_{\Check{\tau}^x_0}\cdot n(\Check{q}^x_{\Check{\tau}^x_0})=0 , \Check{\tau}^x_0\leq T ,\sup_{0\leq t\leq T}\vert \Check{p}^x_t\vert\leq M ,\sup_{0\leq s<t\leq T}\frac{\vert \Check{p}^x_t-\Check{p}^x_s\vert}{\vert t-s\vert^{\alpha}}\leq M \right)\\&=\sum_{k=0}^{N-1}\mathbb{P}\left(\Check{p}^x_{\Check{\tau}^x_0}\cdot n(\Check {q}^x_{\Check{\tau}^x_0})=0 , \Check{\tau}^x_0\in(t_k,t_{k+1}] ,\sup_{0\leq t\leq T}\vert \Check{p}^x_t\vert\leq M ,\sup_{0\leq s<t\leq T}\frac{\vert \Check{p}^x_t-\Check {p}^x_s\vert}{\vert t-s\vert^{\alpha}}\leq M \right) .
\end{aligned}\]
    Let us denote by $\overline{\mathrm{d}}_\partial$ the signed Euclidean distance to the boundary $\partial\mathcal{O}$, i.e.
\begin{equation*}
  \overline{\mathrm{d}}_\partial : q \in \mathbb{R}^d \mapsto \begin{cases}
    \mathrm{d}(q,\partial\mathcal{O}) & \text{if $q \in \mathcal{O}$,}\\
    -\mathrm{d}(q,\partial\mathcal{O}) & \text{if $q \not\in \mathcal{O}$,}
  \end{cases}
\end{equation*}
so that $\mathrm{d}_\partial$ is the positive part of $\overline{\mathrm{d}}_\partial$, and the function $\overline{\mathrm{d}}_\partial$ is $1$-Lipschitz continuous.

On the event $$\mathcal{A}_{k,M}:=\left\{\Check{p}^x_{\Check{\tau}^x_0}\cdot n(\Check {q}^x_{\Check{\tau}^x_0})=0 , \Check{\tau}^x_0\in(t_k,t_{k+1}] , \sup_{0\leq t\leq T}\vert \Check {p}^x_t\vert\leq M , \sup_{0\leq s<t\leq T}\frac{\vert \Check {p}^x_t-\Check {p}^x_s\vert}{\vert t-s\vert^{\alpha}}\leq M\right\} ,$$ we have 
$$\left\vert  \Check {q}^x_{\Check{\tau}^x_0}-\Check {q}^x_{t_k} \right\vert = \left\vert  \int_{t_k}^{\Check{\tau}^x_0} \Check {p}^x_u \mathrm{d}u \right\vert\leq M(\Check{\tau}^x_0-t_k)\leq M\eta_N .$$
Thus,
$$ \vert\overline{\mathrm{d}}_\partial(\Check {q}^x_{t_k})\vert\leq \left\vert  \Check {q}^x_{t_k}-\Check {q}^x_{\Check \tau^x_0} \right\vert\leq  M\eta_N .$$
For $\mu>0$, let 
\begin{equation}\label{C1 distance}
\overline{\mathcal{O}}_\mu:=\{q\in\mathbb{R}^d:\vert\overline{\mathrm{d}}_\partial(q)\vert\leq \mu\}.
\end{equation}
Since the bounded open set $\mathcal{O}$ is $\mathcal{C}^2$ there exists a constant $\mu>0$ such that the signed distance $\overline{\mathrm{d}}_\partial(q)$ to $\partial\mathcal{O}$ is $\mathcal{C}^2$ on the set $\overline{\mathcal{O}}_\mu$ according to~\cite[Lemma 14.16]{GT}. Moreover $\overline{\mathrm{d}}_\partial(q)$ satisfies the following eikonal equation
\begin{equation}
\left\{
\begin{aligned}\label{eikonal equation}
       \vert\nabla \overline{\mathrm{d}}_\partial(q)\vert &= 1\quad&&\text{for }q\in\overline{\mathcal{O}}_\mu,\\
        \nabla \overline{\mathrm{d}}_\partial(q)&=- n(q)\quad&&\text{for }q\in\partial\mathcal{O}.
\end{aligned}
\right. 
\end{equation}
Let us now choose $N$ large enough so that $M\eta_N=M\frac{T}{N}\leq \mu$. As a result, since $\Check {q}^x_{t_k}\in\overline{\mathcal{O}}_\mu$ on $\mathcal{A}_{k,M}$, 
\[\begin{aligned}
\left\vert\Check {p}^x_{t_k}\cdot\nabla \overline{\mathrm{d}}_\partial(\Check {q}^x_{t_k})\right\vert &\leq \left\vert\left(\Check {p}^x_{t_k}-\Check {p}^x_{\Check{\tau}^x_0}\right)\cdot\nabla \overline{\mathrm{d}}_\partial(\Check {q}^x_{t_k})\right\vert+\left\vert\Check {p}^x_{\Check{\tau}^x_0}\cdot(\nabla \overline{\mathrm{d}}_\partial(\Check {q}^x_{t_k})-\nabla \overline{\mathrm{d}}_\partial(q^x_{\Check{\tau}^x_0}))\right\vert + \underbrace{\left\vert\Check {p}^x_{\Check{\tau}^x_0}\cdot\nabla \overline{\mathrm{d}}_\partial(\Check {q}^x_{\Check{\tau}^x_0})\right \vert}_{=0\text{ on }\mathcal{A}_{k,M}\text{ by } \eqref{eikonal equation}}\\
&\leq \left\vert  \Check {p}^x_{t_k}-\Check {p}^x_{\Check{\tau}^x_0} \right\vert + M\left\vert  \nabla \overline{\mathrm{d}}_\partial(\Check {q}^x_{t_k})-\nabla \overline{\mathrm{d}}_\partial(\Check {q}^x_{\Check{\tau}^x_0}) \right\vert \\
&\leq M\eta_N^\alpha + M^2K\eta_N
\end{aligned}\]
with $K$ the Lipschitz constant of $\nabla \overline{\mathrm{d}}_\partial$ on the compact set $\overline{\mathcal{O}}_\mu$ since $\overline{\mathrm{d}}_\partial$ is $\mathcal{C}^2$ on $\overline{\mathcal{O}}_\mu$.
Defining $M_1:=M+M^2K$, one gets for $N$ large enough so that $\eta_N\leq1$
\[\begin{aligned}
\left\vert\Check {p}^x_{t_k}\cdot\nabla \overline{\mathrm{d}}_\partial(\Check {q}^x_{t_k})\right\vert&\leq M_1 \eta_N^\alpha .
\end{aligned}\]

\noindent This yields that
\begin{align*}
&\mathbb{P}\left(\Check {p}^x_{\Check{\tau}^x_0}\cdot n(\Check {q}^x_{\Check{\tau}^x_0})=0 , \Check{\tau}^x_0\leq T ,\sup_{0\leq t\leq T}\vert \Check {p}^x_t\vert\leq M ,\sup_{0\leq s<t\leq T}\frac{\vert \Check {p}^x_t-\Check {p}^x_s\vert}{\vert t-s\vert^{\alpha}}\leq M \right)\\&\leq\sum_{k=0}^{N-1}\mathbb{P}\left( \left\vert\Check {p}^x_{t_k}\cdot\nabla \overline{\mathrm{d}}_\partial(\Check {q}^x_{t_k})\right\vert\leq M_1\eta_N^\alpha ,  \vert\overline{\mathrm{d}}_\partial(\Check {q}^x_{t_k})\vert\leq M\eta_N \right) .\numberthis \label{eqn}
\end{align*} 
 
Let $k_0:=\left\lceil\frac{4M}{\vert p\cdot n(q)\vert}\right\rceil$. For $k\in\llbracket 0,k_0-1\rrbracket$, the summand in \eqref{eqn}   vanishes when $N$ goes to infinity since $t_k=k\frac{T}{N}\leq (k_0-1)\frac{T}{N}\underset{N\rightarrow\infty}{\longrightarrow}0$ and either $\vert\overline{\mathrm{d}}_\partial(q)\vert>0$ (if $q \not\in \partial\mathcal{O})$ or $\vert p\cdot\nabla \overline{\mathrm{d}}_\partial(q)\vert>0$ (if $q \in \partial\mathcal{O}$, because $(q,p)\notin\Gamma^0$). 


\medskip \noindent \textbf{Step 2}. Let us now prove that for $k\in\llbracket k_0+1,N-1\rrbracket$, the summand in \eqref{eqn} is of order $\eta_N^{1+\alpha}$. 

It is easy to check that $(\Check {q}^x_t,\Check {p}^x_t)$ is a Gaussian vector in $\mathbb{R}^{2d}$ with law
\begin{equation}
\begin{pmatrix}
\Check{q}^x_t  \\
\Check{p}^x_t
\end{pmatrix}
  \sim  \mathcal{N}_{2d}\left(
\begin{matrix} 
\begin{pmatrix}
q+t p \\
p 
\end{pmatrix},
\begin{pmatrix}
\frac{\sigma^2 t^3}{3} I_d & \frac{\sigma^2 t^2}{2} I_d \\
\frac{\sigma^2 t^2}{2} I_d & \sigma^2 t I_d 
\end{pmatrix}
\end{matrix}
\right),
\end{equation}
where $I_d$ denotes the identity matrix of $\mathbb{R}^d$. In particular, 
$$ \Check {q}^x_t   \sim  \mathcal{N}_{d}\left(q+tp,\frac{\sigma^2 t^3}{3} I_d \right)$$ with a density denoted by $f_t$ for $t>0$. We can also compute the conditional law of $\Check {p}^x_t$ knowing $\Check {q}^x_t$ using~\cite[Prop 3.13]{Conditioning}. It is given by
$$ \mathcal{N}_{d}\left(p+\frac{3}{2t} (\Check {q}^x_t-q-tp
  ),\frac{\sigma^2 t}{4}  I_d \right),$$ with a density denoted by
$h_t(\cdot\vert \Check {q}^x_t)$. As a result, using the fact that $\vert \nabla \overline{\mathrm{d}}_\partial\vert=1$ on $\overline{\mathcal{O}}_\mu$, (see \eqref{eikonal equation}), the conditional law of $\Check {p}^x_{t}\cdot\nabla \overline{\mathrm{d}}_\partial(\Check {q}^x_{t})$ knowing $\Check {q}^x_t$, when $\Check {q}^x_{t} \in\overline{\mathcal{O}}_\mu$,  is given by
$$ \mathcal{N}\left(\left(p+\frac{3}{2t} (\Check {q}^x_t-q-tp
    )\right)\cdot\nabla \overline{\mathrm{d}}_\partial(\Check
    {q}^x_{t}),\frac{\sigma^2 t}{4}\right)$$ with a
density denoted by $g_t(\cdot\vert\Check {q}^x_t)$. 

As a consequence, for $k\in\llbracket k_0,N-1\rrbracket$, 
\[\begin{aligned}
&\mathbb{P}\big(\left\vert\Check {p}^x_{t_k}\cdot\nabla \overline{\mathrm{d}}_\partial(\Check {q}^x_{t_k})\right\vert\leq M_1\eta_N^\alpha , |\overline{\mathrm{d}}_\partial(\Check {q}^x_{t_k})| \leq M\eta_N\big)\\ 
&=\int_{q' \in \mathbb{R}^d} \mathbb{1}_{\vert\overline{\mathrm{d}}_\partial(q')\vert\leq M\eta_N} f_{t_k}(q')\left(\int_{y\in\mathbb{R}}\mathbb{1}_{\vert y\vert\leq M_1\eta_N^\alpha } g_{t_k}(y\vert q') \mathrm{d}y\right)\mathrm{d}q' .
\end{aligned}\]
Let $m_k(q'):=(p+\frac{3}{2t_k} (q'-q-t_kp ))\cdot\nabla \overline{\mathrm{d}}_\partial(q')$ then for $q'\in\overline{\mathcal{O}}_\mu$,
\[\begin{aligned}
\int_{y\in\mathbb{R}}\mathbb{1}_{\vert y\vert\leq M_1\eta_N^\alpha }  g_{t_k}(y\vert q') \mathrm{d}y&=\int_{y\in\mathbb{R}}\mathbb{1}_{\vert y\vert\leq M_1\eta_N^\alpha } \underbrace{\frac{\mathrm{e}^{-2\frac{(y-m_k(q'))^2}{ \sigma^2  t_k}}}{\sqrt{\frac{\pi \sigma^2  t_k}{2} }}}_{\leq \frac{1}{\sqrt{\frac{\pi \sigma^2  t_k}{2} }}} \mathrm{d}y\leq \frac{2\sqrt{2}M_1\eta_N^\alpha}{\sqrt{\pi \sigma^2  t_k}} .
\end{aligned}\]
Hence
\begin{align}
&\mathbb{P}\big(\left\vert\Check {p}^x_{t_k}\cdot\nabla \overline{\mathrm{d}}_\partial(\Check {q}^x_{t_k})\right\vert\leq M_1\eta_N^\alpha ,|\overline{\mathrm{d}}_\partial(\Check {q}^x_{t_k})| \leq M\eta_N\big)\nonumber\\
&\leq  \int_{q' \in \mathbb{R}^d} \mathbb{1}_{|\overline{\mathrm{d}}_\partial(q')|\leq M\eta_N} \frac{2\sqrt{2}M_1\eta_N^\alpha}{\sqrt{\pi \sigma^2  t_k}} \left(\frac{3}{2\pi  \sigma^2  t_k^3}\right)^{\frac{d}{2}} \mathrm{e}^{-\frac{3\vert z-q-t_kp \vert^2}{2 \sigma^2 t_k^3}} \mathrm{d}q'\nonumber\\
&\leq \left(\frac{3}{2\pi \sigma^2 }\right)^{\frac{d}{2}} \frac{2\sqrt{2}M_1\eta_N^\alpha}{\sqrt{\pi \sigma^2 }}\int_{\mathbb{R}^d}\frac{\mathrm{e}^{-\frac{3\vert q'-q-t_kp \vert^2}{2\sigma^2 t_k^3 }}}{t_k^{\frac{3d+1}{2}}} \mathrm{d}q'. \label{integrand non attaign}
\end{align} 
Let us now prove that the integrand is bounded by a constant independent of~$k$.

\noindent \textbf{Case a).} Assume that $q\notin\partial\mathcal{O}$, then $\vert\overline{\mathrm{d}}_\partial(q)\vert>0$ and
there exists $\mu_1>0$ such that for any $q'\in\mathcal{O}_{\mu_1}$,
$$\vert q'-q\vert\geq \sqrt{\frac{2}{3}} \vert\overline{\mathrm{d}}_\partial(q)\vert .$$
Let us pick $N$ large enough so that
\begin{equation}\label{eta}
   M\eta_N\leq \min(\mu,\mu_1).
\end{equation}
Let $C_2:=\sup_{q'\in\mathcal{O}}\vert q-q'\vert$. For $C_2 \vert p\vert t_k\leq \frac{\overline{\mathrm{d}}_\partial(q)^2}{6}$,  
\[\begin{aligned}
-\vert q'-q-t_kp \vert^2 &=-\vert q'-q\vert^2-t_k^2  \vert p\vert^2+2t_k
(q'-q)\cdot p \leq -\vert q'-q\vert^2+2t_k C_2 \vert p\vert\leq -\frac{\overline{\mathrm{d}}_\partial(q)^2}{3}
\end{aligned}\]
and $$\frac{\mathrm{e}^{-\frac{3\vert q'-q-t_kp \vert^2}{2\sigma^2 t_k^3 }}}{t_k^{\frac{3d+1}{2}}}\leq\frac{\mathrm{e}^{-\frac{\overline{\mathrm{d}}_\partial(q)^2}{2\sigma^2 t_k^3 }}}{t_k^{\frac{3d+1}{2}}} . $$
Moreover, if $C_2 \vert p\vert t_k>\frac{\overline{\mathrm{d}}_\partial(q)^2}{6 }$ (necessarily $\vert p\vert\neq0$),
\begin{equation}\label{eq:maj cas 1}
    \frac{\mathrm{e}^{-\frac{3\vert z-q-t_kp  \vert^2}{2\sigma^2 t_k^3 }}}{t_k^{\frac{3d+1}{2}}}\leq\frac{1}{\left(\frac{\overline{\mathrm{d}}_\partial(q)^2}{6 C_2 \vert p\vert}\right)^{\frac{3d+1}{2}}} .
\end{equation}
Besides, the function $t>0\mapsto
\frac{\mathrm{e}^{-\frac{\overline{\mathrm{d}}_\partial(q)^2}{2\sigma^2 t^3
    }}}{t^{\frac{3d+1}{2}}}+\frac{1}{(\frac{\overline{\mathrm{d}}_\partial(q)^2}{6
    C_2 \vert p\vert})^{\frac{3d+1}{2}}} $ is bounded by a constant
$C_3$ which depends only on $q,p$ and $d$. 

\noindent \textbf{Case b).} Assume that $q\in\partial\mathcal{O}$, then necessarily $\vert p\cdot n(q)\vert>0$ since $(q,p)\notin\Gamma^0$. By the right continuity in $0$ of $s\mapsto p\cdot\nabla\overline{\mathrm{d}}_\partial(q+sp)$, there exists $\beta>0$ such that for all $s\in[0,\beta]$, $\vert p\cdot\nabla\overline{\mathrm{d}}_\partial(q+sp)\vert\geq\frac{\vert p\cdot n(q)\vert}{2}$ (and $p\cdot\nabla\overline{\mathrm{d}}_\partial(q+sp)$ has constant sign on $[0,\beta]$). 

Assume that $t_k\leq\beta$. One has that 
$$\overline{\mathrm{d}}_\partial(q+t_kp)=\int_0^{t_k}p\cdot\nabla\overline{\mathrm{d}}_\partial(q+sp)\mathrm{d}s.$$
Therefore, since the integrand $p\cdot\nabla\overline{\mathrm{d}}_\partial(q+sp)$ has constant sign on $[0,t_k]$,
$$\vert\overline{\mathrm{d}}_\partial(q+t_kp)\vert\geq t_k\frac{\vert p\cdot n(q)\vert}{2}.$$
As a result, for $q'\in\mathcal{O}_{M\eta_N}$ since $\overline{\mathrm{d}}_\partial$ is $1-$Lipschitz continuous,
\begin{align}
    \vert q'-q-t_kp\vert&\geq\vert\overline{\mathrm{d}}_\partial(q+t_kp)-\overline{\mathrm{d}}_\partial(q')\vert\nonumber\\
    &\geq t_k\frac{\vert p\cdot\nabla\overline{\mathrm{d}}_\partial(q)\vert}{2} -M\eta_N\nonumber\\ &= t_k\frac{\vert p\cdot\nabla\overline{\mathrm{d}}_\partial(q)\vert}{2}\left(1-\frac{\eta_N}{t_k}\frac{2M}{\vert p\cdot\nabla\overline{\mathrm{d}}_\partial(q)\vert}\right).\nonumber 
\end{align}
Besides, since $k\geq k_0$,  
\begin{align}
    \frac{\eta_N}{t_k}\frac{2M}{\vert p\cdot\nabla\overline{\mathrm{d}}_\partial(q)\vert}&=\frac{1}{k}\frac{2M}{\vert p\cdot\nabla\overline{\mathrm{d}}_\partial(q)\vert}\\
    &\leq\frac{1}{k_0}\frac{2M}{\vert p\cdot\nabla\overline{\mathrm{d}}_\partial(q)\vert}\leq\frac{1}{2},
\end{align}
by definition of $k_0$. Therefore, for $t_k\leq\beta$,
$$\vert q'-q-t_kp\vert\geq t_k\frac{\vert p\cdot\nabla\overline{\mathrm{d}}_\partial(q)\vert}{4},$$
which ensures that the integrand in \eqref{integrand non attaign} is smaller than $\frac{\mathrm{e}^{-\frac{3\vert p\cdot\nabla\overline{\mathrm{d}}_\partial(q)\vert^2}{32\sigma^2 t_k} }}{t_k^{\frac{3d+1}{2}}}$ which is smaller than a constant $C_4>0$ only depending on $q$, $p$ and $d$. On the other hand, if $t_k\geq \beta$, the integrand \eqref{integrand non attaign} also admits a constant upper-bound independent of $k$. 
 
As a result, one gets that $q' \in
\mathcal{O}_{M\eta_N} \mapsto \frac{\mathrm{e}^{-\frac{3\vert q'-q-t_kp \vert^2}{2\sigma^2
      t_k^3 }}}{t_k^{\frac{3d+1}{2}}}$ is bounded by $C_5:=C_3\land C_4$, which is
independent of $k$. 

Last, using Weyl's tube formula~\cite{WeylTube}, one gets that there exists $C_6>0$ only depending on $\mathcal{O}$ such that 
$$\int_{\mathcal{O}_{M\eta_N}} \mathrm{d}q'\leq C_6 M\eta_N . $$ As a consequence,
\begin{equation}
\mathbb{P}\big(\left\vert\Check {p}^x_{t_k}\cdot\nabla
                \overline{\mathrm{d}}_\partial(\Check {q}^x_{t_k})\right\vert\leq M_1\eta_N^\alpha,  \vert\overline{\mathrm{d}}_\partial(\Check
                {q}^x_{t_k})\vert\leq M\eta_N \big)
    \leq 
\left(\frac{3}{2\pi \sigma^2 }\right)^{\frac{d}{2}} \frac{2\sqrt{2}M_1\eta_N^\alpha}{\sqrt{ \pi \sigma^2 }} C_5 C_6  M\eta_N \numberthis\label{eq12}
\end{equation} 
which is independent of $k$.

\medskip \noindent \textbf{Step 3}. Finally summing over all $k$ one gets from \eqref{eqn} and \eqref{eq12}, for $N$ large enough: 
\[\begin{aligned}
&\mathbb{P}\left(\Check {p}^x_{\Check{\tau}^x_0}\cdot n(\Check {q}^x_{\Check{\tau}^x_0})=0 , \Check{\tau}^x_0\leq T ,\sup_{0\leq t\leq T}\vert \Check {p}^x_t\vert\leq M ,\sup_{0\leq s<t\leq T}\frac{\vert \Check {p}^x_t-\Check {p}^x_s\vert}{\vert t-s\vert^{\alpha}}\leq M \right)\\ 
&\leq\underbrace{\sum_{k=0}^{k_0-1}\mathbb{P}\left( \left\vert\Check {p}^x_{t_k}\cdot\nabla \overline{\mathrm{d}}_\partial(\Check {q}^x_{t_k})\right\vert\leq M_1\eta_N^\alpha, \vert\overline{\mathrm{d}}_\partial(\Check {q}^x_{t_k})\vert\leq M\eta_N \right)}_{\underset{N\rightarrow \infty}{\longrightarrow}0}+\frac{N}{N-k_0}\underbrace{\left(\frac{3}{2\pi \sigma^2 }\right)^{\frac{d}{2}} \frac{2\sqrt{2}M_1}{\sqrt{ \pi \sigma^2 }} C_5 C_6  MT}_{\text{not depending on }N} \eta_N^{\alpha}.
\end{aligned}\]

\noindent Letting $\eta_N\underset{N\rightarrow \infty}{\longrightarrow}0$, we get $$\mathbb{P}\left(\Check {p}^x_{\Check{\tau}^x_0}\cdot n(\Check {q}^x_{\Check{\tau}^x_0})=0 , \Check{\tau}^x_0\leq T ,\sup_{0\leq t\leq T}\vert \Check {p}^x_t\vert\leq M ,\sup_{0\leq s<t\leq T}\frac{\vert \Check {p}^x_t-\Check {p}^x_s\vert}{\vert t-s\vert^{\alpha}}\leq M \right)=0 .$$
Thus, for all $\epsilon>0$,

$$ \mathbb{P}\left(\Check {p}^x_{\Check{\tau}^x_0}\cdot n(\Check {q}^x_{\Check{\tau}^x_0})=0 , \Check{\tau}^x_0\leq T \right)\leq 2 \epsilon$$
which concludes the proof of~\eqref{eqn2}. 
\end{proof} 
\section{Harnack inequality and Maximum principle}
\label{Section 3}
This section is devoted to the proof of the Harnack inequality stated in Theorem \ref{Harnack} and of the maximum principle of Theorem~\ref{maximum principle}. The proofs are respectively detailed in Sections~\ref{ss:harnack} and~\ref{ss:maximum}.

\subsection{Proof of Theorem~\ref{Harnack}}\label{ss:harnack}

A Harnack inequality for weak solutions of~\eqref{Langevin PDE} was already proven in~\cite{Har}. It says that for every point $(t_0,x_0)\in \mathbb{R}_+^*\times D$ there exist $T>0$, two small disjoint cylinders $Q^+,Q^-\subset D$ close to $x_0$ and a constant $C>0$ such that for any non-negative distributional solution $u$ of $\partial_t u=\mathcal{L}u$, we have for all $t_0\geq0$,
$$\sup_{x\in Q^-}u(t_0,x)\leq C\inf_{x\in Q^+}u(t_0+T,x). $$
 Adapting a chaining argument from~\cite{Polidoro} with a suitable chaining function, we extend this inequality to any compact set $K$ of $D$ to obtain the result \eqref{Harnack inequality}. In order to prepare the chaining argument, we first introduce some notation. For all $x,y\in D$ and $T,M,\delta>0$, we denote by $\mathcal{H}_{T,x,y,M,\delta}$ the set of $\mathcal{C}^1$ and piecewise $\mathcal{C}^2$ paths $\phi : [0,T] \to \mathcal{O}$ such that
\begin{equation*}
  \left(\phi(0),\dot{\phi}(0)\right)=x, \quad \left(\phi(T),\dot{\phi}(T)\right)=y, \quad \sup_{s\in[0,T]}\left(\left\vert\dot{\phi}\right\vert+\left\vert\ddot{\phi}\right\vert\right)(s)\leq M, \quad \inf_{s\in[0,T]}\mathrm{d}_\partial(\phi(s))>\delta.
\end{equation*}

\begin{lemma}\label{chemin}
There exists a universal constant $C>1$ such that for all $\Delta>0$, $x=(q,p),y=(q',p')\in\mathbb{R}^{2d}$, there exists $\phi\in\mathcal{C}^{2}([0,\Delta],\mathbb{R}^d)$ such that
\begin{enumerate}[label=(\roman*),ref=\roman*]
    \item $\left(\phi(0),\dot{\phi}(0)\right)=x$ and $\left(\phi(\Delta),\dot{\phi}(\Delta)\right)=y$,
    \item $\sup_{t\in[0,\Delta]}\vert\phi(t)-q\vert\leq C(\vert q'-q\vert+\Delta\vert p'-p\vert+\Delta \vert p\vert)$, 
    \item $\sup_{t\in[0,\Delta]} \vert\dot{\phi}(t)\vert \leq \frac{C}{\Delta}(\vert q'-q\vert+\Delta\vert p'-p\vert+\Delta \vert p\vert)$,
    \item $\sup_{t\in[0,\Delta]} \vert\ddot{\phi}(t)\vert \leq \frac{C}{\Delta^2}(\vert q'-q\vert+\Delta\vert p'-p\vert+\Delta \vert p\vert)$.
\end{enumerate} 
\end{lemma}
\begin{proof}
Let $\Delta>0$. For all $t\in[0,\Delta]$, we define
$$\phi(t):=q+(q'-q)\left(3 \frac{t^2}{\Delta^2}-2 \frac{t^3}{\Delta^3}\right)+\Delta (p'-p)\left(\frac{t^3}{\Delta^3}-\frac{t^2}{\Delta^2}\right)+ \Delta p\left(\frac{t}{\Delta}+2\frac{t^3}{\Delta^3}-3\frac{t^2}{\Delta^2}\right) .$$ It is easy to see that $\phi$ satisfies the conditions above.
\end{proof} 

We recall that under Assumption~\ref{hyp O}, the set $\mathcal{O}$ satisfies the uniform interior sphere condition and denote by $\rho>0$ the associated radius, given by Proposition~\ref{prop:sphere}. The following lemma is proven in Appendix~\ref{pf:cste_Harnack}. 

\begin{lemma}[Admissible paths] \label{cste_Harnack} Under Assumptions~\ref{hyp O} and~\ref{hyp:conn}, let $K\subset D$ be a compact set and $\delta_K:=\mathrm{d}(K,\partial D)\land\rho$. For all $T>0$, there exists $M_{K,T}>0$ such that for all $x,y\in K$ the set $\mathcal{H}_{T,x,y,M_{K,T},\delta_K/2}$ is nonempty.
\end{lemma}

Let us now prove the Harnack inequality in Theorem \ref{Harnack}.
\begin{proof}[Proof of Theorem~\ref{Harnack}]
Let $K\subset D$ be a compact set. Let $T>0$ and let $u$ be a
non-negative distributional solution of $\partial_t u-\mathcal{L}u=0$
on $\mathbb{R}_+^*\times D$. The proof is divided into two steps. In
the first step, we introduce the necessary background in order to
apply the Harnack inequality from~\cite{Har}. In the second step, we
detail the chaining argument, based on Lemma~\ref{cste_Harnack}, which
allows us to obtain the Harnack inequality of Theorem \ref{Harnack}.

\medskip \noindent \textbf{Step 1}. Let $M_{K,T}>0$ and $\delta_{K}>0$
be the constants given in Lemma \ref{cste_Harnack} and let us define
the constant
$r_{K,T}:=\sqrt{\frac{\delta_{K}}{1+M_{K,T}}}\land\frac{1}{2}$. Let $r\in(0,r_{K,T}]$. Let us define $$D_{K,T,r}=\{(t,q,p)\in \mathbb{R}_+^* \times D : t>r^2, \mathrm{d}_\partial(q)>\delta_{K}/2 , \vert p\vert\leq M_{K,T} \}. $$ Notice that $(r^2,\infty)\times K\subset D_{K,T,r}$. Let $Q$ be the following unit box $$Q:=\{(t,q,p)\in\mathbb{R}\times\mathbb{R}^{2d} : t\in(-1,0],\vert q\vert<1 , \vert p\vert<1 \} .$$
For all $z_0=(t_0,q_0,p_0)\in D_{K,T,r}$, let us define the following function on $Q$ 
$$h_{r,z_0}:(t,q,p) \mapsto (r^2t+t_0,q_0-r^2tp_0+r^3q,p_0-r p).$$
 Notice that for all $z_0\in D_{K,T,r}$ and $(t,q,p)\in Q$,
$$\vert-r^2tp_0+r^3q\vert\leq M_{K,T}r^2_{K,T}+r^3_{K,T}<
r^2_{K,T}(1+M_{K,T})\leq\delta_{K} ,$$ since $r_{K,T}\in(0,1)$. As a
result, $h_{r,z_0}$ is a function on $Q$ with values in
$\mathbb{R}_+^*\times D$.

Since $\partial_t-\mathcal{L}$ is a hypoelliptic operator on
$\mathbb{R}_+^*\times D$ it follows that $u$ is in
$\mathcal{C}^\infty(\mathbb{R}_+^*\times D)$. Let us now define the following smooth function
\begin{equation*}
  u_{r,z_0} :=u\circ h_{r,z_0}
\end{equation*}
on $Q$. It satisfies  
\begin{equation*}
  \partial_tu_{r,z_0} =-p\cdot\nabla_qu_{r,z_0}+\gamma (r p_0-r^2 p)\cdot\nabla_pu_{r,z_0}-r F(q_0-r^2tp_0+r^3q)\cdot\nabla_pu_{r,z_0}+\frac{\sigma^2}{2}\Delta_p u_{r,z_0}.
\end{equation*}
Besides for $(t,q,p)\in Q$, 
\[\begin{aligned}
\left\vert\gamma (r p_0-r^2 p)-r F(q_0-r^2tp_0+r^3q)\right\vert&\leq \vert\gamma\vert \left(M_{K,T}+\delta_K\right)+\|F\|_{\mathrm{L}^\infty(\mathcal{O})}
\end{aligned}\] 
which is a constant depending only on the compact $K$, $T$ and the coefficients of $\mathcal{L}$. As a result, Theorem 4 in~\cite{Har} ensures the existence of constants $C_{K,T}>1$ and $R_{K,T},\Delta_{K,T}\in(0,1)$ (which do not depend on $r \in (0,r_{K,T}]$ or $z_0$) such that $\Delta_{K,T} + R^2_{K,T} < 1$ and
\begin{equation}\label{Harnack1}
    \sup_{(t,q,p)\in Q_{K,T}^-}u_{r,z_0}(t,q,p)\leq C_{K,T}\inf_{(t,q,p)\in Q_{K,T}^+}u_{r,z_0}(t,q,p) 
\end{equation} 
where $$ Q_{K,T}^+:=\{(t,q,p) : t\in(-R_{K,T}^2,0],\vert q\vert<R_{K,T}^3 , \vert p\vert<R_{K,T} \}\subset Q,$$ $$ Q_{K,T}^-:=\{(t,q,p) : t\in(-R_{K,T}^2-\Delta_{K,T},-\Delta_{K,T}],\vert q\vert<R_{K,T}^3 , \vert p\vert<R_{K,T} \}\subset Q .$$
Introducing the notation $Q_{K,T,r}^\pm(z_0) = h_{r,z_0}(Q^\pm_{K,T})$, \eqref{Harnack1} rewrites, for all $r \in (0,r_{K,T}]$ and $z_0 \in D_{K,T,r}$,
\begin{equation}\label{Harnack2}
    \sup_{(t,q,p)\in Q_{K,T,r}^-(z_0)}u(t,q,p)\leq C_{K,T}\inf_{(t,q,p)\in Q_{K,T,r}^+(z_0)}u(t,q,p) . 
\end{equation} 

\medskip \noindent\textbf{Step 2}. Let $\epsilon>0$. Let us choose $r^\epsilon_{K,T}$ satisfying
\begin{enumerate}[label=(\roman*),ref=\roman*]
    \item $0<r^\epsilon_{K,T}\leq r_{K,T}$,
    \item $r^\epsilon_{K,T}<\frac{2  R_{K,T}^3}{M_{K,T}\left(\Delta_{K,T}+\frac{R_{K,T}^2}{2}\right)^2}\land\frac{R_{K,T}}{M_{K,T}\left(\Delta_{K,T}+\frac{R_{K,T}^2}{2}\right)}\land\sqrt{\frac{\epsilon}{1-\Delta_{K,T}-\frac{R_{K,T}^2}{2}}}$,
    \item the quantity $\alpha^\epsilon_{K,T}:=(r_{K,T}^\epsilon)^2\left(\Delta_{K,T}+\frac{R_{K,T}^2}{2}\right)$ is such that $n^\epsilon_{K,T}:=\frac{T}{\alpha_{K,T}^\epsilon}\in\mathbb{N}$.
\end{enumerate}
\noindent Let $t\geq T+\epsilon$. Let $(x,y)$ be two arbitrary points in the compact set $K$. Let $\phi\in\mathcal{H}_{T,x,y,M_{K,T},\delta_{K}/2}$ (which exists by Lemma \ref{cste_Harnack}) and let $$\Phi:s\in[0,T]\mapsto 
\begin{pmatrix}
\phi(s)  \\
\dot{\phi}(s)\\
t-s
\end{pmatrix}\in\mathbb{R}^{2d+1} .$$

\noindent Now let $(s^\epsilon_j)_{0\leq j\leq n^\epsilon_{K,T}}$ be the sequence defined by  $s^\epsilon_j:=j \alpha^\epsilon_{K,T}$ for $0\leq j\leq n^\epsilon_{K,T}$.

Let us show that $\Phi(s^\epsilon_j)\in D_{K,T,r^\epsilon_{K,T}}$ for all $0\leq j\leq n^\epsilon_{K,T}-1$. Indeed, one has  
$$t-s_{n^\epsilon_{K,T}-1}=t-(n^\epsilon_{K,T}-1) \alpha^\epsilon_{K,T}=t-T+(r_{K,T}^\epsilon)^2\left(\Delta_{K,T}+\frac{R_{K,T}^2}{2}\right)>(r^\epsilon_{K,T})^2,$$
since $(r^\epsilon_{K,T})^2<\frac{\epsilon}{1-\Delta_{K,T}-\frac{R_{K,T}^2}{2}}.$ The rest follows from the definition of $\mathcal{H}_{T,x,y,M_{K,T},\delta_{K}/2}$.
Hence, \eqref{Harnack2} is satisfied for $r=r^\epsilon_{K,T}$ and $z_0=\Phi(s^\epsilon_j)$ for all $0\leq j\leq n^\epsilon_{K,T}-1$, i.e. $$\sup_{(t,q,p)\in Q_{K,T,r^\epsilon_{K,T}}^-(\Phi(s^\epsilon_j))}u(t,q,p)\leq C_{K,T}\inf_{(t,q,p)\in Q_{K,T,r^\epsilon_{K,T}}^+(\Phi(s^\epsilon_j))}u(t,q,p).$$

Let us now prove that for every $0\leq j\leq n^\epsilon_{K,T}-1$, $\Phi(s^\epsilon_{j+1})\in Q_{K,T,r^\epsilon_{K,T}}^-(\Phi(s^\epsilon_j))$. Let
\begin{enumerate}[label=(\roman*),ref=\roman*]
    \item $\widehat{t}_j:=-\frac{\alpha^\epsilon_{K,T}}{(r^\epsilon_{K,T})^2}=-\Delta_{K,T}-\frac{R_{K,T}^2}{2}$,
    \item $\widehat{q}_j:=\frac{1}{(r^\epsilon_{K,T})^3}\left(\phi(s^\epsilon_{j+1})-\phi(s^\epsilon_{j})-\alpha^\epsilon_{K,T} \dot{\phi}(s^\epsilon_{j})\right)$,
    \item $\widehat{p}_j:=\frac{1}{r^\epsilon_{K,T}}\left(\dot{\phi}(s^\epsilon_{j})-\dot{\phi}(s^\epsilon_{j+1})\right)$.
\end{enumerate}
  Then it only remains to prove that
  $(\widehat{t}_j,\widehat{q}_j,\widehat{p}_j)\in Q_{K,T}^-$ for every
  $0\leq j\leq n^\epsilon_{K,T}-1$, i.e. that
  $$h_{r^\epsilon_{K,T},\Phi(s^\epsilon_{j})}(\widehat{t}_j,\widehat{q}_j,\widehat{p}_j)=\Phi(s^\epsilon_{j+1}).$$

First, concerning $\widehat{t}_j$, it is clear by definition of $\alpha^\epsilon_{K,T}$ that $$-\Delta_{K,T}-R_{K,T}^2< -\frac{\alpha^\epsilon_{K,T}}{(r^\epsilon_{K,T})^2}\leq -\Delta_{K,T} .$$

Second, for $\widehat{q}_j$,
\[\begin{aligned}
\left\vert\phi(s^\epsilon_{j+1})-\phi(s^\epsilon_{j})-\alpha^\epsilon_{K,T} \dot{\phi}(s^\epsilon_j)\right\vert&=\left\vert\int_{s^\epsilon_j}^{s^\epsilon_{j+1}}\left(\dot{\phi}(\eta)-\dot{\phi}(s^\epsilon_j)\right)\mathrm{d}\eta\right\vert\\
&\leq \int_{s^\epsilon_{j}}^{s^\epsilon_{j+1}}\int_{s^\epsilon_j}^{\eta}\left\vert\ddot{\phi}(\mu)\right\vert \mathrm{d}\mu \mathrm{d}\eta\\
&\leq
M_{K,T}\int_{s^\epsilon_{j}}^{s^\epsilon_{j+1}}\left(\eta-s^\epsilon_j\right)
\mathrm{d}\eta \leq M_{K,T} \frac{(\alpha^\epsilon_{K,T})^2}{2}
\end{aligned}\]
and since 
\begin{equation*}
  r_{K,T}^\epsilon<\frac{2  R_{K,T}^3}{M_{K,T}\left(\Delta_{K,T}+\frac{R_{K,T}^2}{2}\right)^2},
\end{equation*}
we have
\begin{equation*}
  (r_{K,T}^\epsilon)^4\left(\Delta_{K,T}+\frac{R_{K,T}^2}{2}\right)^2<\frac{2 (r_{K,T}^\epsilon)^3 R_{K,T}^3}{M_{K,T}},
\end{equation*}
and therefore
\begin{equation*}
  M_{K,T} \frac{(\alpha^\epsilon_{K,T})^2}{2}<(r_{K,T}^\epsilon)^3 R_{K,T}^3.
\end{equation*}

Third, for $\widehat{p}_j$,
\[\begin{aligned}
\left\vert\dot{\phi}(s^\epsilon_{j+1})-\dot{\phi}(s^\epsilon_{j})\right\vert&\leq \int_{s^\epsilon_{j}}^{s^\epsilon_{j+1}}\left\vert\ddot{\phi}(\eta)\right\vert \mathrm{d}\eta\leq M_{K,T} \alpha^\epsilon_{K,T}\\
\end{aligned}\]
and the assumption that
\begin{equation*}
  r^\epsilon_{K,T}<\frac{R_{K,T}}{M_{K,T}\left(\Delta_{K,T}+\frac{R_{K,T}^2}{2}\right)}
\end{equation*}
ensures that
\begin{equation*}
  M_{K,T} \alpha^\epsilon_{K,T}<r^\epsilon_{K,T} R_{K,T}.
\end{equation*}
Hence $\Phi(s^\epsilon_{j+1})\in Q_{K,T,r^\epsilon_{K,T}}^-(\Phi(s^\epsilon_j))$.

Finally, one gets for $0\leq j\leq n^\epsilon_{K,T}-1$
$$u(\Phi(s^\epsilon_{j+1}))\leq\sup_{Q_{K,T,r^\epsilon_{K,T}}^-(\Phi(s^\epsilon_j))}u\leq C_{K,T}\inf_{Q_{K,T,r^\epsilon_{K,T}}^+(\Phi(s^\epsilon_j))}u\leq C_{K,T} u(\Phi(s^\epsilon_{j}))$$
which yields by iterating,
$$u(\Phi(T))=u(t-T,y)\leq C_{K,T}^{n^\epsilon_{K,T}} u(\Phi(0))=C_{K,T}^{n^\epsilon_{K,T}} u(t,x) $$
where $C_{K,T}^{n^\epsilon_{K,T}}$ does not depend on $(t,x,y)$ but only on the compact set $K$ and $T,\epsilon$. As a result, we have for all $t\geq T+\epsilon$,
$$\sup_{x\in K} u(t-T,x)\leq C_{K,T}^{n^\epsilon_{K,T}} \inf_{x\in K}u(t,x) $$ which concludes the proof.
\end{proof}  

\subsection{Proof of Theorem~\ref{maximum principle}}\label{ss:maximum}

In order to prove the maximum principle stated in Theorem \ref{maximum principle}, we need the following lemma.

\begin{lemma}[Irreducibility] \label{lemma irreducibilty} Let Assumptions \ref{hyp F1}, \ref{hyp O} and \ref{hyp:conn} hold. Let $A$ be an open subset of~$D$, then 
$$\forall x\in D,\quad\forall   t>0,\quad\forall s\in(0,t),\qquad\mathbb{P}(X^x_s\in A,\tau^x_{\partial}>t)>0.$$ 
\end{lemma} 
\begin{proof}
Let $x\in D,t>0,s\in(0,t)$. Let $A$ be an open subset of $D$, the Markov property at time $s$ ensures that
$$\mathbb{P}(X^x_s\in A,\tau^x_{\partial}>t)=\mathbb{E}\left[\mathbb{1}_{X^x_s\in A,\tau^x_{\partial}>s}\mathbb{P}(\tau^y_{\partial}>t-s)\vert_{y=X^x_s}\right].$$
By Theorem \ref{thm density intro}, which is proven in Sections \ref{Section 4} and \ref{Section Adjoint process and compactness}, the kernel $\mathrm{P}^D_s(x,\cdot)$ defined in \eqref{kernel D} admits a positive density function $\mathrm{p}^D_t(x,\cdot)$. Therefore, $\mathbb{P}(X^x_s\in A,\tau^x_{\partial}>s)>0$. Again, by the positivity of  $\mathrm{p}^D_{t-s}(\cdot,\cdot)$ in Theorem \ref{thm density intro}, on the event $\{X^x_s\in A,\tau^x_{\partial}>s\}$, one has that $\mathbb{P}(\tau^y_{\partial}>t-s)\vert_{y=X^x_s}>0$ almost surely. This concludes the proof using the Markov property stated above.
\end{proof}
Let us now prove Theorem~\ref{maximum principle}.
\begin{proof}[Proof of Theorem \ref{maximum principle}]

\noindent Let $x\in D$. Let $(X^x_t=(q^x_t,p^x_t))_{t\geq0}$ be the strong solution of \eqref{Langevin} on $\mathbb{R}^{2d}$. For $k>0$, let $V_k$ be the following open and bounded subset of $D$  
\begin{equation*}
    V_k:=\left\{(q,p)\in D : \vert p\vert< k, \mathrm{d}_\partial(q)>\frac{1}{k} \right\}.
\end{equation*}
Let $\tau^{x}_{V_k^c}$ be the following stopping time:
$$\tau^x_{V_k^c}=\inf \{t>0: X^x_t\notin V_k\}.$$
Let $t>0$ and $s\in[0,t)$. Since
$u\in\mathcal{C}^{1,2}(\mathbb{R}_+^{*}\times D)$, Itô's formula
applied to the process $(u(t-r,X^x_{r}))_{0\leq r\leq s}$ between $0$
and $s\land \tau^x_{V_k^c}$ yields: almost surely, for $s\in[0,t)$, 
\begin{align*}
u(t,x)&=\mathbb{E}\bigg[ \mathbb{1}_{\tau^x_{V_k^c}>s} u(t-s,X^x_s) \bigg]+\mathbb{E}\bigg[ \mathbb{1}_{\tau^x_{V_k^c}\leq s} u(t-\tau^x_{V_k^c},X^x_{\tau^x_{V_k^c}})\bigg]\\
&\quad +\mathbb{E}\bigg[ \int_0^{s\land\tau^x_{V_k^c}}\left(\partial_tu(t-r,X^x_{r})-\mathcal{L}u(t-r,X^x_{r})\right)\mathrm{d}r\bigg].\numberthis\label{formule ito max principle}
\end{align*}

\medskip \noindent\textbf{Step 1}. Let us prove Assertion \eqref{weak
  principle} in Theorem \ref{maximum principle} using \eqref{formule
  ito max principle}. It follows from \eqref{formule ito max principle} and the inequality $\partial_tu-\mathcal{L}u\leq0$ on $\mathbb{R}_+^*\times D$, that
$$u(t,x)\leq\mathbb{E}\bigg[ \mathbb{1}_{\tau^x_{V_k^c}>s} u(t-s,X^x_s) \bigg]+\mathbb{E}\bigg[ \mathbb{1}_{\tau^x_{V_k^c}\leq s} u(t-\tau^x_{V_k^c},X^x_{\tau^x_{V_k^c}})\bigg].$$
By assumption, $u\in \mathcal{C}^b((\mathbb{R}_+\times\overline{D})
\setminus (\{0\}\times(\Gamma^+\cup\Gamma^0)))$. Therefore, following
the same reasoning as in the proof of Assertion~\eqref{it:ibvp:uniq}
of Theorem~\ref{Solution PDE} in Section~\ref{ss:ibvp:uniq}, one
obtains by letting $s\rightarrow t$ and $k\rightarrow\infty$ that
$$u(t,x)\leq\mathbb{E}\bigg[ \mathbb{1}_{\tau^x_\partial>t} u(0,X^x_t) \bigg]+\mathbb{E}\bigg[ \mathbb{1}_{\tau^x_\partial< t} u(t-\tau^x_\partial,X^x_{\tau^x_\partial})\bigg].$$ 
Since $X^x_{\tau^x_\partial} \in\Gamma^+$ almost surely by Proposition \ref{prop:tau}, the inequality above immediately yields Assertion~\eqref{weak principle}.

\medskip \noindent\textbf{Step 2}. We now prove Assertion
\eqref{strong principle}. Applying the equality \eqref{formule ito max principle} for $(t,x)=(t_0,x_0)$ and subtracting $u(t_0,x_0)$, we obtain that for all $s\in[0,t_0)$,
\begin{align*}
    0&=\mathbb{E}\bigg[ \mathbb{1}_{\tau^{x_0}_{V_k^c}>s} \bigg(u(t_0-s,X^{x_0}_s)-u(t_0,x_0)\bigg)\bigg]+\mathbb{E}\bigg[ \mathbb{1}_{\tau^{x_0}_{V_k^c}\leq s} \bigg(u(t_0-\tau^{x_0}_{V_k^c},X^{x_0}_{\tau^{x_0}_{V_k^c}})-u(t_0,x_0)\bigg)\bigg]\\
    &\quad +\mathbb{E}\bigg[ \int_0^{s\land\tau^{x_0}_{V_k^c}}\bigg(\partial_tu(t_0-r,X^{x_0}_{r})-\mathcal{L}u(t_0-r,X^{x_0}_{r})\bigg)\mathrm{d}r\bigg].
\end{align*}
Using the fact that $u(t_0,x_0)=\Vert u\Vert_\infty$ and that
$\partial_tu-\mathcal{L}u\leq0$ on $\mathbb{R}_+^*\times D$, it
follows that, necessarily, for all $k>0$ and $s\in[0,t_0)$, (since $\mathbb{1}_{\tau^{x_0}_{V_k^c}>t_0}
\le \mathbb{1}_{\tau^{x_0}_{V_k^c}>s}$)
$$\mathbb{E}\bigg[ \mathbb{1}_{\tau^{x_0}_{V_k^c}>t_0} \bigg(u(t_0-s,X^{x_0}_s)-u(t_0,x_0)\bigg)\bigg]=0.$$
Taking $k\rightarrow\infty$ as in the proof of Assertion~\eqref{it:ibvp:uniq} of Theorem~\ref{Solution PDE}, one obtains that for all $s\in[0,t_0)$,
\begin{equation}\label{ineq step 2 max principle}
    \mathbb{E}\bigg[ \mathbb{1}_{\tau^{x_0}_\partial>t_0} \bigg(u(t_0-s,X^{x_0}_s)-u(t_0,x_0)\bigg)\bigg]=0. 
\end{equation}

Assume now that Assertion \eqref{strong principle} is not satisfied, then there exist $c>0$, $s_0\in(0,t_0)$ and an open subset $A$ of $D$ such that for all $y\in A$, $u(t_0-s_0,y)-u(t_0,x_0)\leq-c$. Therefore,
$$\mathbb{E}\bigg[\mathbb{1}_{\tau^{x_0}_\partial>t_0,X^{x_0}_{s_0}\in A} \bigg(u(t_0-s_0,X^{x_0}_{s_0})-u(t_0,x_0)\bigg)\bigg]\leq-c\mathbb{P}\bigg(\tau^{x_0}_\partial>t_0,X^{x_0}_{s_0}\in A\bigg)<0,$$ by Lemma \ref{lemma irreducibilty}. Moreover,
$$\mathbb{E}\bigg[ \mathbb{1}_{\tau^{x_0}_\partial>t_0} \bigg(u(t_0-s_0,X^{x_0}_{s_0})-u(t_0,x_0)\bigg)\bigg]\leq\mathbb{E}\bigg[\mathbb{1}_{\tau^{x_0}_\partial>t_0,X^{x_0}_{s_0}\in A} \bigg(u(t_0-s_0,X^{x_0}_{s_0})-u(t_0,x_0)\bigg)\bigg]<0,$$ which is in contradiction with \eqref{ineq step 2 max principle}, hence Assertion \eqref{strong principle}.
\end{proof}

\section{Gaussian upper bound and existence of a smooth transition density for the absorbed Langevin process}
\label{Section 4}\sectionmark{Gaussian upper-bound and absorbed transition density}
 
The proof of  the Gaussian upper bound stated in Theorem \ref{borne
  densite thm} is provided in Section~\ref{subsection
  Gaussian}. Section~\ref{smooth density subsection} is devoted to the
proof of the existence of a smooth transition density for the absorbed
Langevin process from Definition~\ref{defi:absorbed}, and the fact that this density
satisfies the backward and forward Kolmogorov equations. This yields
the first part of Theorem~\ref{thm density intro}, the boundary continuity
will be proven in Section~\ref{Section Adjoint process and
  compactness}. Section~\ref{sec:boundary} is devoted to the study of some preliminary boundary continuity properties of the transition density for the absorbed
Langevin process \eqref{Langevin} which will be useful in Section~\ref{Section Adjoint process and
  compactness}.

\subsection{\texorpdfstring{Gaussian upper bound for the Langevin process in $\mathbb{R}^d$}{}} \label{subsection Gaussian}

The purpose of this section is to provide a Gaussian upper bound
satisfied by the transition density $\mathrm{p}_t(x,y)$ of the process
$(X^x_t=(q^x_t,p^x_t))_{t\geq0}$ defined by \eqref{Langevin} under
Assumption \ref{hyp F}. We do not consider absorption in this section.

For $x=(q,p)\in\mathbb{R}^{2d}$, let  $(\widehat{X}^x_t=(\widehat{q}_t^{x},\widehat{p}_t^{x}))_{t\geq0}$ be the strong solution on $\mathbb{R}^{2d}$ of the following SDE
\begin{equation}\label{Langevin no potential}
  \left\{
    \begin{aligned}
&        \mathrm{d}\widehat{q}^{x}_t=\widehat{p}^{x}_t \mathrm{d}t ,\\
 &       \mathrm{d}\widehat{p}^{x}_t=-\gamma \widehat{p}^{x}_t \mathrm{d}t+\sigma \mathrm{d}B_t ,\\
  &      (\widehat{q}^{x}_0,\widehat{p}^{x}_0)= x,
    \end{aligned}
\right.   
\end{equation}
with infinitesimal generator $\widehat{\mathcal{L}} := \mathcal{L}_{0,\gamma,\sigma}$. Let $\Phi_1,\Phi_2$ be the following positive continuous functions on~$\mathbb{R}$:
\begin{equation}\label{Phi_1}
\Phi_1:\rho\in\mathbb{R}\mapsto\begin{cases}
  \frac{1-\mathrm{e}^{-\rho}}{\rho}&\text{if $\rho\neq0$,}\\ 
  1 & \text{if $\rho=0$,} 
\end{cases}
\end{equation}
\begin{equation}\label{Phi_2}
\Phi_2:\rho\in\mathbb{R}\mapsto\begin{cases}
  \frac{3}{2 \rho^3}\left[2\rho-3+4 \mathrm{e}^{-\rho}-\mathrm{e}^{-2\rho}\right] & \text{if $\rho\neq0$,}\\
  1 &\text{if $\rho=0$.}
\end{cases}
\end{equation}
The process $(\widehat{q}_t^{x},\widehat{p}_t^{x})_{t \geq 0}$ is Gaussian and for all $t \geq 0$, the vector $(\widehat{q}_t^{x},\widehat{p}_t^{x})$ admits the following law 
\begin{equation}\label{loi gaussienne}
\begin{pmatrix}
\widehat{q}^x_t  \\
\widehat{p}^x_t
\end{pmatrix}
  \sim  \mathcal{N}_{2d}\left(
\begin{matrix} 
\begin{pmatrix}
m_q^x(t)  \\
m_p^x(t) 
\end{pmatrix},
C(t)
\end{matrix} 
\right)  , 
\end{equation}
where the mean vector is $$m_q^x(t):=q+t p \Phi_1(\gamma t),\qquad m_p^x(t):=p \mathrm{e}^{-\gamma  t},$$
and the covariance matrix is
$$C(t):= \begin{pmatrix}
c_{qq}(t)  I_d & c_{qp}(t)  I_d \\
c_{qp}(t)  I_d & c_{pp}(t)  I_d 
\end{pmatrix}, 
$$
where $I_d$ is the identity matrix in $\mathbb{R}^{d\times d}$ and
\begin{equation}\label{coeff cov}
    c_{qq}(t):=\frac{\sigma^2 t^3}{3} \Phi_2(\gamma t),\qquad c_{qp}(t):=\frac{\sigma^2 t^2}{2} \Phi_1(\gamma t)^2,\qquad c_{pp}(t):=\sigma^2 t \Phi_1(2\gamma t)  . 
\end{equation} 
The determinant of the covariance matrix $C(t)$ is $\mathrm{det}(C(t))=(\frac{\sigma^4 t^4}{12} \phi(\gamma t))^d$ where $\phi$ is the positive continuous function defined by 
\begin{equation}\label{expr phi}
\phi:\rho\in\mathbb{R}\mapsto 4 \Phi_2(\rho) \Phi_1(2 \rho)-3 \Phi_1(\rho)^4=\begin{cases}
    \frac{6 (1-\mathrm{e}^{-\rho})}{\rho^4}\left[-2+\rho+(2+\rho)\mathrm{e}^{-\rho}\right]&\text{if $\rho\neq0$,}\\
   1 &\text{if $\rho=0$.}
\end{cases}
\end{equation} As a result, one can easily obtain an explicit expression of the transition density $\widehat{\mathrm{p}}_t((q,p),(q',p'))$ of the process $(\widehat{q}_t^{x},\widehat{p}_t^{x})_{t\geq0}$: for $t>0$, $(q,p),(q',p')\in\mathbb{R}^{2d}$,
\begin{equation}\label{expr densite}
    \widehat{\mathrm{p}}_t((q,p),(q',p')):=\frac{1}{ \sqrt{(2 \pi)^{2d} \left(\frac{\sigma^4 t^4}{12} \phi(\gamma t)\right)^d}} \mathrm{e}^{-\frac{\delta x(t)\cdot C^{-1}(t) \delta x(t)}{2}} 
\end{equation}
where 
\begin{align*}
\delta x(t):=
\begin{pmatrix}
\delta q(t) \\
\delta p(t) 
\end{pmatrix}:=
\begin{pmatrix}
q'-m_q^x(t)  \\
p'-m_p^x(t) 
\end{pmatrix},\qquad C^{-1}(t)=\frac{1}{\frac{\sigma^4 t^4}{12} \phi(\gamma t)} \begin{pmatrix}
c_{pp}(t)  I_d & -c_{qp}(t)  I_d \\
-c_{qp}(t)  I_d & c_{qq}(t)  I_d 
\end{pmatrix} . \numberthis\label{def moyenne cov}
\end{align*} 

We now give a useful rewriting of $\delta x(t)\cdot C^{-1}(t)\delta x(t)$ as
a sum of squares, inspired by~\cite[Equation~(2.5)]{FS}, using an
additional positive continuous function on $\mathbb{R}$:
\begin{equation}\label{Phi_3}
\Phi_3:\rho\in\mathbb{R}\mapsto
\begin{cases}
   \frac{2(1-\Phi_1(\rho))}{\rho}&\text{if $\rho\neq0$,}\\
   1 &\text{if $\rho=0$.} 
\end{cases}
\end{equation}
The proof of the following lemma is detailed in Appendix~\ref{pf:arg densite}.
\begin{lemma}[Covariance decomposition] \label{arg densite}
For all $t>0$, $\delta x=\begin{pmatrix}
\delta q\\
\delta p
\end{pmatrix}\in\mathbb{R}^{2d}$,
\begin{equation}\label{ecriture arg densite}
    \delta x\cdot C^{-1}(t)\delta x =\frac{1}{\sigma^2 t} \left\vert\Pi_1 \delta x\right\vert^2+\frac{12}{\sigma^2 t^3 \phi(\gamma t)} \left\vert\Pi_2(t) \delta x\right\vert^2 , 
\end{equation}
where $\Pi_1:=\begin{pmatrix}
\gamma I_d & I_d
\end{pmatrix}\in\mathbb{R}^{d\times 2d}$ and $\Pi_2(t):=\begin{pmatrix}
\Phi_1(\gamma t) I_d & -\frac{t}{2} \Phi_3(\gamma t) I_d
\end{pmatrix}\in\mathbb{R}^{d\times 2d}$. 
\end{lemma}

Now let $\alpha\in(0,1]$. For $x=(q,p),y=(q',p')\in\mathbb{R}^{2d}$ and $t>0$, let $\widehat{\mathrm{p}}^{(\alpha)}_t(x,y)$ be the transition density of the process $(\alpha^{-1/2}\widehat{X}^{\sqrt{\alpha}x}_t)_{t\geq0}$, with infinitesimal generator $\mathcal{L}_{0,\gamma,\sigma/\sqrt{\alpha}}$, i.e.
\begin{equation}\label{densite p^alpha}
    \widehat{\mathrm{p}}^{(\alpha)}_t((q,p),(q',p')):=\sqrt{\alpha^{2d}} \widehat{\mathrm{p}}_t(\sqrt{\alpha} (q,p),\sqrt{\alpha} (q',p')) . 
\end{equation}
Let us state the following useful properties which are also proven in Appendix~\ref{pf:arg densite}.
\begin{lemma}[Transition density properties]\label{prop densite}
 The transition densities $\widehat{\mathrm{p}}_t$ and
 $\widehat{\mathrm{p}}^{(\alpha)}_t$ satisfy:
 \begin{enumerate}[label=(\roman*),ref=\roman*]
 \item For all $t>0$,  and $x=(q,p),y=(q',p') \in\mathbb{R}^{2d}$,
   (using the notation~\eqref{def moyenne cov}) 
\begin{equation}\label{eq:ppalpha}
   \widehat{\mathrm{p}}_t(x,y)=\frac{1}{\sqrt{\alpha^{2d}}}
     \mathrm{e}^{-\frac{1-\alpha}{2}\delta x(t)\cdot C^{-1}(t)\delta x(t)} \widehat{\mathrm{p}}^{(\alpha)}_t(x,y).
\end{equation}
\item Chapman-Kolmogorov relation: For all $t>0$, for all $u\in(0,t)$ and $x,y\in\mathbb{R}^{2d}$,
\begin{equation}\label{chp kolmogorov}
    \int_{\mathbb{R}^{2d}}\widehat{\mathrm{p}}^{(\alpha)}_{u}(x,z) \widehat{\mathrm{p}}^{(\alpha)}_{t-u}(z,y) \mathrm{d}z=\widehat{\mathrm{p}}^{(\alpha)}_{t}(x,y) . 
\end{equation}
\item For all $t>0$,  $\varphi\in\mathcal{C}^b(\mathbb{R}_+\times\mathbb{R}^{2d})$ and $y,x_0, y_0\in\mathbb{R}^{2d}$, 
\begin{equation}\label{eq8}
    \int_{\mathbb{R}^{2d}} \widehat{\mathrm{p}}^{(\alpha)}_{t}(x,y) \mathrm{d}x=\mathrm{e}^{d \gamma t} 
\end{equation}
and  
\begin{align*}
     \int_{\mathbb{R}^{2d}}\widehat{\mathrm{p}}^{(\alpha)}_{t}(x,y) \varphi(t,x) \mathrm{d}x \underset{(t,y)\rightarrow(0,y_0)}{{\longrightarrow}} \varphi(0,y_0),\qquad\int_{\mathbb{R}^{2d}}\widehat{\mathrm{p}}^{(\alpha)}_{t}(x,y) \varphi(t,y) \mathrm{d}y \underset{(t,x)\rightarrow(0,x_0)}{{\longrightarrow}} \varphi(0,x_0) . \numberthis\label{chgt de var + cv}
\end{align*}
\item For all $\alpha \in (0,1)$, there exists $c_\alpha>0$ depending only on $\alpha$ such that for all $ t>0$, $x,y\in\mathbb{R}^{2d}$,
\begin{equation}\label{maj gradient p}
    \vert\nabla_{p }\widehat{\mathrm{p}}_{t }(x,y)\vert\leq \frac{c_\alpha(1+\sqrt{\gamma_- t}) }{\sqrt{\sigma^2 t }} \widehat{\mathrm{p}}^{(\alpha)}_{t }(x,y), 
\end{equation}
where $\gamma_-$ is the negative part of $\gamma\in\mathbb{R}$. 
\end{enumerate}  
\end{lemma}

We are now in position to prove Theorem \ref{borne densite thm}. 
\begin{proof}[Proof of Theorem \ref{borne densite thm}]
The idea is to first establish a mild formulation of the difference
between the two transition densities $\mathrm{p}_t(x,y)$ and
$\widehat{\mathrm{p}}_t(x,y)$, adapting the reasoning from
\cite{Menozzi}. Secondly, iterating the obtained equality, one
obtains, following the steps of \cite{Menozzi}, an expression of the difference between $\mathrm{p}_t(x,y)$ and $\widehat{\mathrm{p}}_t(x,y)$ in the form of a series, which then yields the Gaussian
upper bound stated in Theorem \ref{borne densite thm}.

\medskip \noindent\textbf{Step 1}. Let us first obtain the mild
formulation linking $\mathrm{p}_t(x,y)$ and
$\widehat{\mathrm{p}}_t(x,y)$. Let $T>0$ and
$\varphi\in\mathcal{C}_c^\infty(\mathbb{R}^d)$. Then the function
$$\Phi:(t,(q,p))\in[0,T)\times\mathbb{R}^{2d}\mapsto\int_{\mathbb{R}^{2d}} \widehat{\mathrm{p}}_{T-t}((q,p),y) \varphi(y) \mathrm{d}y ,$$
is in $\mathcal{C}^\infty([0,T)\times\mathbb{R}^{2d})$ by the Lebesgue
differentiation theorem. Besides, it satisfies
$\partial_t\Phi+\widehat{\mathcal{L}}\Phi=0$ using the backward
Kolmogorov equation satisfied by $\widehat{\mathrm{p}}_{t}(x,y)$ (see Proposition~\ref{prop:kolmo-langevin}).
As a result, the Itô formula ensures that for all $x\in\mathbb{R}^{2d}$, $t\in[0,T)$,
\begin{equation}\label{Ito formula}
    \Phi(t,X_t^x)=\underbrace{\Phi(0,x)}_{\int_{\mathbb{R}^{2d}} \widehat{\mathrm{p}}_{T}(x,y) \varphi(y) \mathrm{d}y}+\int_0^t \underbrace{(\mathcal{L}-\widehat{\mathcal{L}})\Phi(u,X_u^x)}_{F(q_{u}^x)\cdot\nabla_p \Phi(u,X^x_u)} \mathrm{d}u+\sigma\int_0^t  \nabla_p \Phi(u,X^x_u)\cdot\mathrm{d}B_u .  
\end{equation}
Besides, one has for $u\in[0,t]$, $(q,p)\in\mathbb{R}^{2d}$,
$$\nabla_p\Phi(u,(q,p))=\int_{\mathbb{R}^{2d}} \nabla_p \widehat{\mathrm{p}}_{T-u}((q,p),y) \varphi(y) \mathrm{d}y .$$
Let $\alpha\in(0,1)$. It follows from Lemma \ref{prop densite} that
there exist $C_1,C_2>0$ depending only on $\alpha,\sigma,\gamma, T$ such
that for all $t\in[0,T)$, $u\in[0,t]$, $(q,p)$ and $y$ in
$\mathbb{R}^{2d}$, 
\begin{equation}\label{majoration gradient densite}
    \vert\nabla_{p}\widehat{\mathrm{p}}_{T-u}((q,p),y)\vert\leq \frac{C_1}{\sqrt{T-u}} \widehat{\mathrm{p}}^{(\alpha)}_{T-u}((q,p),y)\leq \frac{C_2}{(T-t)^{2d+1/2} \phi(\gamma(T-u))^{d/2}} . 
\end{equation}
Therefore $\nabla_p \Phi$ is bounded on $[0,t]\times\mathbb{R}^{2d}$
and the integrand of the last term in the right-hand side of the
equality \eqref{Ito formula} is bounded, which implies that its expectation vanishes. 

 Furthermore, using the Fubini-Tonnelli theorem along with \eqref{majoration gradient densite}, one gets
\begin{equation}\label{integrabilite en T}
    \mathbb{E}\left(\int_{\mathbb{R}^{2d}} \left\vert F(q_u^x)\right\vert  \left\vert\nabla_p\widehat{\mathrm{p}}_{T-u}(X_u^x,y)\right\vert \vert\varphi(y)\vert \mathrm{d}y\right)\leq\frac{C_1 \Vert\varphi\Vert_\infty \Vert F\Vert_\infty}{\sqrt{T-u}}  \mathbb{E}\bigg(\underbrace{\int_{\mathbb{R}^{2d}}\widehat{\mathrm{p}}^{(\alpha)}_{T-u}(X_u^x,y) \mathrm{d}y}_{=1}\bigg) , 
\end{equation}
which is integrable on $[0,T)$.
Consequently,
\begin{equation}\label{egalite 1 mild}
    \mathbb{E}\left(\int_{\mathbb{R}^{2d}} \widehat{\mathrm{p}}_{T-t}(X_t^x,y) \varphi(y) \mathrm{d}y\right)=\int_{\mathbb{R}^{2d}} \widehat{\mathrm{p}}_{T}(x,y) \varphi(y) \mathrm{d}y+\int_0^t\int_{\mathbb{R}^{2d}} \mathbb{E}\left(F(q_u^x)\cdot \nabla_p\widehat{\mathrm{p}}_{T-u}(X_u^x,y)\right) \varphi(y) \mathrm{d}y \mathrm{d}u . 
\end{equation}
It follows from Lemma \ref{prop densite} that $\int_{\mathbb{R}^{2d}}
\widehat{\mathrm{p}}_{T-t}(X_t^x,y) \varphi(y) \mathrm{d}y$ converges
almost surely to $\varphi(X_T^x)$ when $t$ converges to $T$. By
considering the limit $t \to T$ (using
the dominated convergence theorem in the term in the left-hand side
of~\eqref{egalite 1 mild}), one obtains from \eqref{egalite 1 mild}
and \eqref{integrabilite en T} that for all $x \in \mathbb{R}^{2d}$,
\begin{align*}
    \int_{\mathbb{R}^{2d}} \mathrm{p}_{T}(x,y) \varphi(y) \mathrm{d}y&=\mathbb{E}\left(\varphi(X_T^x)\right)\\
    &=\int_{\mathbb{R}^{2d}} \widehat{\mathrm{p}}_{T}(x,y) \varphi(y) \mathrm{d}y+\int_0^T\int_{\mathbb{R}^{2d}} \mathbb{E}\left(F(q_u^x)\cdot \nabla_p\widehat{\mathrm{p}}_{T-u}(X_u^x,y) \varphi(y)\right) \mathrm{d}y \mathrm{d}u\\
    &=\int_{\mathbb{R}^{2d}} \widehat{\mathrm{p}}_{T}(x,y) \varphi(y) \mathrm{d}y+\int_0^T\int_{\mathbb{R}^{2d}}\int_{\mathbb{R}^{2d}} \mathrm{p}_{u}(x,(q,p)) F(q)\cdot \nabla_p\widehat{\mathrm{p}}_{T-u}((q,p),y) \varphi(y) \mathrm{d}q \mathrm{d}p \mathrm{d}y \mathrm{d}u .
\end{align*}
Since this is satisfied for all $T>0$ and
$\varphi\in\mathcal{C}_c^\infty(\mathbb{R}^d)$, then by continuity of
the transition density, one obtains that for all $t>0$ and $x,y\in\mathbb{R}^{2d}$,
$$\mathrm{p}_{t}(x,y)-\widehat{\mathrm{p}}_{t}(x,y)=\int_0^t\int_{\mathbb{R}^{2d}}\mathrm{p}_{u}(x,(q,p)) F(q)\cdot \nabla_p\widehat{\mathrm{p}}_{t-u}((q,p),y) \mathrm{d}q \mathrm{d}p \mathrm{d}u$$
which is a mild formulation of the Fokker-Planck equation associated with~\eqref{Langevin}.

In order to rewrite this mild formulation, let us define the following kernel $\mathrm{H}$ for $t>0$, $(q,p)$ and $y$ in $\mathbb{R}^{2d}$,
\begin{align*}
\mathrm{H}(t,(q,p),y)&:=F(q)\cdot\nabla_{p}\widehat{\mathrm{p}}_{t}((q,p),y) .
\end{align*}
For $t>0$, $x,y$ in $\mathbb{R}^{2d}$, let us define $\mathrm{p}\otimes \mathrm{H}(t,x,y)$  by
\begin{equation}\label{mild formulation 1}
    (\mathrm{p}\otimes \mathrm{H})(t,x,y)=\int_0^t\int_{\mathbb{R}^{2d}}\mathrm{p}_u(x,z) \mathrm{H}(t-u,z,y) \mathrm{d}z \mathrm{d}u. 
\end{equation} 
The mild formulation can thus be rewritten: for all $t>0$ and $x,y\in\mathbb{R}^{2d}$,
\begin{equation}\label{mild formulation 2}
    \mathrm{p}_t(x,y)-\widehat{\mathrm{p}}_t(x,y)=\mathrm{p}\otimes \mathrm{H}(t,x,y) . 
\end{equation}
We notice that $(\mathrm{p}\otimes \mathrm{H})\otimes
\mathrm{H}=\mathrm{p}\otimes(\mathrm{H}\otimes
  \mathrm{H})$, which allows us to define univocally $\mathrm{H}^{(k)}=\underbrace{\mathrm{H}\otimes\dots\otimes \mathrm{H}}_{k\text{ times}}$. Besides, iterating $r$ times the equality \eqref{mild formulation 2} we get
\begin{equation}\label{mild formulation 3}
    \mathrm{p}_t(x,y)=\widehat{\mathrm{p}}_t(x,y)+\sum_{j=1}^r\widehat{\mathrm{p}}\otimes \mathrm{H}^{(j)}(t,x,y)+\mathrm{p}\otimes \mathrm{H}^{(r+1)}(t,x,y) .  
  \end{equation}
  
\medskip \noindent\textbf{Step 2}. Let us prove that the series
$\sum_{j=1}^\infty\widehat{\mathrm{p}}\otimes \mathrm{H}^{(j)}(t,x,y)$
converges by getting upper bounds on $\widehat{\mathrm{p}}\otimes
\mathrm{H}^{(j)}$ for $j\geq1$. Let $\alpha\in(0,1)$. By Lemma
\ref{prop densite}, there exists $c_\alpha>0$ such that for all $t>0$,
$x,y \in\mathbb{R}^{2d}$, $\vert \mathrm{H}(t,x,y) \vert\leq\Vert F\Vert_\infty \frac{c_\alpha(1+\sqrt{\gamma_- t}) }{\sqrt{\sigma^2 t }} \widehat{\mathrm{p}}^{(\alpha)}_{t }(x,y)$. Therefore, for a fixed $T>0$, for all $t \in (0,T]$ and $x,y\in\mathbb{R}^{2d}$, 
\begin{equation}\label{majoration H}  
\left\vert \mathrm{H}(t,x,y) \right\vert\leq \frac{C_3}{\sqrt{t}}
\widehat{\mathrm{p}}^{(\alpha)}_{t }(x,y) \text{  where  }     C_3:=\Vert F\Vert_\infty \frac{c_\alpha(1+\sqrt{\gamma_- T}) }{\sigma} . 
\end{equation} 
Besides, for $u\in(0,t)$ and $t\in(0,T]$, $x,z,y\in\mathbb{R}^{2d}$,
one has from \eqref{majoration H}, since
$\widehat{\mathrm{p}}_{t}(x,y)\leq\alpha^{- d}
\widehat{\mathrm{p}}^{(\alpha)}_{t}(x,y)$ (from~\eqref{eq:ppalpha}), that
\begin{align*}
\vert \widehat{\mathrm{p}}_u(x,z) \mathrm{H}(t-u,z,y)\vert&\leq \frac{C_3}{\alpha^d}  \widehat{\mathrm{p}}^{(\alpha)}_{u}(x,z) \frac{\widehat{\mathrm{p}}^{(\alpha)}_{t-u}(z,y)}{\sqrt{t-u}} .
\end{align*} 
For $m,n>0$, let $B(m,n):=\int_0^1 u^{m-1} (1-u)^{n-1}
\mathrm{d}u=\frac{\Gamma(m)\Gamma(n)}{\Gamma(m+n)}$,  with $\Gamma$
the Gamma function. Therefore, for $t\in(0,T]$,
$x,y\in\mathbb{R}^{2d}$, by the Chapman-Kolmogorov relation~\eqref{chp kolmogorov},
\begin{align*}
\left\vert (\widehat{\mathrm{p}}\otimes\mathrm{H})(t,x,y)
  \right\vert&=\left\vert\int_0^t\int_{\mathbb{R}^{2d}}\widehat{\mathrm{p}}_u(x,z)
               \mathrm{H}(t-u,z,y) \mathrm{d}z \mathrm{d}u\right\vert \\
&\leq
                                                                       \frac{C_3}{\alpha^d} \widehat{\mathrm{p}}^{(\alpha)}_{t}(x,y) t^{\frac{1}{2}} B\left(1,\frac{1}{2}\right) .
                                                                       \end{align*}
By induction, for all $j\geq1$,
\begin{equation}\label{maj p H_j}
    \vert \widehat{\mathrm{p}}\otimes \mathrm{H}^{(j)}(t,x,y)\vert\leq \frac{C_3^j}{\alpha^d} \widehat{\mathrm{p}}^{(\alpha)}_{t}(x,y) t^{\frac{j}{2}} \prod_{l=1}^{j}B\left(\frac{l+1}{2},\frac{1}{2}\right) . 
\end{equation}
Consequently, since $\prod_{l=1}^{j}B\left(\frac{l+1}{2},\frac{1}{2}\right)=\frac{\sqrt{\pi}^j}{\Gamma(\frac{j+2}{2})}$ it is easy to see from the Stirling formula that for all $t>0$, $x,y\in\mathbb{R}^{2d}$, the series $\sum_{j=1}^\infty\widehat{\mathrm{p}}\otimes \mathrm{H}^{(j)}(t,x,y)$ converges absolutely.

\medskip \noindent \textbf{Step 3.} Let us now prove that
$\mathrm{p}\otimes
\mathrm{H}^{(r+1)}(t,x,y)\underset{r\rightarrow\infty}{\longrightarrow}0$
for all $t \in (0,T]$, $x,y\in\mathbb{R}^{2d}$. By \eqref{majoration H} and the Chapman-Kolmogorov relation~\eqref{chp kolmogorov}, we have for all $t \in (0,T]$, $x,y\in\mathbb{R}^{2d}$ $$\left\vert \mathrm{H}\otimes \mathrm{H}(t,x,y) \right\vert\leq C_3^2 B\left(\frac{1}{2},\frac{1}{2}\right)  \widehat{\mathrm{p}}^{(\alpha)}_{t}(x,y) .$$
By induction, for all $j\geq2$, for all $t \in (0,T]$,
$x,y\in\mathbb{R}^{2d}$, 
\begin{equation}\label{majoration H^j}
    \vert \mathrm{H}^{(j)}(t,x,y)\vert\leq C_{3}^j \widehat{\mathrm{p}}^{(\alpha)}_{t}(x,y) t^{\frac{j}{2}-1} \frac{\sqrt{\pi}^{j-1}}{\Gamma(\frac{j+1}{2})} . 
\end{equation}
As a consequence, 
\begin{align*}
\vert \mathrm{p}\otimes \mathrm{H}^{(r+1)}(t,x,y)\vert&\leq C_3^{r+1} \frac{\sqrt{\pi}^{r+1}}{\Gamma(\frac{r+3}{2})} \int_0^t\int_{\mathbb{R}^{2d}} \mathrm{p}_u(x,z)  \widehat{\mathrm{p}}^{(\alpha)}_{t-u}(z,y) (t-u)^{\frac{r-1}{2}} \mathrm{d}z \mathrm{d}u .
\end{align*}
From the expression of $\widehat{\mathrm{p}}^{(\alpha)}_{t-u}(z,y)$ it follows that there exists $C_4>0$ depending only on $\alpha,\gamma,\sigma$ such that
$$\widehat{\mathrm{p}}^{(\alpha)}_{t-u}(z,y) (t-u)^{\frac{r-1}{2}}\leq \frac{C_4 (t-u)^{\frac{r-1}{2}}}{(t-u)^{2d} \phi(\gamma (t-u))^{d/2}} .$$
Let us choose $r\geq4d+1$. Since $\phi$ is positive and continuous
then it is bounded from below for $s\in(-\vert\gamma\vert
T,\vert\gamma\vert T)$, therefore there exists $C_5>0$ depending on
$\alpha,\gamma,\sigma, d$ and $T$ such that for all $t\in(0,T)$,
$r\geq4d+1$, $u\in(0,t)$ and $z,y\in\mathbb{R}^{2d}$, 
$$\widehat{\mathrm{p}}^{(\alpha)}_{t-u}(z,y) (t-u)^{\frac{r-1}{2}}\leq C_5. $$
Therefore, for all $t \in (0,T]$, $x,y\in\mathbb{R}^{2d}$,
$$\vert  \mathrm{p}\otimes \mathrm{H}^{(r+1)}(t,x,y)\vert\leq C_3^{r+1} C_5 T B\left(\frac{1}{2},\frac{1}{2}\right) B\left(1,\frac{1}{2}\right)\cdots B\left(\frac{r}{2},\frac{1}{2}\right)\underset{r\longrightarrow \infty}{{\longrightarrow}}0 . $$ 

\medskip \noindent\textbf{Step 4}. As a result, using the results of Step 2 and 3 in the equality \eqref{mild formulation 3} we get for all $t \in (0,T]$, $x,y\in\mathbb{R}^{2d}$,
$$\mathrm{p}_t(x,y)-\widehat{\mathrm{p}}_t(x,y)=\sum_{j=1}^\infty \widehat{\mathrm{p}}\otimes \mathrm{H}^{(j)}(t,x,y) .$$
From the formula defining $C_3$ in~\eqref{majoration H} and \eqref{maj p H_j}, the inequality \eqref{borne densité} follows.
\end{proof}

\begin{remark}
In view of the proof of Theorem~\ref{borne densite thm}, it is clear
that~\eqref{borne densité} holds for $F$ only bounded (dropping the
assumptions that $F$ is $\mathcal{C}^\infty$ and globally Lipschitz continuous in
Assumption~\ref{hyp F}), as soon as there exists a weak solution to~\eqref{Langevin}. Gaussian upper bounds for the Langevin process thus hold under slightly more general assumptions than those originally stated in~\cite{Menozzi}.
\end{remark}

\subsection{Existence of a smooth transition density for the absorbed Langevin process}\label{smooth density subsection}

\begin{proposition}[Existence of a measurable transition density] \label{existence densite} Under Assumption \ref{hyp F1}, there exists a measurable function $$(t,x,y) \in \mathbb{R}_+^{*}\times D\times D\mapsto \mathrm{p}_t^D(x,y)$$ such that for all $t>0$ and $x \in D$, the kernel $\mathrm{P}^D_t(x,\cdot)$, defined in \eqref{kernel D}, has the density $\mathrm{p}^D_t(x,\cdot)$ with respect to the Lebesgue measure on $D$.  
\end{proposition}

The proof of the proposition above is detailed in Appendix~\ref{pf:arg densite}. Let us now prove that this transition density $\mathrm{p}_t^D(x,y)$ is smooth on $\mathbb{R}_+^{*}\times D\times D$. This will be managed by first showing that it is a distributional solution of the backward and forward Kolmogorov equations. The smoothness of $\mathrm{p}^D_t$ will then follow from the hypoellipticity of the differential operators $\partial_t-\mathcal{L}$ and $\partial_t-\mathcal{L}^*$, see Definition~\ref{def:hypoell}. This scheme of proof is inspired from \cite[Section 3.5]{McK}. Notice that Proposition~\ref{existence densite} only defines the transition density on $\mathbb{R}_+^*\times D \times D$. The extension to a continuous function on $\mathbb{R}_+^*\times \overline{D} \times \overline{D}$ will be done in Section~\ref{Section Adjoint process and compactness} (see Theorem~\ref{boundary property density}).
\begin{proposition}[Kolmogorov equations] \label{regularité densité} Under Assumption \ref{hyp F1},
the transition density $(t,x,y)\mapsto \mathrm{p}_t^D(x,y)$ is a $\mathcal{C}^\infty(\mathbb{R}_+^{*}\times D\times D)$ function. Besides it satisfies the backward and forward Kolmogorov equations:
\begin{enumerate}[label=(\roman*),ref=\roman*]
    \item $(t,x)\mapsto \mathrm{p}_t^D(x,y)$ is a solution of $\partial_t \mathrm{p}^D=\mathcal{L}_x\mathrm{p}^D$ on $\mathbb{R}_+^{*}\times D$,
    \item $(t,y)\mapsto \mathrm{p}_t^D(x,y)$ is a solution of $\partial_t \mathrm{p}^D=\mathcal{L}_y^*\mathrm{p}^D$ on $\mathbb{R}_+^{*}\times D$.
\end{enumerate} 
\end{proposition}

\begin{proof}

Let $\Phi\in\mathcal{C}_c^\infty(\mathbb{R}_+^{*}\times D)$. Notice
that $\Phi$ can be extended by zero to a
$\mathcal{C}^\infty(\mathbb{R}_+\times \overline{D})$
function. 
 Let $(X^x_t=(q^x_t,p^x_t))_{t\geq0}$ be the process satisfying \eqref{Langevin}. Using Itô's formula, one gets for all $x\in D$ and $t>0$,
$$ \Phi(t,X^x_t)=\underbrace{\Phi(0,x)}_{=0}+\int_0^t\left[\partial_s\Phi(s,X^x_s)+\mathcal{L}\Phi(s,X^x_s)\right] \mathrm{d}s+\sigma  \int_0^t\nabla_p\Phi(s,X^x_s)\cdot \mathrm{d}B_s.$$
Thus, 
$$ \Phi(\tau^x_\partial \land t,X^x_{\tau^x_\partial \land t})=\int_0^t\mathbb{1}_{\tau^x_\partial>s} \left[\partial_s\Phi(s,X^x_s)+\mathcal{L}\Phi(s,X^x_s)\right]\mathrm{d}s+ \sigma \int_0^{\tau^x_\partial \land t}\underbrace{\nabla_p\Phi(s,X^x_s)}_{\text{bounded}}\cdot  \mathrm{d}B_s .$$

\noindent As a result, the stochastic integral in the right-hand side is a martingale. Taking the expectation, we get
$$ \mathbb{E}\left[ \Phi(\tau^x_\partial \land t,X^x_{\tau^x_\partial \land t}) \right]=\int_0^t \mathbb{E}\left[ \mathbb{1}_{\tau^x_\partial>s} \left(\partial_s\Phi(s,X^x_s)+\mathcal{L}\Phi(s,X^x_s)\right)  \right]\mathrm{d}s .$$

Since $X^x_{\tau^x_\partial}\in\partial D$ and $\Phi$ vanishes on $\mathbb{R}_+\times\partial D$,
\[\begin{aligned}
\mathbb{E}\left[ \Phi(\tau^x_\partial \land t,X^x_{\tau^x_\partial \land t}) \right]&=\mathbb{E}[ \mathbb{1}_{\tau^x_\partial>t} \Phi(t,X^x_{t})+ \mathbb{1}_{\tau^x_\partial\leq t} \underbrace{\Phi(\tau^x_\partial,X^x_{\tau^x_\partial})}_{=0} ]=\mathbb{E}\left[ \mathbb{1}_{\tau^x_\partial>t} \Phi(t,X^x_{t}) \right] .
\end{aligned}\]
Thus
$$ \mathbb{E}\left[ \mathbb{1}_{\tau^x_\partial>t} \Phi(t,X^x_{t}) \right]=\int_0^t \mathbb{E}\left[ \mathbb{1}_{\tau^x_\partial>s} \left(\partial_s\Phi(s,X^x_s)+\mathcal{L}\Phi(s,X^x_s)\right)  \right]\mathrm{d}s .$$
For $t$ large enough since $\Phi$ has a compact support on
$\mathbb{R}_+^*\times D$, the left-hand side in the equality above is zero. 
Therefore,
$$ \iint_{\mathbb{R}_+^{*}\times D} \left(\partial_s\Phi(s,y)+\mathcal{L}\Phi(s,y)\right) \mathrm{p}_s^D(x,y)  \mathrm{d}s \mathrm{d}y=0 .$$
As a result for all $x\in D$, $$(t,y)\in\mathbb{R}_+^{*}\times D \mapsto \mathrm{p}_t^D(x,y)$$ is a distributional solution of $$\partial_t\mathrm{p}_t^D=\mathcal{L}^*_y\mathrm{p}_t^D$$ on $\mathbb{R}_+^{*}\times D$. Since the operator $\partial_t-\mathcal{L}^*$ is hypoelliptic one has that 
\begin{equation}\label{reg t,y}
    (t,y)\in\mathbb{R}_+^{*}\times D \mapsto \mathrm{p}_t^D(x,y)  \in\mathcal{C}^\infty(\mathbb{R}_+^{*}\times D), 
\end{equation}
which proves the forward Kolmogorov equation.

We now address the backward Kolmogorov equation. Let $\Phi_1\in\mathcal{C}_c^\infty(\mathbb{R}_+^{*}\times D)$, $\Phi_2\in\mathcal{C}_c^\infty(D)$, and let us define the function $\Phi$ as follows: for all $(t,x,y)\in\mathbb{R}_+^{*}\times D\times D$,
$$\Phi(t,x,y)=\Phi_1(t,x) \Phi_2(y) .$$
Let us compute the following integral
\[\begin{aligned}
I&=\iiint_{\mathbb{R}_+^{*}\times D\times D} \mathrm{p}_t^D(x,y)\left(\partial_t\Phi(t,x,y)+\mathcal{L}_x^*\Phi(t,x,y)\right)\mathrm{d}t \mathrm{d}x \mathrm{d}y\\
&=\iint_{\mathbb{R}_+^{*}\times D}\left(\partial_t\Phi_1(t,x)+\mathcal{L}_x^*\Phi_1(t,x)\right)\left(\int_D \mathrm{p}_t^D(x,y) \Phi_2(y) \mathrm{d}y\right)\mathrm{d}t \mathrm{d}x.
\end{aligned}\]
On the one hand, since $(t,x)\mapsto \int_D \mathrm{p}_t^D(x,y)\Phi_2(y) \mathrm{d}y=\mathbb{E}[ \mathbb{1}_{\tau^x_\partial>t} \Phi_2(X^x_t) ]$ is a solution of $\partial_t u=\mathcal{L}_xu$ by Theorem~\ref{Solution PDE}, then $I=0$. On the other hand, it follows from Fubini's theorem that 
$$\int_D \Phi_2(y) \left(\iint_{\mathbb{R}_+^{*}\times D}\mathrm{p}_t^D(x,y)\left(\partial_t\Phi_1(t,x)+\mathcal{L}_x^*\Phi_1(t,x)\right) \mathrm{d}t \mathrm{d}x\right)\mathrm{d}y=0 . $$ Since $y\in D\mapsto\iint_{\mathbb{R}_+^{*}\times D}\mathrm{p}_t^D(x,y)\left(\partial_t\Phi_1(t,x)+\mathcal{L}_x^*\Phi_1(t,x)\right) \mathrm{d}t \mathrm{d}x\in \mathrm{L}_1^{\mathrm{loc}}(D)$ (since $\Phi_1\in\mathcal{C}_c^\infty(\mathbb{R}_+^{*}\times D)$), this ensures that for almost every $y\in D$,
$$\iint_{\mathbb{R}_+^{*}\times D}\mathrm{p}_t^D(x,y)\left(\partial_t\Phi_1(t,x)+\mathcal{L}_x^*\Phi_1(t,x)\right) \mathrm{d}t \mathrm{d}x=0 .$$ Using the continuity of $y\in D\mapsto\mathrm{p}_t^D(x,y)$ from \eqref{reg t,y}, the equality above remains true for all $y\in D$. Thus, for all $y \in D$, $(t,x) \mapsto \mathrm{p}_t^D(x,y)$ is a distributional solution of the backward Kolmogorov equation
$$\partial_t \mathrm{p}_t^D = \mathcal{L}_x \mathrm{p}_t^D$$ on $\mathbb{R}_+^* \times D$.

Consequently, the hypoellipticity of $\partial_t-\mathcal{L}$ on the open set $\mathbb{R}_+^{*}\times D$ ensures that for all $y\in D$
$$(t,x)\in \mathbb{R}_+^{*}\times D\mapsto \mathrm{p}_t^D(x,y)   \in \mathcal{C}^\infty(\mathbb{R}_+^{*}\times D) .$$
Therefore, using \eqref{reg t,y}, it follows that
$$ (t,x,y)\in \mathbb{R}_+^{*}\times D\times D\mapsto \mathrm{p}_t^D(x,y)   \in \mathcal{C}^\infty(\mathbb{R}_+^{*}\times D\times D),$$
which concludes the proof of Proposition \ref{regularité densité}.
\end{proof}

Corollary~\ref{Rq densite estimation} shows that the Gaussian upper bound on the transition density $\mathrm{p}_t$ immediately transfers to the transition density $\mathrm{p}^D_t$. In fact, in the next lemma, we show that the latter also satisfies a mild formulation of the form
\begin{equation}\label{eq:mildptD}
  \mathrm{p}^D_t - \widehat{\mathrm{p}}_t = \mathrm{p}^D \otimes \mathrm{H}^D,
\end{equation}
for some kernel $\mathrm{H}^D$, and compute estimates on this kernel to obtain an asymptotic expansion of $\mathrm{p}^D_t$ in compact sets of $D$. This lemma will be useful in Section~\ref{Section Adjoint process and compactness}.
\begin{lemma}[Local asymptotic expansion around $t=0$]\label{lemma mild inequality}
Under Assumption \ref{hyp F1}, the density $\mathrm{p}^D_t$ is such that for all compact sets $K\subset D$, $T>0$ and $\alpha\in(0,1)$, there exists $C>0$ such that for all $x,y\in K$ and $t\in (0,T)$,
\begin{equation}\label{mild inequality}
\left\vert\mathrm{p}^D_t(x,y)-\widehat{\mathrm{p}}_t(x,y)\right\vert\leq  C \sqrt{t}\,\widehat{\mathrm{p}}^{(\alpha)}_{t}(x,y) .  
\end{equation} 
\end{lemma}

\begin{proof}
 Since the density $\mathrm{p}^D_t$ only
depends on the values of $F$ in $\mathcal{O}$ (see
Remark~\ref{solution dans D}), we can assume that $F$ satisfies Assumption~\ref{hyp F} for the sake of simplicity. The first step of the proof consists in establishing the mild formulation~\eqref{eq:mildptD}. In contrast with the proof of Theorem \ref{borne densite thm}, where a mild formulation of the forward Kolmogorov equation satisfied by $\mathrm{p}_t$ is established, the absorbing boundary condition makes the use of the Itô formula inappropriate. We adopt a different approach, inspired from~\cite[Proposition~2.2]{Menozzi}.

Let $T>0$ and $K\subset D$ be a compact set. Let $x=(q,p)$, $y=(q',p')\in K$ and $t\in(0,T]$. Let us define $\varphi\in\mathcal{C}^\infty_c(D)$ such that 
\begin{equation}
    0\leq\varphi(z)\leq1 \text{ for all }z\in D,\text{ and
    }\varphi(z)=1\text{ for all }z\in K .
\end{equation}
Let us define the function $h_t$ as follows:
$$h_t:u\in(0,t)\mapsto\int_{D}\mathrm{p}^D_u(x,z) \widehat{\mathrm{p}}_{t-u}(z,y) \varphi(z) \mathrm{d}z.$$

Let us identify the limits of $h_t(u)$ when $u\rightarrow0$ and
 $u\rightarrow t$. First we have that 
$$h_t(u)=\mathbb{E}\left[ \widehat{\mathrm{p}}_{t-u}(X^x_u,y) \varphi(X_u^x) \mathbb{1}_{\tau_\partial^x>u} \right]\underset{u\rightarrow 0}{{\longrightarrow}}  \widehat{\mathrm{p}}_{t}(x,y)\varphi(x)=\widehat{\mathrm{p}}_{t}(x,y),$$
by the dominated convergence theorem using Lemma \ref{cv indicatrices
  lemma} and the continuity and boundedness of
$\widehat{\mathrm{p}}_{s}(\cdot,y)\varphi(\cdot)$ when $s$ is close to
$t$. Second, it follows from the convergence \eqref{chgt de var + cv}
in Lemma \ref{prop densite} and the boundedness and continuity of the
product $\mathrm{p}^D_u(x,\cdot)\varphi(\cdot)$ when $u$ is close to
$t$ that (remember that $y \in K$)
$$h_t(u)\underset{u\rightarrow t}{{\longrightarrow}} \mathrm{p}^D_{t}(x,y)\varphi(y)=\mathrm{p}^D_{t}(x,y).$$ 

Therefore, using the fact that $h_t\in\mathcal{C}^1((0,t))$, we have
that  
\begin{align*}
\mathrm{p}^D_t(x,y)-\widehat{\mathrm{p}}_t(x,y)&=\int_0^t\frac{\mathrm{d}h_t}{\mathrm{d} u}(u) \mathrm{d}u\\
&=\int_0^t\int_{D}\left(\partial_u [\mathrm{p}^D_u(x,z)]
                                                                                                               \widehat{\mathrm{p}}_{t-u}(z,y)
                                                                                                               +\mathrm{p}^D_u(x,z)\partial_u
                                                                                                               [\widehat{\mathrm{p}}_{t-u}(z,y)]\right)
                                                                                                               \varphi(z) \mathrm{d}z \mathrm{d}u,
\end{align*}
by the Lebesgue differentiation theorem as $\mathrm{p}^D_t(x,y)$, $\widehat{\mathrm{p}}_t(x,y)$ are smooth on $\mathbb{R}_+^*\times D\times D$ and $\varphi\in\mathcal{C}^\infty_c(D)$. 

Recall that we denote by $\widehat{\mathcal{L}} = \mathcal{L}_{0,\gamma,\sigma}$ the infinitesimal generator of $(\widehat{X}^x_t)_{t \geq 0}$. Since $\partial_t  \mathrm{p}_t^D(x,y) = \mathcal L^*_y
\mathrm{p}_t^D(x,y)$ (see Theorem~\ref{thm density intro}) and $\partial_t
\widehat{\mathrm{p}}_t (x,y) = \widehat{\mathcal{L}}_x
\widehat{\mathrm{p}}_t (x,y) $ (see
Proposition~\ref{prop:kolmo-langevin}), one has (using the notation $z=(q'',p'') \in \mathbb{R}^{2d}$),  
\begin{align*}
\mathrm{p}^D_t(x,y)-\widehat{\mathrm{p}}_t(x,y)&=\int_0^t\int_{D}\left(\mathcal{L}^*_z \mathrm{p}^D_u(x,z) \widehat{\mathrm{p}}_{t-u}(z,y)-\mathrm{p}^D_u(x,z) \widehat{\mathcal{L}}_z \widehat{\mathrm{p}}_{t-u}(z,y)\right) \varphi(z) \mathrm{d}z \mathrm{d}u\\
&=\int_0^t\int_{D}\mathrm{p}^D_u(x,z) \left(\mathcal{L}_z\left(\widehat{\mathrm{p}}_{t-u}(z,y) \varphi(z)\right)-\widehat{\mathcal{L}}_z (\widehat{\mathrm{p}}_{t-u}(z,y)) \varphi(z)\right)  \mathrm{d}z \mathrm{d}u\\
&=\int_0^t\int_{D}\mathrm{p}^D_u(x,z)
                                                                                                                                                                                                                          \left[\left(\mathcal{L}_z-\widehat{\mathcal{L}}_z\right) (\widehat{\mathrm{p}}_{t-u}(z,y)) \varphi(z)+\sigma^2\nabla_{p''}\widehat{\mathrm{p}}_{t-u}(z,y)\cdot\nabla_{p''}\varphi(z) \right] \mathrm{d}z \mathrm{d}u\\
&\quad +\int_0^t\int_{D}\mathrm{p}^D_u(x,z) \widehat{\mathrm{p}}_{t-u}(z,y) \mathcal{L}_z\varphi(z)  \mathrm{d}z \mathrm{d}u,
\end{align*}
which is the claimed mild formulation~\eqref{eq:mildptD}. We have
$\mathcal{L}_z-\widehat{\mathcal{L}}_z=F(q'')\cdot\nabla_{p''}$. Furthermore,
$\varphi\in\mathcal{C}^\infty_c(D)$, therefore its gradient is bounded
on $D$ and $\mathcal{L}\varphi\in\mathcal{C}^\infty_c(D)$. Besides, it
follows from~\eqref{maj gradient p} in Lemma \ref{prop densite}  that for any $\alpha\in(0,1)$
there exists $C_1>0$ such that for all $t\in(0,T]$, $u\in[0,t)$ and $(q'',p''),y\in\mathbb{R}^{2d}$,
$$\vert\nabla_{p''}\widehat{\mathrm{p}}_{t-u}((q'',p''),y)\vert\leq
\frac{C_1}{\sqrt{t-u}}
\widehat{\mathrm{p}}^{(\alpha)}_{t-u}((q'',p''),y) .$$
In addition,  from~\eqref{eq:ppalpha}, $\widehat{\mathrm{p}}_t(x,y)\leq \alpha^{-d} \widehat{\mathrm{p}}^{(\alpha)}_{t}(x,y)$ for all $t>0$, $x,y\in\mathbb{R}^{2d}$. Consequently, under Assumption \ref{hyp F}, there exists a constant $C_K>0$ such that
\begin{align*}
\left\vert\mathrm{p}^D_t(x,y)-\widehat{\mathrm{p}}_t(x,y)\right\vert&\leq C_K \int_0^t\int_{D}\mathrm{p}^D_u(x,z) \frac{\widehat{\mathrm{p}}^{(\alpha)}_{t-u}(z,y)}{\sqrt{t-u}} \mathrm{d}z \mathrm{d}u .
\end{align*}
Furthermore, by Corollary~\ref{Rq densite estimation} there exists $C_2>0$ such that for all $u\in(0,t)$, $t\in(0,T)$, $\mathrm{p}^D_u(x,y)\leq\mathrm{p}_u(x,y)\leq C_2 \widehat{\mathrm{p}}^{(\alpha)}_{u}(x,y)$. Hence the existence of $C'_K>0$ such that for all $t\in(0,T)$ and $x,y\in K$,
\begin{align*}
\left\vert\mathrm{p}^D_t(x,y)-\widehat{\mathrm{p}}_t(x,y)\right\vert&\leq C'_K \int_0^t\int_{D}\widehat{\mathrm{p}}^{(\alpha)}_{u}(x,z) \frac{\widehat{\mathrm{p}}^{(\alpha)}_{t-u}(z,y)}{\sqrt{t-u}} \mathrm{d}z \mathrm{d}u\\
&\leq C'_K \sqrt{t} \widehat{\mathrm{p}}^{(\alpha)}_{t}(x,y) \int_0^1\frac{\mathrm{d}s}{\sqrt{1-s}} ,
\end{align*}
since $\widehat{\mathrm{p}}^{(\alpha)}_t$ satisfies the Chapman-Kolmogorov relation \eqref{chp kolmogorov} in Lemma \ref{prop densite}. This concludes the proof of~\eqref{mild inequality}.
\end{proof}

\subsection{Boundary behavior of the transition density}\label{sec:boundary}
  The purpose of this subsection is to study the behavior of
$\mathrm{p}^D_t(x,y)$ at the boundaries
$(t,x)\in\mathbb{R}_+^*\times(\Gamma^+\cup\Gamma^0)$ and
$(t,y)\in\mathbb{R}_+^*\times\Gamma^-$ (see Proposition
\ref{comportement densite} below). This result will be useful for the proof of Theorem \ref{duality thm}, which will then allow to complete the proof of Theorem~\ref{thm density intro}.

\begin{proposition}[Boundary limits]\label{comportement densite} Let Assumptions \ref{hyp O} and \ref{hyp F1} hold. Let $t_0>0$, $x_0\in \Gamma^+\cup\Gamma^0$ and $y_0\in\Gamma^-$. Let $(t_n,x_n,y_n)_{n\geq1}$ be a sequence of points in $\mathbb{R}_+^*\times D\times D$ converging towards $(t_0,x_0,y_0)$, then one has the following convergences: 
\begin{enumerate}[label=(\roman*),ref=\roman*]
    \item\label{it:comportement-densite:1} For all $y\in D$, $\mathrm{p}^D_{t_n}(x_n,y)\underset{n\rightarrow \infty}{{\longrightarrow}}0$.
    \item\label{it:comportement-densite:2} For all $x\in D$, $\mathrm{p}^D_{t_n}(x,y_n)\underset{n\rightarrow \infty}{{\longrightarrow}}0$.
\end{enumerate} 
\end{proposition}

The proof of this proposition relies partly on the following lemma which is proven in Appendix~\ref{pf:arg densite}.
\begin{lemma}\label{lemma:majoration densite}
Let $y_0 \in \mathbb{R}^{2d}$, $M>0$ and $\alpha \in (0,1)$. There exist $C_0>0$, $\mu>0,\delta_0>0$ such that for all $s \in (0,\delta_0]$, $(q',p') \in \mathrm{B}(y_0,M/6)$ and $(q,p) \in \mathbb{R}^{2d}$ satisfying $|p-p'| \geq M/3$,
  \begin{equation}\label{ineq lemma maj densite}
    \widehat{\mathrm{p}}^{(\alpha)}_s((q,p),(q',p')) \leq C_0\exp(-\mu/s).
  \end{equation}
\end{lemma}

\begin{proof}[Proof of Proposition \ref{comportement densite}]
 Since the density $\mathrm{p}^D_t$ only
depends on the values of $F$ in $\mathcal{O}$ (see
Remark~\ref{solution dans D}), we can assume that $F$ satisfies Assumption~\ref{hyp F} for the sake of simplicity.

Both proofs of~\eqref{it:comportement-densite:1} and~\eqref{it:comportement-densite:2} rely on the elementary remark that, for any $t>0$ and $x \in D$, since the function $y \mapsto \mathrm{p}^D_t(x,y)$ is continuous on $D$, we have for any $y \in D$,
\begin{equation}\label{cv densite 2}
  \mathrm{p}^D_t(x,y) = \lim_{h \to 0} \frac{\mathrm{P}^D_t(x,D \cap \mathrm{B}(y,h))}{|\mathrm{B}(y,h)|}.
\end{equation}
Notice that, here and in the sequel, we take the intersection of $\mathrm{B}(y,h)$ with $D$ because $\mathrm{P}^D_t(x,\cdot)$ is defined as a measure on $\mathcal{B}(D)$.

\medskip\noindent\textbf{Proof of~\eqref{it:comportement-densite:1}.} Let $t_0>0$, $x_0\in\Gamma^+\cup\Gamma^0$. Let $(t_n,x_n)_{n\geq1}$ be a sequence of points in $\mathbb{R}_+^*\times D$ converging towards $(t_0,x_0)$. Let $N \geq 1$ be such that, for any $n \geq N$, $t_0/2 \leq t_n \leq 3t_0/2$. For any $n \geq N$, $h>0$ and $y \in D$, the Markov property shows that
\begin{equation*}
  \mathrm{P}^D_{t_n/2}(x_n, D \cap \mathrm{B}(y,h)) = \mathbb{E}\left[\mathbb{1}_{\tau^{x_n}_\partial > t_n/2} \mathrm{P}^D_{t_n/2}(X^{x_n}_{t_n/2}, D \cap \mathrm{B}(y,h))\right].
\end{equation*}
Besides, by Corollary~\ref{Rq densite estimation}, there exists a constant $C \geq 0$ which depends on $t_0$ such that for any $n \geq N$, the transition density $\mathrm{p}^D_{t_n/2}$ is uniformly bounded on $D \times D$ by $C$, therefore
\begin{equation*}
  \frac{\mathrm{P}^D_{t_n/2}(x_n, D \cap \mathrm{B}(y,h))}{|\mathrm{B}(y,h)|} \leq C \mathbb{P}(\tau^{x_n}_\partial>t_n/2).
\end{equation*}
The right-hand side no longer depends on $h$ and vanishes when $n \to +\infty$ by Lemma \ref{cv indicatrices lemma} and Proposition \ref{prop:tau}, therefore by~\eqref{cv densite 2} we get Assertion~\eqref{it:comportement-densite:1}.

\begin{remark}\label{rk:comportement-densite:cvunif}
  The proof shows that the convergence of Assertion~\eqref{it:comportement-densite:1} is actually uniform in $y$, that is to say $\sup_{y \in D} \mathrm{p}^D_{t_n}(x_n,y) \underset{n\rightarrow \infty}{{\longrightarrow}}0$.
\end{remark}

\medskip\noindent\textbf{Proof of~\eqref{it:comportement-densite:2}.} The proof of~\eqref{it:comportement-densite:2} needs more work. Let $x\in D$, $t_0>0$ and $y_0=(q_0,p_0)\in\Gamma^-$. Let $(t_n,y_n)_{n\geq1}$, with $y_n:=(q_n,p_n)$, be a sequence of points in $\mathbb{R}_+^*\times D$ converging towards $(t_0,y_0)$. In order to prove the convergence $\mathrm{p}^D_{t_n}(x,y_n)\underset{n\rightarrow \infty}{{\longrightarrow}}0$, it is enough by \eqref{cv densite 2} to prove the following double limit 
\begin{equation}\label{lim n limsup h}
    \lim_{n\rightarrow\infty}\lim_{h\rightarrow 0}\frac{\mathrm{P}^D_{t_n}(x,D \cap \mathrm{B}(y_n,h))}{\vert \mathrm{B}(y_n,h)\vert}=0. 
\end{equation}
 Let us define for $0\leq r\leq t$ the following modulus of continuity $$Z^x_{r,t}:=\sup_{r\leq s\leq t}\vert p^x_s-p^x_t\vert.$$
 For two
   constants $\delta  \in (0,t_0/2]$ and $M>0$ to be fixed later on,
   let us rewrite the numerator in~\eqref{lim n limsup h} as follows:
   for $n$ sufficiently large so that $t_n \geq t_0/2$ (and thus $t_n -
   \delta \geq 0$),
\begin{align*}
\mathrm{P}^D_{t_n}(x, D \cap \mathrm{B}(y_n,h))&=\mathbb{P}((q_{t_n}^x,p_{t_n}^x)\in \mathrm{B}(y_n,h), \tau^x_\partial>t_n)\\
&=\mathbb{P}((q_{t_n}^x,p_{t_n}^x)\in \mathrm{B}(y_n,h), Z^x_{t_n-\delta,t_n}\leq M , \tau^x_\partial>t_n)\\
&\quad+\mathbb{P}((q_{t_n}^x,p_{t_n}^x)\in \mathrm{B}(y_n,h), Z^x_{t_n-\delta,t_n}>M , \tau^x_\partial>t_n) . \numberthis \label{terme 2}
\end{align*}

The idea of the proof of \eqref{lim n limsup h} relies on the decomposition in \eqref{terme 2} and is divided into two steps. In \textbf{Step 1}, we consider the probability corresponding to the first term in the right-hand side of the equality \eqref{terme 2}. We show that there is a value of $M$ and a $\delta_1 \in (0, t_0/2]$ such that, for all $\delta \in (0, \delta_1]$, there exist $N_1 \geq 1$ and $h_1 > 0$ such that for any $n \geq N_1$ and $h \leq h_1$, the event $\{(q_{t_n}^x,p_{t_n}^x)\in \mathrm{B}(y_n,h), Z^x_{t_n-\delta,t_n}\leq M , \tau^x_\partial>t_n\}$ has probability $0$. Indeed, for $n$ large and $h$ small the event $\{(q_{t_n}^x,p_{t_n}^x)\in \mathrm{B}(y_n,h)\}$ implies that $(q_{t_n}^x,p_{t_n}^x)$ is "close" to $y_0\in\Gamma
^-$ which is a boundary point with inward velocity. Therefore, using our control on the modulus of  continuity of the velocity, we can prove the existence of a time $s\in(0,t_n)$ such that $(q_s^x,p_s^x)$ is outside of $D$, which contradicts the fact that $\tau^x_\partial>t_n$.

In \textbf{Step 2}, we consider the second term in the right-hand side of the equality \eqref{terme 2}, divided by $\vert \mathrm{B}(y_n,h)\vert$. For the value of $M$ determined in \textbf{Step~1}, we show the existence of $\delta_2 \in (0,t_0/2]$ and $C,\mu>0$ such that, for all $\delta \in (0, \delta_2]$, there exist $N_2 \geq 1$ and $h_2 > 0$ such that, for any $n \geq N_2$ and $h \leq h_2$,
\begin{equation}\label{eq:step2 estimate}
    \frac{\mathbb{P}((q_{t_n}^x,p_{t_n}^x)\in \mathrm{B}(y_n,h),
  Z^x_{t_n-\delta,t_n}>M)}{\vert \mathrm{B}(y_n,h)\vert}\leq
C\mathrm{e}^{-\frac{\mu}{\delta}}.
\end{equation}
As a result, the two steps yield the following inequality
$$\limsup_{n\rightarrow\infty}\limsup_{h\rightarrow0}\frac{\mathrm{P}^D_{t_n}(x, D \cap \mathrm{B}(y_n,h))}{\vert \mathrm{B}(y_n,h)\vert}\leq C\mathrm{e}^{-\frac{\mu}{\delta}}$$
for any $\delta \in (0,\delta_1 \wedge \delta_2]$, then taking $\delta\rightarrow0$ we are able to conclude the proof of \eqref{lim n limsup h}.

\medskip \noindent \textbf{Step 1}. Let us prove here that one can fix $M>0$ and choose
$\delta>0$ small enough such that the first term in the
right-hand side of the equality \eqref{terme 2} vanishes for $n$
sufficiently large and $h$ sufficiently small. By Proposition~\ref{prop:sphere}, for $\eta>0$,
there exists $c_0:=c_0(\eta)>0$ such that if $q\in\mathbb{R}^d$
satisfies $(q-q_0)\cdot n(q_0)\geq\eta \vert
  q-q_0\vert$ and $\vert q-q_0\vert\leq
c_0$ then $q\notin\mathcal{O}$. Now let $M:=-\frac{p_0\cdot n(q_0)}{3} $ which is positive because
  $y_0=(q_0,p_0)\in\Gamma^-$. Let $\eta:=\frac{M}{2M+\vert p_0\vert}$
  and $c_0:=c_0(\eta)$ as defined above.
  
  Let $\delta_1:=\frac{t_0}{2} \land\frac{c_0}{2 M+\vert p_0\vert}.$ We fix $\delta\in(0,\delta_1]$ and define $h_1 := \frac{\delta M}{2(1+\delta)}$. Remembering that $t_n\underset{n\rightarrow \infty}{{\longrightarrow}}t_0$ and $y_n\underset{n\rightarrow \infty}{{\longrightarrow}}y_0$ we can choose $N_1\geq1$ such that for $n\geq N_1$, $$t_n\in[t_0/2,3t_0/2]\quad\text{and}\quad y_n=(q_n,p_n)\in \mathrm{B}\left(y_0,\frac{\delta M}{2(1+\delta)}\right).$$

Let $n \geq N_1$ and $h \in (0,h_1]$. Notice that 
\[\begin{aligned}
\mathbb{P}((q_{t_n}^x,p_{t_n}^x)\in \mathrm{B}(y_n,h), Z^x_{t_n-\delta,t_n}\leq M , \tau^x_\partial>t_n)&\leq\mathbb{P}((q_{t_n}^x,p_{t_n}^x)\in \mathrm{B}(y_n,h), Z^x_{t_n-\delta,t_n}\leq M , q^x_{t_n-\delta}\in\mathcal{O}). 
\end{aligned}\]
Therefore, the first term in \eqref{terme 2} vanishes if one can prove that $q^x_{t_n-\delta}\notin\mathcal{O}$ on the event $\{(q_{t_n}^x,p_{t_n}^x)\in\mathrm{B}(y_n,h), Z^x_{t_n-\delta,t_n}\leq M\}$. 
By \eqref{Langevin}, one has that
$$q_{t_n}^x=q^x_{t_n-\delta}+\int_{t_n-\delta}^{t_n}p^x_s \mathrm{d}s.$$
Therefore,
$$q^x_{t_n-\delta}=q^x_{t_n}-\delta  p^x_{t_n}-\int_{t_n-\delta}^{t_n}(p^x_s-p^x_{t_n}) \mathrm{d}s . $$
Let $$v_{t_n}^x:=q^x_{t_n-\delta}-\left(q_0-\delta p_0-\int_{t_n-\delta}^{t_n}(p^x_s-p^x_{t_n}) \mathrm{d}s\right)=q^x_{t_n}-q_0-\delta \left(p^x_{t_n}- p_0\right). $$
As a result, on the event $\{(q_{t_n}^x,p_{t_n}^x)\in\mathrm{B}(y_n,h)\}$, the triangle inequality ensures that
\begin{align*}
    \vert v_{t_n}^x\vert&\leq\vert q_{t_n}^x-q_n\vert +\vert q_n-q_0\vert+\delta\vert p_{t_n}^x-p_n\vert+\delta\vert p_n-p_0\vert\\
    &\leq\vert X_{t_n}^x-y_n\vert(1+\delta)+\vert y_n-y_0\vert(1+\delta)\\
    &\leq h(1+\delta)+\frac{\delta
      M}{2(1+\delta)}(1+\delta)\leq\delta M,
\end{align*}
 by definition of $N_1$ and $h_1$. Consequently, we have on the event $\{(q_{t_n}^x,p_{t_n}^x)\in\mathrm{B}(y_n,h), Z^x_{t_n-\delta,t_n}\leq M\}$,
\[\begin{aligned}
(q^x_{t_n-\delta}-q_0)\cdot n(q_0) &=\underbrace{-\delta p_0\cdot n(q_0)}_{=3 \delta M}+\underbrace{\int_{t_n-\delta}^{t_n}(p^x_{t_n}-p^x_s)\cdot n(q_0)\mathrm{d}s}_{\vert \cdot\vert\leq \delta M }+\underbrace{v_{t_n}^x\cdot n(q_0)}_{\vert \cdot\vert\leq \delta M }\geq \delta M .
\end{aligned}\]
Furthermore, on the event $\{q_{t_n}^x,p_{t_n}^x)\in\mathrm{B}(y_n,h),Z^x_{t_n-\delta,t_n}\leq M\}$,
\[\begin{aligned}
\left\vert q^x_{t_n-\delta}-q_0\right\vert&=\left\vert -\delta p_0-\int_{t_n-\delta}^{t_n}(p^x_s-p^x_{t_n})\mathrm{d}s+v_{t_n}^x\right\vert\\
&\leq\delta (\vert p_0\vert+2M) .
\end{aligned}\]
As a result, since $(q^x_{t_n-\delta}-q_0)\cdot n(q_0) \geq \delta M \geq \eta \vert q^x_{t_n-\delta}-q_0\vert$ and $\vert q^x_{t_n-\delta}-q_0\vert\leq\delta (\vert p_0\vert+2M)\leq c_0$, the exterior sphere condition ensures that $q^x_{t_n-\delta}\notin\mathcal{O}$. 

\medskip \noindent \textbf{Step 2}. Let $M > 0$ be defined as in \textbf{Step~1}. We fix a value of $\alpha \in (0,1)$, let $C_0,\mu,\delta_0 > 0$ be given by Lemma~\ref{lemma:majoration densite} and define $\delta_2 := \delta_0 \wedge (t_0/2)$. We now let $\delta \in (0,\delta_2]$, define $N_2 \geq 1$ be such that for any $n \geq N_2$, $|y_n-y_0| \leq M/12$, and finally set $h_2 := M/12$.

Let $n \geq N_2$, $h \in (0,h_2]$ and define the following stopping time  $$\tau^{(\delta)}_{n}:=\inf \{s\geq t_n-\delta: \vert p^x_s-p_0\vert\geq M/2\}.$$ On the event $\{(q_{t_n}^x,p_{t_n}^x)\in \mathrm{B}(y_n,h), Z^x_{t_n-\delta,t_n}>M\}$, one has by the triangle inequality
\begin{align*}
\sup_{t_n-\delta\leq s\leq t_n}\vert p^x_s-p_0\vert
  &\geq\sup_{t_n-\delta\leq s\leq t_n}\vert p^x_s-p^x_{t_n}\vert
    -\vert p^x_{t_n}-p_n\vert -\vert p_n-p_0\vert \geq M-h-\frac{M}{12} \geq \frac{5M}{6} > \frac{M}{2},
\end{align*}
by the definitions of $N_2$ and $h_2$. Therefore, $\tau^{(\delta)}_n\leq t_n$ and
$$\mathbb{P}((q_{t_n}^x,p_{t_n}^x)\in \mathrm{B}(y_n,h),
Z^x_{t_n-\delta,t_n}>M)\leq \mathbb{P}((q_{t_n}^x,p_{t_n}^x)\in
\mathrm{B}(y_n,h), \tau^{(\delta)}_n<t_n),$$ since $\mathbb{P}(\tau^{(\delta)}_n=t_n)\leq \mathbb{P}(\vert p^x_{t_n}-p_0\vert=M/2)=0$ because $p^x_{t_n}$ admits a density on $\mathbb{R}^d$ with respect to the Lebesgue measure by Proposition~\ref{prop:kolmo-langevin}.

Therefore, applying the strong Markov property at $\tau^{(\delta)}_{n}$, one has
\begin{equation}\label{ineq estimee 1}
    \frac{\mathbb{P}((q_{t_n}^x,p_{t_n}^x)\in \mathrm{B}(y_n,h), Z^x_{t_n-\delta,t_n}>M)}{\vert \mathrm{B}(y_n,h)\vert}\leq\mathbb{E}\left[ \mathbb{1}_{\tau^{(\delta)}_{n}< t_n} \frac{\mathbb{P}\left((q_{t_n-r}^z,p_{t_n-r}^z)\in \mathrm{B}(y_n,h)\right)\Big\vert_{z=\left(q^x_{\tau^{(\delta)}_{n}},p^x_{\tau^{(\delta)}_{n}}\right),r=\tau^{(\delta)}_{n}}}{\vert \mathrm{B}(y_n,h)\vert}\right]\\ 
\end{equation}
Let $s_n:=t_n-\tau^{(\delta)}_n$. On the event $\{\tau^{(\delta)}_{n}< t_n\}$, one has $s_n\in(0,\delta]$ and
\begin{equation}\label{ineq estimee 2}
    \mathbb{P}\left((q_{t_n-r}^z,p_{t_n-r}^z)\in \mathrm{B}(y_n,h)\right)\Big\vert_{z=\left(q^x_{\tau^{(\delta)}_{n}},p^x_{\tau^{(\delta)}_{n}}\right),r=\tau^{(\delta)}_{n}}=\int_{\mathrm{B}(y_n,h)} \mathrm{p}_{s_n}((q^x_{\tau^{(\delta)}_{n}},p^x_{\tau^{(\delta)}_{n}}),y') \mathrm{d}y'.  
\end{equation}
Besides, since $s_n\leq\delta$ and $\delta\leq t_0/2$, one has by Theorem
\ref{borne densite thm} that there exists $C'>0$
depending only on $\alpha$ and $t_0$, but not on $n$, such that for any $y'\in\mathrm{B}(y_n,h)$, 
\begin{equation}\label{ineq estimee 3}
    \mathrm{p}_{s_n}((q^x_{\tau^{(\delta)}_{n}},p^x_{\tau^{(\delta)}_{n}}),y')\leq C'\widehat{\mathrm{p}}^{(\alpha)}_{s_n}((q^x_{\tau^{(\delta)}_{n}},p^x_{\tau^{(\delta)}_{n}}),y'). 
\end{equation}

It now follows from the definition of $N_2$ and $h_2$ that $\mathrm{B}(y_n,h) \subset \mathrm{B}(y_0, M/6)$ and, from the continuity of the trajectories of $(p^x_t)_{t\geq0}$, one has almost surely that $|p^x_{\tau^{(\delta)}_n}-p_0| \geq M/2$ so that for any $y' \in \mathrm{B}(y_n,h)$,
\begin{align*}
    \left|p^x_{\tau^{(\delta)}_n}-p'\right|&\geq\left|p^x_{\tau^{(\delta)}_n}-p_0\right|-\left|p_0-p'\right| \geq M/2-M/6\geq M/3.
\end{align*}
These estimates allow to apply Lemma~\ref{lemma:majoration densite} and deduce that, on the event $\{\tau^{(\delta)}_{n}< t_n\}$,
\begin{equation*}
  \widehat{\mathrm{p}}^{(\alpha)}_{s_n}((q^x_{\tau^{(\delta)}_{n}},p^x_{\tau^{(\delta)}_{n}}),y') \leq C_0\exp(-\mu/s_n) \leq C_0\exp(-\mu/\delta),
\end{equation*}
which, combined with~(\ref{ineq estimee 1}--\ref{ineq estimee 3}), concludes to~\eqref{eq:step2 estimate}. 
\end{proof}

\begin{remark}
  A formal conditioning argument shows that, for $x,y \in D$,
  \begin{align*}
    \mathrm{p}^D_t(x,y) &= \lim_{h \to 0} \frac{\mathbb{P}(X^x_t \in \mathrm{B}(y,h), \tau^x_\partial > t)}{|\mathrm{B}(y,h)|}\\
    &= \lim_{h \to 0} \frac{\mathbb{P}(\tau^x_\partial > t | X^x_t \in \mathrm{B}(y,h))\mathbb{P}(X^x_t \in \mathrm{B}(y,h))}{|\mathrm{B}(y,h)|}\\
    &= \mathbb{P}(\tau^x_\partial > t | X^x_t=y)\mathrm{p}_t(x,y),
  \end{align*}
  so that Proposition~\ref{comportement densite} should amount to studying the limiting behavior, when $x$ or $y$ respectively approach $\Gamma^+$ or $\Gamma^-$, of the probability that the diffusion bridge associated with~\eqref{Langevin} between~$x$ and~$y$ remains in $D$. With this interpretation at hand, both convergence results~\eqref{it:comportement-densite:1} and~\eqref{it:comportement-densite:2} become very intuitive, and they seem to be the time-reversal statement of each other --- a point which will be clarified with the introduction of the adjoint process, and the proof of the reversibility relation~\eqref{duality}, in the next section.
  
 Our proof of Proposition~\ref{comportement densite}, and in particular of~\eqref{it:comportement-densite:2}, can be related to the work on diffusion bridges of Chaumont and Uribe-Bravo in~\cite{MarkovBridge}, where they study formal characterizations of the $h \to 0$ limit of such an expression as $\mathbb{P}(\tau^x_\partial > t | X^x_t \in \mathrm{B}(y,h))$. 
\end{remark}

\section{Reversibility and boundary continuity}
\label{Section Adjoint process and compactness}\sectionmark{Reversibility and compactness properties}

In this section we define the "adjoint" Langevin process, which is later shown to be closely related to the Langevin process, through a reversibility result linking both transition densities of the respective absorbed processes. This result is useful for being able to describe precisely the boundary behavior of~$\mathrm{p}^D_t$ and thereby complete the proof of Theorem~\ref{thm density intro}.

\subsection{Adjoint process}
 Let $x=(q,p)\in\mathbb{R}^{2d}$. Let us call the  "adjoint" Langevin
 process the diffusion process
 $(\widetilde{X}^x_t=(\widetilde{q}^x_t,\widetilde{p}^x_t))_{t\geq0}$
 with infinitesimal generator
 $\widetilde{\mathcal{L}}:=\mathcal{L}^*-d\gamma$
 (see~\eqref{generateur adjoint} for the definition of $\mathcal{L}^*$), satisfying  the following SDE:  
\begin{equation}\label{Langevin adjoint}
  \left\{
    \begin{aligned}
&        \mathrm{d}\widetilde{q}^x_t=-\widetilde{p}^x_t \mathrm{d}t , \\
 &       \mathrm{d}\widetilde{p}^x_t=-F(\widetilde{q}^x_t) \mathrm{d}t+ \gamma \widetilde{p}^x_t \mathrm{d}t+\sigma \mathrm{d}B_t ,\\
  &      (\widetilde{q}^x_0,\widetilde{p}^x_0)=x .
    \end{aligned}
\right.   
\end{equation} 
Let $\widetilde{\tau}^x_{\partial}$ be the first exit time from $D$ of $(\widetilde{X}^x_t)_{t \geq 0}$, i.e.
$$ \widetilde{\tau}^x_{\partial}=\inf \{t>0: \widetilde{X}^x_t\notin D\} .$$
\noindent Let $y:=(q,-p)$. Let us now define the process $(\widetilde{X}_t^{\diamond,y}=(\widetilde{q}_t^{\diamond,y},\widetilde{p}_t^{\diamond,y}))_{t\geq0}:=(\widetilde{q}_t^x,-\widetilde{p}_t^x)_{t\geq0}$,
it is easy to see that it satisfies the following SDE
\begin{equation}\label{Langevin bis}
  \left\{
    \begin{aligned}
&        \mathrm{d}\widetilde{q}^{\diamond,y}_t=\widetilde{p}^{\diamond,y}_t \mathrm{d}t , \\
 &       \mathrm{d}\widetilde{p}^{\diamond,y}_t=F(\widetilde{q}^{\diamond,y}_t) \mathrm{d}t+\gamma \widetilde{p}^{\diamond,y}_t \mathrm{d}t+\sigma \mathrm{d}B^\diamond_t ,\\
  &      (\widetilde{q}^{\diamond,y}_0,\widetilde{p}^{\diamond,y}_0)=(q,-p)=y ,
    \end{aligned}
\right.   
\end{equation}
where $(B^\diamond_t)_{t\geq0}=(-B_t)_{t\geq0}$ is a Brownian motion on $\mathbb{R}^{d}$. Its infinitesimal generator therefore writes 
\begin{equation*}
  \widetilde{\mathcal{L}}^\diamond := \mathcal{L}_{F,-\gamma,\sigma}
\end{equation*}
with the notation of~\eqref{generateur Langevin}. Hence all the results proven in the previous sections apply to $(\widetilde{X}_t^{\diamond,y})_{t\geq0}$ as well.
Furthermore, $\widetilde{X}^{\diamond,y}_t$ and $\widetilde{X}^x_t$ share the same first exit time from $D$, i.e. $$\widetilde{\tau}^{\diamond,y}_{\partial}:=\inf \{t>0: \widetilde{X}^{\diamond,y}_t\notin D\}= \widetilde{\tau}^x_{\partial}\quad\text{almost surely.}$$

Let us now write and prove the equivalent of Theorem \ref{Solution PDE} for the process $(\widetilde{X}^x_t)_{t\geq0}$. 

By Proposition~\ref{prop:tau} applied to $(\widetilde{X}^{\diamond,y}_t)_{t \geq 0}$, we have, for all $t \geq 0$, almost surely, if $\widetilde{\tau}^x_{\partial}>t$ then $\widetilde{X}^x_t \in D\cup\Gamma^+$, and
if $\widetilde{\tau}^x_{\partial}\leq t$ then $\widetilde{X}^x_{\tau^x_{\partial}} \in
\Gamma^-\cup\Gamma^0$. This ensures that the definition of the
function $\widetilde{u}$ in Equation~\eqref{v} below is legitimate. 

\begin{proposition}[Classical solution and probabilistic representation for the adjoint kinetic Fokker-Planck equation]\label{Solution adjoint PDE} Under Assumptions~\ref{hyp O} and~\ref{hyp F1}, let $\widetilde{f}\in\mathcal{C}^b(D\cup\Gamma^+)$ and $\widetilde{g}\in\mathcal{C}^b(\Gamma^-\cup\Gamma^0)$, and define the function $\widetilde{u}$ on $\mathbb{R}_+\times \overline{D}$ by
\begin{equation}\label{v}
    \widetilde{u}:(t,x)\mapsto\mathbb{E}\left[ \mathbb{1}_{\widetilde{\tau}^x_{\partial}>t} \widetilde{f}(\widetilde{X}^x_t) +\mathbb{1}_{\widetilde{\tau}^x_{\partial}\leq t} \widetilde{g}(\widetilde{X}^x_{\tau^x_{\partial}}) \right].
    \end{equation}
  Then we have the following results:
\begin{enumerate}[label={\rm(\roman*)},ref=\roman*]
  \item\label{it:ibvp:val 2} Initial and boundary values: the function $\widetilde{u}$ satisfies
 \begin{equation*}
\widetilde{u}(0,x)=\left\{
\begin{aligned}
    \widetilde{f}(x) &\quad \text{if $x\in D\cup\Gamma^+$,}\\
    \widetilde{g}(x) &\quad \text{if $x\in \Gamma^-\cup\Gamma^0$},
\end{aligned}
\right. 
\end{equation*}
and 
\begin{equation*}
  \forall t>0, \quad \forall x\in\Gamma^-\cup\Gamma^0, \qquad \widetilde{u}(t,x)=\widetilde{g}(x).
\end{equation*}
  \item\label{it:ibvp:cont 2} Continuity: $\widetilde{u} \in \mathcal{C}^b((\mathbb{R}_+\times\overline{D}) \setminus (\{0\}\times(\Gamma^-\cup\Gamma^0)) )$, and if $\widetilde{f}$ and $\widetilde{g}$ satisfy the compatibility condition
\begin{equation}\label{compatibility cond 2}
    x\in\overline{D}\mapsto\mathbb{1}_{x\in
      D\cup\Gamma^+}\widetilde{f}(x)+\mathbb{1}_{x \in \Gamma^-\cup\Gamma^0}\widetilde{g}(x)\in\mathcal{C}^b(\overline{D}),
      \end{equation}
then $\widetilde{u}\in\mathcal{C}^b(\mathbb{R}_+\times\overline{D})$.
  \item\label{it:ibvp:reg 2} Interior regularity: $\widetilde{u} \in \mathcal{C}^\infty(\mathbb{R}_+^*\times D)$ and, for all $t>0$, $x \in D$,
  \begin{equation}\label{edp adjoint}
    \partial_t\widetilde{u}(t,x)=\widetilde{\mathcal{L}}\widetilde{u}(t,x).
    \end{equation}
  \item\label{it:ibvp:uniq 2} Uniqueness: let $\widetilde{v}$ be a
    classical solution, in the sense of
    Definition~\ref{def:class}, to the Initial-Boundary Value Problem
    \begin{equation*}
      \left\{\begin{aligned}
        \partial_t \widetilde{v} &= \widetilde{\mathcal{L}} \widetilde{v} && t > 0, \quad x \in D,\\
        \widetilde{v}(0,x) &= \widetilde{f}(x) && x \in D,\\
        \widetilde{v}(t,x) &= \widetilde{g}(x) && t > 0, \quad x \in \Gamma^-.
      \end{aligned}\right.
    \end{equation*}
    If, for all $T>0$, $\widetilde{v}$ is bounded on the set $[0,T] \times D$, then $\widetilde{v}(t,x)=\widetilde{u}(t,x)$ for all $(t,x)\in (\mathbb{R}_+\times(D\cup\Gamma^-))\setminus(\{0\}\times \Gamma^-)$.
\end{enumerate}  
\end{proposition}
\begin{proof}
Let $\widetilde{f}\in\mathcal{C}^b(D\cup\Gamma^+)$ and
$\widetilde{g}\in\mathcal{C}^b(\Gamma^-\cup\Gamma^0)$. Let $\widetilde{f}^{\diamond}$,
$\widetilde{g}^{\diamond}$ be defined by
$$\widetilde{f}^{\diamond}(q,p)=\widetilde{f}(q,-p),\qquad\widetilde{g}^{\diamond}(q,p)=\widetilde{g}(q,-p).$$
It is easy to see that $\widetilde{f}^{\diamond}\in\mathcal{C}^b(D\cup\Gamma^-)$ and $\widetilde{g}^{\diamond}\in\mathcal{C}^b(\Gamma^+\cup\Gamma^0)$. Using the process $(\widetilde{X}^{\diamond,x}_t)_{t\geq0}$ defined in~\eqref{Langevin bis}, the function $\widetilde{u}$ defined in \eqref{v} also writes for $(q,p)\in D$ as follows:
$$\widetilde{u}(t,(q,p))=\mathbb{E}\left[
  \mathbb{1}_{\widetilde{\tau}^{\diamond,(q,-p)}_{\partial}>t}
  \widetilde{f}^{\diamond}(\widetilde{X}^{\diamond,(q,-p)}_t)
  +\mathbb{1}_{\widetilde{\tau}^{\diamond,(q,-p)}_{\partial}\leq t} \widetilde{g}^\diamond \big(\widetilde{X}^{\diamond,(q,-p)}_{\widetilde{\tau}^{\diamond,(q,-p)}_{\partial}}\big) \right] . $$

Let us define $\widetilde{u}^\diamond$ for $t\geq0$,
$(q,p)\in\overline{D}$ by $\widetilde{u}^\diamond(t,(q,p)):=\widetilde{u}(t,(q,-p))$. Then,
$\widetilde{u}^\diamond$ satisfies all the assertions of Theorem~\ref{Solution PDE} for the kinetic Fokker-Planck equation
\begin{equation*}
  \left\{\begin{aligned}
    \partial_t \widetilde{u}^\diamond(t,x) & =\widetilde{\mathcal{L}}^\diamond\widetilde{u}^\diamond(t,x) && t>0, \quad x\in D ,\\
    \widetilde{u}^\diamond(0,x) &=\widetilde{f}^\diamond(x) && x\in D ,\\
    \widetilde{u}^\diamond(t,x) &=\widetilde{g}^\diamond(x)  && t>0, \quad x\in\Gamma^+ .
  \end{aligned}\right. 
\end{equation*}

Therefore, $\widetilde{u}$ as defined in \eqref{v} satisfies all the assertions of Proposition \ref{Solution adjoint PDE}.
\end{proof}

Let us define the transition kernel
$\widetilde{\mathrm{P}}^D_t$ for the absorbed adjoint process $(\widetilde{X}^x_t)_{0 \leq t \leq \widetilde{\tau}^x_\partial}$:
$$\forall t \ge 0, \quad \forall x \in D, \quad \forall A\in\mathcal{B}(D), \qquad \widetilde{\mathrm{P}}^D_t(x,A):=\mathbb{P}(\widetilde{X}^x_t\in A,\widetilde{\tau}^x_\partial>t) .$$
In the next theorem, we show that this kernel admits a transition density $\widetilde{\mathrm{p}}_t^D$ which satisfies a simple reversibility relation with $\mathrm{p}_t^D$.  

\begin{theorem}[Reversibility]\label{duality thm} Let Assumptions \ref{hyp O} and \ref{hyp F1} hold. For all $t>0$, $x,y\in D$, let us define
\begin{equation}\label{duality}
    \widetilde{\mathrm{p}}_t^D(x,y)=\mathrm{e}^{-d \gamma t} \mathrm{p}_t^D(y,x).
\end{equation}
For any $t>0$, $x \in D$ and $A \in \mathcal{B}(D)$,
\begin{equation*}
  \widetilde{\mathrm{P}}^D_t(x,A)=\int_A\widetilde{\mathrm{p}}_t^D(x,y) \mathrm{d}y.
\end{equation*}
\end{theorem}
\begin{proof} 
Let $\varphi\in\mathcal{C}_c^\infty(D)$. Let $\widetilde{u} : \mathbb{R}_+\times\overline{D} \to \mathbb{R}$ be  defined by $$\widetilde{u}(t,x):=\mathbb{E}\left[ \mathbb{1}_{\widetilde{\tau}^x_{\partial}>t} \varphi(\widetilde{X}^x_t) \right].$$
By Assertion~\eqref{it:ibvp:cont 2} in Proposition~\ref{Solution adjoint PDE}, this function is continuous on $\mathbb{R}_+\times\overline{D}$. Let us define the  function~$\widetilde{v}$ on $\mathbb{R}_+\times\overline{D}$ by  
\begin{equation*}
\widetilde{v}(t,x)=\begin{cases}
 \displaystyle\mathrm{e}^{-d\gamma t}\int_D \mathrm{p}_t^D(y,x) \varphi(y) \mathrm{d}y & \text{if $t>0$ and $x\in D$,}\\
 \varphi(x) & \text{if $(t,x)\in(\mathbb{R}_+\times\overline{D})\setminus(\mathbb{R}^*_+\times D)$.}
\end{cases}
\end{equation*}

Let us prove that $\widetilde{v}(t,x)=\widetilde{u}(t,x)$ for $t>0$ and $x \in D$, which will ensure \eqref{duality}. In this purpose, we use the uniqueness result of Assertion \eqref{it:ibvp:uniq 2} in Proposition \ref{Solution adjoint PDE}. By Definition~\ref{def:class}, we need to check that:
\begin{enumerate}[label=(\roman*),ref=\roman*]
   \item $(t,x)\mapsto
      \widetilde{v}(t,x)\in\mathcal{C}^{1,2}(\mathbb{R}_+^{*}\times D)$
      and $\widetilde{v}$ satisfies $\partial_t \widetilde{v} = \widetilde{\mathcal{L}}\widetilde{v}$,
    \item $(t,x)\mapsto \widetilde{v}(t,x)
      \in\mathcal{C}((\mathbb{R}_+\times(
      D\cup\Gamma^-))\setminus(\{0\}\times\Gamma^-))$, $\widetilde{v} (0,\cdot)=\varphi$ on $D$ and,
      for $t>0$, $\widetilde{v} (t,\cdot)=0$ on $\Gamma^-$,
   \item $\forall  T>0$, $\sup_{t\in[0,T],x\in D}\vert \widetilde{v}(t,x)\vert<\infty$.
 \end{enumerate}
 
 Since $\varphi$ has a compact support in $D$ and, by Proposition~\ref{regularité densité}, $\mathrm{p}^D$ is $\mathcal{C}^\infty$ on $\mathbb{R}_+^* \times D \times D$ and satisfies $\partial_t \mathrm{p}^D(x,y) = \mathcal{L}_x \mathrm{p}^D(x,y)$, we deduce that $\widetilde{v}$ is $\mathcal{C}^{1,2}$ on $\mathbb{R}_+^{*}\times D$ and satisfies $\partial_t
 \widetilde{v}=(\mathcal{L}^*-d\gamma)\widetilde{v}=\widetilde{\mathcal{L}}\widetilde{v}$.

 Let $t>0$ and $x \in \Gamma^-$. By the definition of $\widetilde{v}$, we have $\widetilde{v}(t,x)=\varphi(x)=0$ since $\varphi$ has a compact support in $D$. On the other hand, if $(t_n,x_n)_{n \geq 1}$ is a sequence of elements of $(\mathbb{R}_+\times(
      D\cup\Gamma^-))\setminus(\{0\}\times\Gamma^-)$ which converge to $(t,x)$, then it follows from Assertion~\eqref{it:comportement-densite:2} in Proposition \ref{comportement densite}, Remark~\ref{Rq
   densite estimation} and the dominated convergence
 theorem that $\widetilde{v}(t_n,x_n)$ converges to $0$.
 
 Similarly, if $x \in D$ then it follows from the definition of $\widetilde{v}$ that $\widetilde{v}(0,x)=\varphi(x)$. Now let $(t_n,x_n)_{n \geq 1}$ be a sequence of elements of $(\mathbb{R}_+\times(D\cup\Gamma^-))\setminus(\{0\}\times\Gamma^-)$ which converge to $(0,x)$, and let us check that $\widetilde{v}(t_n,x_n)$ converges to $\varphi(x)$. We first remark that if $t_n=0$ then $\tilde{v}(t_n,x_n)=\varphi(x_n)$, so that along the subsequence $\{n \geq 1: t_n=0\}$, the claimed convergence is immediate. Therefore, we may now assume that $t_n>0$ for any $n \geq 1$. Let $K \subset D$ be a compact set which contains the support of $\varphi$ and an open ball centered at $x$. There exists $N_1\geq1$ such that for all $n\geq N_1$, $x_n \in K$. Moreover, there exists $N_2\geq1$ such that for $n\geq N_2$, $t_n \in (0,1]$ since $t_n\underset{n\rightarrow\infty}{{\longrightarrow}} 0$. Therefore, by Lemma \ref{lemma mild inequality}, there exist a constant $C>0$ and $\alpha\in(0,1)$ such that for all $y\in K$ and $n\geq N_1 \vee N_2$,
$$\left\vert\mathrm{p}^D_{t_n}(y,x_n)-\widehat{\mathrm{p}}_{t_n}(y,x_n)\right\vert\leq C \sqrt{t_n} \widehat{\mathrm{p}}^{(\alpha)}_{t_n}(y,x_n) . $$
Consequently, since $\varphi=0$ outside $K$,
\begin{align*}
\left\vert\widetilde{v}(t_n,x_n) -\mathrm{e}^{-d \gamma t_n}\int_{D}\widehat{\mathrm{p}}_{t_n}(y,x_n) \varphi(y) \mathrm{d}y\right\vert&\leq C\sqrt{t_n}\mathrm{e}^{-d \gamma t_n} \int_D \widehat{\mathrm{p}}^{(\alpha)}_{t_n}(y,x_n) \varphi(y) \mathrm{d}y\\
&\leq C \Vert\varphi\Vert_\infty \sqrt{t_n}\mathrm{e}^{-d \gamma t_n} \int_{\mathbb{R}^{2d}} \widehat{\mathrm{p}}^{(\alpha)}_{t_n}(y,x_n) \mathrm{d}y .
\end{align*}
By Lemma \ref{prop densite}, one has that  
$$\int_{\mathbb{R}^{2d}} \widehat{\mathrm{p}}^{(\alpha)}_{t_n}(y,x_n)
\mathrm{d}y=\mathrm{e}^{d \gamma
  t_n},\qquad\int_{D}\widehat{\mathrm{p}}_{t_n}(y,x_n) \varphi(y)
\mathrm{d}y \underset{n\rightarrow\infty}{{\longrightarrow}}
\varphi(x) . $$ Therefore,
$\widetilde{v}(t_n,x_n)\underset{n\rightarrow\infty}{{\longrightarrow}}
\varphi(x)$. 

We finally fix $T>0$ and show that $\sup_{t\in[0,T],x\in D}\vert
\widetilde{v}(t,x)\vert<\infty$. Again, Corollary~\ref{Rq densite estimation} ensures the existence of $C'>0$ such that for all $t\in (0,T]$, $x\in D$,
\begin{align*}
\left\vert\widetilde{v}(t,x)\right\vert&=\mathrm{e}^{-d \gamma t}\left\vert\int_{D}\mathrm{p}^D_t(y,x) \varphi(y) \mathrm{d}y\right\vert\\
&\leq C' \Vert\varphi\Vert_\infty\mathrm{e}^{-d \gamma t} \int_{\mathbb{R}^{2d}} \widehat{\mathrm{p}}^{(\alpha)}_{t}(y,x) \mathrm{d}y\\
&\leq C' \Vert\varphi\Vert_\infty,
\end{align*}
using  Lemma \ref{prop densite}, which concludes the proof.
\end{proof}

\subsection{Completion of the proof of Theorem~\ref{thm density intro}}

Let us now conclude this section with results on the boundary
continuity of the density $\mathrm{p}^D_t(x,y)$ for a fixed $t>0$. The
proof relies on Theorem~\ref{duality thm}, and this result will
complete the proof of Theorem~\ref{thm density intro}. It also completes the results of Proposition \ref{comportement densite} since we consider the continuity with respect to the three variables $(t,x,y)$ at the same time and extend the limit with respect to $y$ going to a point in $\Gamma^0$. 

\begin{theorem}[Boundary continuity]\label{boundary property density}
Under Assumptions \ref{hyp O} and \ref{hyp F1}, the transition density $\mathrm{p}^D$ can be extended to a function in $\mathcal{C}(\mathbb{R}_+^*\times\overline{D}\times\overline{D})$ which satisfies for all $t>0$: 
\begin{enumerate}[label=(\roman*),ref=\roman*]
    \item $\mathrm{p}^D_t(x,y)=0$ if $x\in \Gamma^+\cup\Gamma^0$ or if $y\in\Gamma^-\cup\Gamma^0$, 
    \item if Assumption~\ref{hyp:conn} holds, $\mathrm{p}^D_t(x,y)>0$ for all $x\notin \Gamma^+\cup\Gamma^0$ and $y\notin\Gamma^-\cup\Gamma^0$.
\end{enumerate} 
\end{theorem}
\begin{proof}
\medskip \noindent \textbf{Step 1}. We first study the behavior of $\mathrm{p}^D_t(x,y)$ when $x$ and $y$ approach $\partial D$, and show that the function $\mathrm{p}^D$ can be continuously extended on $\mathbb{R}_+^* \times \overline{D} \times \overline{D}$. Let $t_0>0$, $x_0\in \overline{D}$ and
$y_0\in\overline{D}$. Let $(t_n,x_n,y_n)_{n\geq1}$ be a sequence of
points in $\mathbb{R}_+^*\times D\times D$ converging towards
$(t_0,x_0,y_0)$. 

In the next three cases, we show that $\mathrm{p}^D_{t_n}(x_n,y_n)$ has a limit which does not depend on the sequence $(t_n,x_n,y_n)_{n \geq 1}$. If $x_0, y_0 \in D$ then by Proposition~\ref{regularité densité}, this limit coincides with $\mathrm{p}^D_{t_0}(x_0,y_0)$. Otherwise, we denote this limit by $\mathrm{p}^D_{t_0}(x_0,y_0)$, which thereby defines a continuous function $\mathrm{p}^D$ on $\mathbb{R}_+^* \times \overline{D} \times \overline{D}$.

\medskip\noindent \textbf{Case 1:} Assume that $x_0\in\Gamma^+\cup\Gamma^0$. By Assertion~\eqref{it:comportement-densite:1} in Proposition~\ref{comportement densite} and Remark~\ref{rk:comportement-densite:cvunif}, we immediately get $$\mathrm{p}^D_{t_n}(x_n,y_n)\underset{n\rightarrow\infty}{{\longrightarrow}} 0, $$
and therefore set $\mathrm{p}^D_{t_0}(x_0,y_0)=0$.

\medskip\noindent \textbf{Case 2:} Assume that $y_0=(q_0,p_0)\in\Gamma^-\cup\Gamma^0$. For any $(q,p) \in \mathbb{R}^{2d}$, let us define $\diamond (q,p) := (q,-p)$. From the definition of the process $(\widetilde{X}^{\diamond,x}_t)_{t \geq 0} = (\diamond \widetilde{X}^{\diamond x}_t)_{t \geq 0}$, we deduce that the absorbed version of the latter possesses a transition density $\widetilde{\mathrm{p}}^{\diamond, D}_t$ which satisfies
\begin{equation*}
  \widetilde{\mathrm{p}}^{\diamond, D}_t(x,y) = \widetilde{\mathrm{p}}^D_t(\diamond x, \diamond y).
\end{equation*}
As a consequence, using Theorem~\ref{duality thm} we rewrite
\begin{equation*}
  \mathrm{p}^D_{t_n}(x_n,y_n) = \mathrm{e}^{d \gamma t_n} \widetilde{\mathrm{p}}^D_{t_n}(y_n,x_n) = \mathrm{e}^{d \gamma t_n} \widetilde{\mathrm{p}}^{\diamond,D}_{t_n}(\diamond y_n,\diamond x_n).
\end{equation*}
On the one hand, $\diamond y_n \to \diamond y_0 \in \Gamma^+ \cup \Gamma^0$. On the other hand, the process $(\widetilde{X}^{\diamond,x}_t)_{t \geq 0}$ has infinitesimal generator $\widetilde{\mathcal{L}}^\diamond = \mathcal{L}_{F,-\gamma,\sigma}$, and therefore Assertion~\eqref{it:comportement-densite:1} in Proposition~\ref{comportement densite} and Remark~\ref{rk:comportement-densite:cvunif} apply to show that 
\begin{equation*}
  \sup_{x \in D} \widetilde{\mathrm{p}}^{\diamond, D}_{t_n}(\diamond y_n,x) \underset{n\rightarrow\infty}{{\longrightarrow}} 0,
\end{equation*}
from which we deduce that $\mathrm{p}^D_{t_n}(x_n,y_n)\underset{n\rightarrow\infty}{{\longrightarrow}} 0$ and set $\mathrm{p}^D_{t_0}(x_0,y_0)=0$.

\medskip\noindent\textbf{Case 3:} Assume that $x_0\in D\cup\Gamma^-$
and $y_0=(q_0,p_0)\in D\cup\Gamma^+$. 
For $h>0$, by the Markov property,
$$\forall 0 \leq s < t, \quad \forall x,y\in D,\qquad\frac{\mathrm{P}^D_{t}(x, D \cap \mathrm{B}(y,h))}{\vert\mathrm{B}(y,h)\vert} = \mathbb{E}\left[\mathbb{1}_{\tau^{x}_\partial > s}\frac{\mathrm{P}^D_{t-s}(X^{x}_{s}, D \cap \mathrm{B}(y,h))}{\vert\mathrm{B}(y,h)\vert} \right].$$
Using the Gaussian upper-bound from Corollary \ref{Rq densite estimation} and the dominated convergence theorem when $h\rightarrow0$, we obtain from the equality above the following Chapman-Kolmogorov relation: 
\begin{equation}\label{chapman kolmogorov 2}
    \forall 0 \leq s < t, \quad \forall x,y\in D, \qquad\mathrm{p}^D_t(x,y)=\mathbb{E}\left[\mathbb{1}_{\tau^{x}_\partial>s}  \mathrm{p}^D_{t-s}(X_s^x,y)\right]. 
\end{equation}
By~\eqref{chapman kolmogorov 2} applied with $s=t_n/3$, one has that 
\begin{equation}\label{chapman kolmogorov integrale}
    \mathrm{p}^D_{t_n}(x_n,y_n)=\int_D\mathrm{p}^D_{t_n/3}(x_n,z)\mathrm{p}^D_{2t_n/3}(z,y_n)\mathrm{d}z.
\end{equation}
Let us prove now the convergence of both integrands in \eqref{chapman kolmogorov integrale}. Using Theorem \ref{duality thm} and~\eqref{chapman kolmogorov 2} again, one has for all $z\in D$,
\begin{align*} 
\mathrm{p}^D_{2t_n/3}(z,y_n)&=\mathrm{e}^{2d\gamma t_n/3}\Tilde{\mathrm{p}}^D_{2t_n/3}(y_n,z)\\
&=\mathrm{e}^{2d\gamma t_n/3}\mathbb{E}\left[\mathbb{1}_{\Tilde{\tau}^{y_n}_\partial>t_n/3}  \Tilde{\mathrm{p}}^D_{t_n/3}(\Tilde{X}_{t_n/3}^{y_n},z)\right]\\ 
&=\mathrm{e}^{2d\gamma t_n/3}\mathbb{E}\left[\mathbb{1}_{\widetilde{\tau}^{\diamond, \diamond y_n}_\partial>t_n/3}  \widetilde{\mathrm{p}}^{\diamond,D}_{t_n/3}(\widetilde{X}^{\diamond, \diamond y_n}_{t_n/3},\diamond z) \right].
\end{align*}
By construction, $\widetilde{\mathrm{p}}^{\diamond,D}$ is continuous on $\mathbb{R}_+^*\times D\times D$ and is the transition density of the process $(\widetilde{X}^{\diamond,x}_t)_{t \geq 0}$ with infinitesimal generator $\widetilde{\mathcal{L}}^\diamond = \mathcal{L}_{F,-\gamma,\sigma}$. Therefore, Lemma \ref{cv indicatrices lemma} and Corollary \ref{Rq densite estimation} apply here and ensure, using the dominated convergence theorem, that for $z\in D$,
\begin{equation}\label{def h_1}
    \mathrm{p}^D_{2t_n/3}(z,y_n)\underset{n\rightarrow\infty}{\longrightarrow}h^{(1)}_t(z):=\mathrm{e}^{2d\gamma t_0/3}\mathbb{E}\left[\mathbb{1}_{\widetilde{\tau}^{\diamond, \diamond y_0}_\partial>t_0/3}  \widetilde{\mathrm{p}}^{\diamond,D}_{t_0/3}(\widetilde{X}^{\diamond, \diamond y_0}_{t_0/3},\diamond z)\right].
\end{equation} 
Furthermore, considering now the first integrand in \eqref{chapman kolmogorov integrale}, for all $z\in D$,
\begin{equation}\label{def h_2}
    \mathrm{p}^D_{t_n/3}(x_n,z)=\mathbb{E}\left[\mathbb{1}_{\tau^{x_n}_\partial>t_n/6}\mathrm{p}^D_{t_n/6}(X_{t_n/6}^{x_n},z)\right]\underset{n\rightarrow\infty}{\longrightarrow}h^{(2)}_t(z):=\mathbb{E}\left[\mathbb{1}_{\tau^{x_0}_\partial>t_0/6}\mathrm{p}^D_{t_0/6}(X_{t_0/6}^{x_0},z)\right],
\end{equation}
using the continuity of $\mathrm{p}^D$ and Lemma \ref{cv indicatrices lemma}. It remains to prove that the integral \eqref{chapman kolmogorov integrale} converges to the integral $\int_Dh^{(1)}_t(z)h^{(2)}_t(z)\mathrm{d}z$.  

Since the term $\mathrm{p}^D_{2t_n/3}(z,y_n)$ is bounded by a constant depending only on $t_n$ by Corollary \ref{Rq densite estimation} and since $t_n\underset{n\rightarrow\infty}{\longrightarrow}t_0>0$ the associated constant can easily be obtained independent of $n$. In order to use the dominated convergence theorem to the product of both integrands in  \eqref{chapman kolmogorov integrale}, it remains to obtain a bound on $\mathrm{p}^D_{t_n/3}(x_n,z)$ in $\mathrm{L}^1(D)$ which is independent of $n$. This follows from Lemma \ref{lem: borne densite indep de n} below since $(t_n,x_n)$ converges to $(t_0,x_0)\in\mathbb{R}_+^*\times \overline{D}$, therefore the sequence $(t_n,x_n)_{n\geq1}$ stays in some compact set $K$ of $\mathbb{R}_+^*\times \overline{D}$. 
We thus obtain a limit independent on the sequence $(t_n,x_y,y_n)_{n \geq 1}$. By Proposition~\ref{regularité densité}, if $x_0,y_0 \in D$, it coincides with $\mathrm{p}^D_{t_0}(x_0,y_0)$, otherwise we denote it by $\mathrm{p}^D_{t_0}(x_0,y_0)$.

\medskip \noindent\textbf{Step 2}. Let us now work under Assumption~\ref{hyp:conn} and first prove that for all $t>0$ and $x,y\in D$,
\begin{equation}\label{positivite densite}
    \mathrm{p}^D_t(x,y)>0 . 
\end{equation}
Let us argue by contradiction. Assume there exists $t_0>0$ and
$x_0,y_0\in D$ such that $\mathrm{p}^D_{t_0}(x_0,y_0)=0$. Let us
introduce $$\Psi:(t,y)\in\mathbb{R}_+^*\times D\mapsto \mathrm{e}^{-d \gamma t} \mathrm{p}^D_t(x_0,\diamond y) .$$
It follows from the forward Kolmogorov equation satisfied by
$\mathrm{p}^D_t(x_0,\cdot)$ (see Proposition \ref{regularité densité}), that on $\mathbb{R}_+^*\times D$,
$$\partial_t\Psi =\widetilde{\mathcal{L}}^\diamond\Psi.$$ 
Using the Harnack inequality stated in Theorem~\ref{Harnack}, we have
that for any $K\subset D$ compact set and $t\in(0,t_0)$ there exists
$C>0$ such that
$$\sup_{y\in K}\Psi(t,y)\leq C \inf_{y\in K}\Psi(t_0,y) .$$
In particular, it yields that for all $y\in D$, $t\in(0,t_0)$,
$\Psi(t,y)=0$. As a result, for all $t\in(0,t_0)$ and $y\in D$, $ \mathrm{p}^D_t(x_0,y)=0.$
Integrating over $y\in D$, it follows that for all $t\in(0,t_0)$,
$$\mathbb{P}(\tau^{x_0}_\partial>t)=0 .$$

Besides, it follows from Lemma \ref{cv indicatrices lemma} and
Proposition \ref{prop:tau} that
$\mathbb{P}(\tau^{x_0}_\partial>t)\underset{t\rightarrow
  0}{{\longrightarrow}}1$. This is contradiction with the equality
above, and this thus concludes the proof of~\eqref{positivite densite}.

\medskip \noindent\textbf{Step 3}. It remains to extend the result of \textbf{Step~2} to show that $\mathrm{p}^D_t(x,y)>0$ for $t>0$, $x \in D \cup \Gamma^-$ and $y \in D \cup \Gamma^+$. In this purpose, we first show that, for any $x_0 \in D \cup \Gamma^-$, for all $t>0$,
\begin{equation}\label{proba gamma -}
    \mathbb{P}(\tau^{x_0}_\partial>t)>0 . 
  \end{equation}
  Indeed, using again Lemma \ref{cv indicatrices lemma} and Proposition \ref{prop:tau}, there exists necessarily $s\in (0,t)$ such that $\mathbb{P}(\tau^{x_0}_\partial>s)>0$. As a result, the Markov property at time $s$ ensures that
$$\mathbb{P}(\tau^{x_0}_\partial>t)=\mathbb{E}\left[\mathbb{1}_{\tau^{x_0}_\partial>s}  \mathbb{P}(\tau^z_\partial>t-s)\vert_{z=X^{x_0}_{s}} \right] >0 $$
by \eqref{positivite densite}, which yields~\eqref{proba gamma -}. 

We now recall from \textbf{Case~3} in \textbf{Step~1} that for $t>0$, $x \in D \cup \Gamma^-$ and $y \in D \cup \Gamma^+$, 
\begin{equation*}
  \mathrm{p}^D_t(x,y) = \int_Dh^{(1)}_t(z)h^{(2)}_t(z)\mathrm{d}z,
\end{equation*}
where $h_1,h_2$ are defined in \eqref{def h_1} and \eqref{def h_2}.
By \textbf{Step~2} applied to $\widetilde{\mathrm{p}}^{\diamond, D}_{t/3}$, for all $z \in D$ we have $\widetilde{\mathrm{p}}^{\diamond, D}_{t/3}(\widetilde{X}^{\diamond,\diamond y}_{t/3}, \diamond z)>0$ on the event $\{\widetilde{\tau}^{\diamond,\diamond y}_\partial > t/3\}$. And by~\eqref{proba gamma -} applied to $(\widetilde{X}^{\diamond,\diamond y}_t)_{t \geq 0}$, this event has positive probability, so that $h^{(1)}_t(z)>0$. Similarly, one has $h^{(2)}_t(z)>0$. Therefore, we conclude that $\mathrm{p}^D_t(x,y)>0$.
\end{proof}

Let us now state and prove Lemma \ref{lem: borne densite indep de n}.

\begin{lemma}[Transition density domination]\label{lem: borne densite indep de n}
Let Assumptions~\ref{hyp O} and~\ref{hyp F1} hold. Let $U$ be a compact set of $\mathbb{R}_+^*\times D$. There exist a constant $C>0$ and a function $h\in\mathrm{L}^1(D)$ such that for all $(t,x)\in U$ and $y\in D$,
$$\mathrm{p}^D_{t}(x,y)\leq C h(y).$$
\end{lemma}
\begin{proof}
Let $U$ be a compact set of $\mathbb{R}_+^*\times D$ and let $(t_0,x_0)$  be a fixed element of $U$. Let $(t,x)\in U$, by Corollary \ref{Rq densite estimation}, for any $\alpha\in(0,1)$, there exists $C_1>0$ such that 
$$\mathrm{p}^D_{t}(x,y)\leq C_1\widehat{\mathrm{p}}^{(\alpha)}_{t}(x,y).$$
Besides, by Proposition \ref{prop:kolmo-langevin} and the equality \eqref{densite p^alpha}, for all $y\in D$, the function $(t,x)\in\mathbb{R}_+^*\times\mathbb{R}^{2d}\mapsto\widehat{\mathrm{p}}^{(\alpha)}_{t}(x,y)$ satisfies $$\partial_t \widehat{\mathrm{p}}^{(\alpha)}=\mathcal{L}_{0,\gamma,\sigma/\sqrt{\alpha}}\widehat{\mathrm{p}}^{(\alpha)}.$$ As a result, the Harnack inequality in Theorem \ref{Harnack} along with Remark \ref{rmk Harnack global}, applied on the compact set $U$, ensure the existence of $C_2>0$ only depending on $U$ such that 
$$\widehat{\mathrm{p}}^{(\alpha)}_{t}(x,y)\leq C_2\widehat{\mathrm{p}}^{(\alpha)}_{t+1}(x_0,y).$$ 
Finally, one has  
\begin{equation}\label{maj p^D harnack}
    \mathrm{p}^D_{t}(x,y)\leq C_1C_2\widehat{\mathrm{p}}^{(\alpha)}_{t+1}(x_0,y),
\end{equation}
where we eliminated the dependence with respect to the variable $x$. It remains to eliminate the dependence with respect to $t$ on the time variable. Consider now the expression of $\widehat{\mathrm{p}}^{(\alpha)}$ following from \eqref{densite p^alpha} and \eqref{expr densite}. Using Lemma \ref{arg densite}, especially the first term in the right-hand side of the equality \eqref{ecriture arg densite}, one has for $x_0=(q_0,p_0)$ and $y=(q',p')\in D$, 
$$\widehat{\mathrm{p}}^{(\alpha)}_{t+1}((q_0,p_0),(q',p'))\leq\frac{\sqrt{\alpha^{2d}}}{ \sqrt{(2 \pi)^{2d} \left(\frac{\sigma^4 (t+1)^4}{12} \phi(\gamma (t+1))\right)^d}} \mathrm{e}^{-\frac{\alpha}{\sigma^2 (t+1)} \left\vert\gamma (q'-q_0)+p'-p_0\right\vert^2}.$$
Hence, since $(t,x)\in U$, $t$ is bounded from above and below by positive constants. Therefore, reinjecting into \eqref{maj p^D harnack} there exist $C_3>0$ and $\beta>0$, which do not depend on $(t,x)$, such that for any $y=(q',p') \in D$,
$$\mathrm{p}^D_{t}(x,y)\leq C_3\mathrm{e}^{-\beta \left\vert\gamma (q'-q_0)+p'-p_0\right\vert^2},$$
which is an integrable function of $y$.
\end{proof}

To complete the proof of Theorem~\ref{thm density intro}, it remains to check that, for any $t>0$, the extension of~$\mathrm{p}^D_t$ constructed in Theorem \ref{boundary property density} remains the density of the kernel $\mathrm{P}_t^D(x,\cdot)$. This is already the case for $x \in D$ by Proposition~\ref{existence densite}, and we prove this fact for $x \in \partial D$ in the next proposition.
 
\begin{proposition}[Identification of the transition density on $\partial D$]\label{prop: extension transition kernel}
Under Assumptions~\ref{hyp O} and~\ref{hyp F1}, for all $t>0$, $x\in \partial D$ and $A\in\mathcal{B}(D)$,
$$\mathrm{P}_t^D(x,A)=\int_A \mathrm{p}_t^D(x,y)\mathrm{d}y.$$
\end{proposition} 
\begin{proof}
Let $t_0>0$, $x_0=(q_0,p_0)\in\partial D$ and $(t_n,x_n)_{n \geq 1}$ be a sequence of points in $\mathbb{R}_+^*\times D$ converging to $(t_0,x_0)$. Let us show that for any open set $A\subset D$, $\mathrm{P}_{t_n}^D(x_n,A)$ admits two limits when $n\rightarrow\infty$ which are $\mathrm{P}_{t_0}^D(x_0,A)$ and $\int_A\mathrm{p}^D_{t_0}(x_0,y)\mathrm{d}y$. This yields the desired equality $\mathrm{P}_{t_0}^D(x_0,A)=\int_A\mathrm{p}^D_{t_0}(x_0,y)\mathrm{d}y$.

Let $A$ be an open subset of $D$. The limit $\mathrm{P}_{t_n}^D(x_n,A)\underset{n\rightarrow\infty}{\longrightarrow}\mathrm{P}_{t_0}^D(x_0,A)$ is a straightforward consequence of Lemmas \ref{couplage Lemma} and \ref{cv indicatrices lemma}, which ensure that 
$$\mathbb{P}(X^{x_n}_{t_n}\in A,\tau^{x_n}_\partial>t_n)\underset{n\rightarrow\infty}{\longrightarrow}\mathbb{P}(X^{x_0}_{t_0}\in A,\tau^{x_0}_\partial>t_0),$$
using the dominated convergence theorem since $\mathbb{P}(X^{x_0}_{t_0}\in \partial A,\tau^{x_0}_\partial>t_0)=0$. 

Let us now prove that $\mathrm{P}_{t_n}^D(x_n,A)=\int_A\mathrm{p}_{t_n}^D(x_n,y)\mathrm{d}y\underset{n\rightarrow\infty}{\longrightarrow}\int_A\mathrm{p}^D_{t_0}(x_0,y)\mathrm{d}y$. In order to do that we apply the dominated convergence theorem on the integrand of $\int_A\mathrm{p}_{t_n}^D(x_n,y)\mathrm{d}y$ which requires an upper bound of $\mathrm{p}_{t_n}^D(x_n,y)$ independent of $n$ for $n$ large enough, and integrable on $D$. Such an upper bound is provided in Lemma \ref{lem: borne densite indep de n} if for all $n\geq1$, $(t_n,x_n)$ is in a compact of $\mathbb{R}_+^*\times\overline{D}$ which is the case since  $t_n\underset{n\rightarrow\infty}{\longrightarrow}t_0>0$ and $x_n\underset{n\rightarrow\infty}{\longrightarrow}x_0\in \overline{D}$. 
\end{proof}

\begin{remark}[Strong Feller property]\label{rk:scheffe}
  Let $(t_n,x_n)_{n\geq1}$ be a sequence in $\mathbb{R}_+^* \times \overline{D}$ converging to $(t,x)\in\mathbb{R}_+^* \times \overline{D}$. By the previous construction, it follows that for any $y\in D$, $\mathrm{p}^D_{t_n}(x_n,y) \underset{n\rightarrow\infty}{\longrightarrow} \mathrm{p}^D_t(x,y)$. Futhermore, the proof of Proposition \ref{prop: extension transition kernel} shows that
  $$\int_D\mathrm{p}^D_{t_n}(x_n,y)\mathrm{d}y\underset{n\rightarrow\infty}{\longrightarrow}\int_D\mathrm{p}^D_{t}(x,y)\mathrm{d}y.$$ 
  These two convergences guarantee by Scheffé's lemma that
  $$\int_D\left\vert\mathrm{p}^D_{t_n}(x_n,y)-\mathrm{p}^D_{t}(x,y)\right\vert\mathrm{d}y\underset{n\rightarrow\infty}{\longrightarrow}0.$$ 
  
  As a consequence, the semigroup $(P^D_t)_{t \geq 0}$ defined on $\mathrm{L}^\infty(\overline{D})$ by 
  \begin{equation*}
      \forall t \geq 0, \quad \forall x \in \overline{D}, \qquad P^D_tf(x) := \mathbb{E}\left[f(X^x_t)\mathbb{1}_{\tau^x_\partial > t}\right]
  \end{equation*}
  satisfies the strong Feller property that $P^D_t f \in \mathcal{C}^b(\overline{D})$ for any $t>0$ and $f \in \mathrm{L}^\infty(\overline{D})$.
\end{remark}

\noindent \textbf{Data availability statement:} Data sharing not applicable to this article as no datasets were generated or analysed during the current study.

\medskip

\noindent \textbf{Acknowledgments:} M. Ramil is supported by the
Région Ile-de- France through a PhD fellowship of the Domaine
d'Intérêt Majeur (DIM) Math Innov.  This work also benefited from the
support of the projects ANR EFI (ANR-17-CE40-0030) and ANR QuAMProcs
(ANR-19-CE40-0010) from the French National Research Agency. Finally,
T. Lelièvre is partially funded by the European Research Council (ERC) under the European Union’s Horizon 2020 research and innovation programme (grant agreement No 810367), project EMC2.

\appendix 
\section{Proof of Lemma~\ref{cste_Harnack}}\label{pf:cste_Harnack}

\begin{proof}
Let $K\subset D$ be a compact set and $\delta_K>0$ be defined
accordingly. Let $k := \sup_{(q,p) \in K} |p|$, $K':=\{(q,p)\in D
  : \vert p\vert\leq k, d_\partial(q)\geq\delta_K \}$ and $M_K
:= k+1$. By Proposition~\ref{prop:sphere} and
Assumption~\ref{hyp:conn}, $K'$ is a connected compact subset of $D$,
and $K \subset K'$.
%
%

Let $T>0$ and $\epsilon:=\delta_K/2$. Let $C>1$ be the constant from Lemma \ref{chemin}. The compact set $K'$ can be covered by $N\geq1$ closed balls  $\overline{\mathrm{B}}(z_1,\frac{ \epsilon}{8C}),\ldots,\overline{\mathrm{B}}(z_N,\frac{ \epsilon}{8C})$ included in $D$ with $z_i=(q_i,p_i)\in K'$ for all $i\in\llbracket1,N\rrbracket$. We can take $N$ large enough so that $\Delta:=\frac{T}{N+1}\in(0,\frac{ \epsilon}{2M_KC}\land1)$. We can now build a graph $\mathcal{G}$ with $N$ vertices corresponding to the points $(z_i)_{1\leq i\leq N}$, and for every $i,j\in\llbracket1,N\rrbracket$, we link $z_i$ to $z_j$ if $\vert z_j-z_i\vert\leq \frac{ \epsilon}{4 C}$. Since the set $K'$ is connected, then so is the graph $\mathcal{G}$.

Furthermore, for all $i,j\in\llbracket1,N\rrbracket$ which are adjacent in $\mathcal{G}$, by Lemma \ref{chemin}, there exists a path $\phi_{i,j}\in\mathcal{C}^{2}([0,\Delta],\mathbb{R}^d)$ such that
\begin{enumerate}[label=(\roman*),ref=\roman*]
    \item $\left(\phi_{i,j}(0),\dot{\phi}_{i,j}(0)\right)=z_i$ and $\left(\phi_{i,j}(\Delta),\dot{\phi }_{i,j}(\Delta)\right)=z_j$
    \item $\sup_{t\in[0,\Delta]}\vert\phi_{i,j}(t)-q_i\vert\leq C(
      \vert z_j-z_i\vert(1+\Delta) +\Delta \vert p_i\vert)\leq
      C\left(2 \frac{  \epsilon}{4C}  + k \frac{ \epsilon}{2M_KC}  \right)< \epsilon$ 
    \item $\sup_{t\in[0,\Delta]}\left(\left\vert\dot{\phi}_{i,j}\right\vert+\left\vert\ddot{\phi}_{i,j}\right\vert\right)(t)\leq 
    C\left(\frac{1}{\Delta}+\frac{1}{\Delta^2}\right)(\vert z_j-z_i\vert(1+\Delta)+\Delta \vert p_i\vert) \leq \epsilon\left(\frac{1}{\Delta}+\frac{1}{\Delta^2}\right)\leq \frac{2  \epsilon}{\Delta^2}$. 
\end{enumerate} 
Since $z_i \in K'$, the second condition ensures that, for all $t \in [0,\Delta]$, $\phi_{i,j}(t)$ remains at a distance from $\partial\mathcal{O}$ strictly larger than $\delta_K-\epsilon\geq\delta_K/2$. 

Now let $x,y\in K'$. By the previous cover, there exist
$i_0,i_N\in\llbracket1,N\rrbracket$ such that $$\vert
x-z_{i_0}\vert\leq \frac{ \epsilon}{8 C},\quad\vert y-z_{i_N}\vert\leq
\frac{ \epsilon}{8 C}.$$  Using Lemma~\ref{chemin} again, we construct
$\psi_0, \psi_N : [0,\Delta] \to \mathcal{O}$ respectively joining $x$
to $z_{i_0}$ and $z_{i_N}$ to $y$ and such that
$\sup_{t\in[0,\Delta]}(\vert\dot{\psi}_0\vert+\vert\ddot{\psi}_0\vert)(t)
\leq \frac{\delta_K}{\Delta^2}$ and
$\sup_{t\in[0,\Delta]}(\vert\dot{\psi}_N\vert+\vert\ddot{\psi}_N\vert)(t)\leq
\frac{\delta_K}{\Delta^2}$. The connectedness of the graph
$\mathcal{G}$ ensures the existence of
$i_1,\dots,i_{N-1}\in\llbracket1,N\rrbracket$ such that for all $j\in\llbracket1,N\rrbracket$, $\vert z_{i_j}-z_{i_{j-1}}\vert\leq\frac{ \epsilon}{4C}$. If the path obtained on the graph $\mathcal{G}$ is smaller than $N$ then we can complete with loops around the same point. It is important here that this path in the graph have exactly $N-1$ vertices, because in the end we aim at constructing a trajectory $\phi$ by piecing together trajectories $\phi_{i_j,i_{j+1}}$, $j=0, \ldots, N$, of length $\Delta = T/(N+1)$ and we want the final trajectory $\phi$ to have exact length $T$. Let us now define the path function $\phi$ on $[0,T]$ as follows:
\begin{equation*}
\phi(t)=
\begin{cases}
    \psi_0(t) &\text{if } t\in [0,\Delta],\\
    \phi_{i_j,i_{j+1}}(t-j\Delta) &\text{if }t\in[j\Delta,(j+1)\Delta],\\
    \psi_N(t-T+\Delta)&\text{if } t\in [T-\Delta,T],
\end{cases}
\end{equation*} then it is easy to see that
$\phi\in\mathcal{H}_{T,x,y,\frac{\delta_K}{\Delta^2},\delta_K/2}$. Since
$K'$ contains the compact set $K$ this concludes the proof for the
compact set $K$. 
\end{proof} 
\section{Proof of Lemmas~\ref{arg densite},~\ref{prop densite},~\ref{lemma:majoration densite} and Proposition~\ref{existence densite}}\label{pf:arg densite}
\begin{proof}[Proof of Lemma~\ref{arg densite}]
On the one hand, from~\eqref{coeff cov} and~\eqref{def moyenne cov},
easy computations show that 
\begin{equation}\label{eq7}
    \delta x \cdot C^{-1}(t)\delta x=\frac{1}{\frac{\sigma^2 t^3}{12} \phi(\gamma t)}\left[\Phi_1(2\gamma t)\vert\delta q\vert^2-\Phi_1(\gamma t)^2\delta q\cdot t\delta p+\frac{1}{3} \Phi_2(\gamma t)\vert t\delta p\vert^2\right]. 
\end{equation}
On the other hand,
\begin{align*}
    &\frac{1}{\sigma^2 t} \left\vert\Pi_1 \delta x\right\vert^2+\frac{12}{\sigma^2 t^3 \phi(\gamma t)} \left\vert\Pi_2(t) \delta x\right\vert^2\\
    &=\frac{1}{\frac{\sigma^2 t^3}{12} \phi(\gamma t)}\left[\left(\frac{(\gamma t)^2}{12}\phi(\gamma t)+\Phi_1(\gamma t)^2\right) \vert\delta q\vert^2+\left(\frac{\gamma t}{6}\phi(\gamma t)-\Phi_1(\gamma t) \Phi_3(\gamma t)\right)\delta q\cdot t\delta p\right.\\
    & \qquad\qquad\qquad \left. + \left(\frac{1}{12}\phi(\gamma t)+\frac{1}{4} \Phi_3(\gamma t)^2\right) \vert t\delta p\vert^2\right],
\end{align*} 
so that the claimed expression follows from the identities
\begin{equation*}
  \frac{\rho^2}{12}\phi(\rho)+\Phi_1(\rho)^2 = \Phi_1(2\rho), \qquad \frac{\rho}{6}\phi(\rho) - \Phi_1(\rho)\Phi_3(\rho) = - \Phi_1(\rho)^2, \qquad \frac{1}{12}\phi(\rho) + \frac{1}{4}\Phi_3(\rho)^2 = \frac{1}{3}\Phi_2(\rho),
\end{equation*}
for all $\rho \in \mathbb{R}$.
\end{proof}

We now move on to the proof of Lemma~\ref{prop densite}.
\begin{proof}[Proof of Lemma~\ref{prop densite}]
The equality~\eqref{eq:ppalpha} easily follows from the formulas
defining $\widehat{\mathrm{p}}_t$ and $\widehat{\mathrm{p}}^{(\alpha)}_{t}$.
Moreover, since $\widehat{\mathrm{p}}^{(\alpha)}_t(x,y)$ is the transition
density of the process
$(\alpha^{-1/2}\widehat{X}^{\sqrt{\alpha}x}_t)_{t\geq0}$,
the Chapman-Kolmogorov relation \eqref{chp kolmogorov} follows from
the Markov property. Let us now prove \eqref{eq8}.  For $\alpha\in(0,1]$, 
$$\int_{\mathbb{R}^{2d}} \widehat{\mathrm{p}}^{(\alpha)}_{t}(x,y) \mathrm{d}x=\int_{\mathbb{R}^{2d}}\frac{\sqrt{\alpha^{2d}}}{ \sqrt{(2 \pi)^{2d} \mathrm{det}(C(t))}} \mathrm{e}^{-\frac{\alpha}{2}\delta x(t) \cdot C^{-1}(t)\delta x(t)} \mathrm{d}x,$$ where $\delta x(t)$ and $C^{-1}(t)$ are defined in \eqref{def moyenne cov}. Let us define the matrix $M(t)$ as follows: $$M(t):=\begin{pmatrix}
I_d & t \Phi_1(\gamma t) I_d \\
0_d & \mathrm{e}^{-\gamma t} I_d 
\end{pmatrix}$$ so that for $x=(q,p),y=(q',p')$ one has $M(t)x=\begin{pmatrix}
m^x_q(t)\\
m^x_p(t)
\end{pmatrix}$. Therefore, $\delta x(t)=y-M(t) x$. As a result, making the following change of variables,
$$x \in \mathbb{R}^{2d} \mapsto z:=y-M(t) x$$ one has $\mathrm{d} z
=\mathrm{e}^{-d \gamma t} \mathrm{d}x$ and one obtains
$$ \int_{\mathbb{R}^{2d}} \widehat{\mathrm{p}}^{(\alpha)}_{t}(x,y) \mathrm{d}x=\mathrm{e}^{d \gamma t} \int_{\mathbb{R}^{2d}}\frac{\sqrt{\alpha^{2d}}}{ \sqrt{(2 \pi)^{2d} \mathrm{det}(C(t))}} \mathrm{e}^{-\frac{\alpha}{2}z\cdot C^{-1}(t) z} \mathrm{d}z=\mathrm{e}^{d \gamma t} .$$

Let us now prove the first convergence in \eqref{chgt de var + cv}. Using the same change of variables as above, one obtains that 
\begin{equation}\label{eq9}
   \int_{\mathbb{R}^{2d}} \widehat{\mathrm{p}}^{(\alpha)}_{t}(x,y) \varphi(t,x) \mathrm{d}x=\mathrm{e}^{d \gamma t} \mathbb{E}\left[\varphi\left(t,M^{-1}(t)(y-Z(t))\right)\right], 
\end{equation} 
where $Z(t)\sim\mathcal{N}_{2d}(0,\frac{C(t)}{\alpha})\overset{\mathcal{L}}{\underset{t\rightarrow0}{\longrightarrow}}0$, since $C(t)\underset{t\rightarrow0}{\longrightarrow}0_{2d}$. Therefore, since $M(t)\underset{t\rightarrow0}{\longrightarrow}I_{2d}$, one has by Slutsky's theorem that $$\left(t,M^{-1}(t)(y-Z(t))\right)\overset{\mathcal{L}}{\underset{(t,y)\rightarrow(0,y_0)}{\longrightarrow}}(0,y_0)$$ which yields the first convergence in \eqref{chgt de var + cv} using \eqref{eq9} and the dominated convergence theorem. The second convergence follows easily with a similar change of variables.

Let us finally prove \eqref{maj gradient p}. By \eqref{expr densite}, \eqref{def moyenne cov} along with Lemma \ref{arg densite} we have 
\begin{align*}
\frac{\nabla_p\widehat{\mathrm{p}}_t((q,p),(q',p'))}{\widehat{\mathrm{p}}_t((q,p),(q',p'))}
&=-\frac{1}{2} \left(\frac{2}{\sigma^2 t}\nabla_p\left(\Pi_1 \delta x(t)\right) \Pi_1 \delta x(t)
+\frac{24}{\sigma^2 t^3\phi(\gamma t)}\nabla_p\left(\Pi_2(t) \delta x(t)\right)\Pi_2(t) \delta x(t)\right) .\numberthis\label{derivee log densite}
\end{align*}
Since $\nabla_p\Pi_1 \delta x(t)=-(\gamma t \Phi_1(\gamma
  t)+\mathrm{e}^{-\gamma t})I_d=-I_d$, the first term in the
right-hand side of the equality~\eqref{derivee log densite} multiplied
by $\widehat{\mathrm{p}}_t$ satisfies
(using~\eqref{eq:ppalpha}, and again Lemma \ref{arg densite} in the first inequality)
\begin{align*}
&\left\vert-\frac{1}{2} \frac{2}{\sigma^2 t}\nabla_p\left(\Pi_1 \delta x(t)\right) \Pi_1 \delta x(t)\right\vert \widehat{\mathrm{p}}_t((q,p),(q',p'))\\ 
&=\frac{1}{\sigma^2 t} \left\vert \Pi_1 \delta x(t) \right\vert \widehat{\mathrm{p}}_t((q,p),(q',p'))\\ 
&=\frac{1}{\sqrt{\alpha^{2d}} \sigma^2 t} \left\vert \Pi_1 \delta x(t) \right\vert \mathrm{e}^{-\frac{1-\alpha}{2}\delta x(t)\cdot C^{-1}(t)\delta x(t)} \widehat{\mathrm{p}}^{(\alpha)}_t((q,p),(q',p'))\\
&\leq\frac{1}{\sqrt{\alpha^{2d}} \sigma^2 t} \left\vert \Pi_1 \delta x(t) \right\vert \mathrm{e}^{-\frac{1-\alpha}{2 \sigma^2 t} \left\vert\Pi_1 \delta x(t)\right\vert^2} \widehat{\mathrm{p}}^{(\alpha)}_t((q,p),(q',p'))\\
&\leq\frac{\sup_{\theta\geq0}\theta \mathrm{e}^{-\frac{1-\alpha}{2} \theta^2}}{\sqrt{\alpha^{2d}} \sigma \sqrt{t}} \widehat{\mathrm{p}}^{(\alpha)}_t((q,p),(q',p')) .\numberthis\label{maj densite 1}
\end{align*} 
Let us now estimate the second term in the right-hand side of the
equality~\eqref{derivee log densite} multiplied
by $\widehat{\mathrm{p}}_t$. Since $\nabla_p\Pi_2(t) \delta
x(t)=(-t \Phi_1(\gamma t)^2+\frac{t}{2} \Phi_3(\gamma
  t)\mathrm{e}^{-\gamma t})I_d$, we have (using the same
reasoning as above)
\begin{align*}
&\left\vert-\frac{1}{2} \frac{24}{\sigma^2 t^3\phi(\gamma t)} \nabla_p\left(\Pi_2(t) \delta x(t)\right)\Pi_2(t) \delta x(t)\right\vert \widehat{\mathrm{p}}_t((q,p),(q',p'))\\   
&=\frac{12 t}{\sigma^2 t^3\phi(\gamma t)} \left\vert-\Phi_1(\gamma t)^2+\frac{1}{2} \Phi_3(\gamma t)\mathrm{e}^{-\gamma t} \right\vert \left\vert \Pi_2(t) \delta x(t) \right\vert \widehat{\mathrm{p}}_t((q,p),(q',p'))\\ 
&\leq\frac{\sqrt{12} t}{\sqrt{\alpha^{2d} \sigma^2 t^3\phi(\gamma t)}} \left\vert-\Phi_1(\gamma t)^2+\frac{1}{2} \Phi_3(\gamma t)\mathrm{e}^{-\gamma t} \right\vert  \frac{\sqrt{12}\left\vert \Pi_2(t) \delta x(t) \right\vert}{\sqrt{\sigma^2 t^3\phi(\gamma t)}} \mathrm{e}^{-\frac{12(1-\alpha)}{2\sigma^2 t^3 \phi(\gamma t)}\left\vert\Pi_2(t) \delta x(t)\right\vert^2} \widehat{\mathrm{p}}^{(\alpha)}_t((q,p),(q',p'))\\
&\leq\frac{\sqrt{12}}{\sqrt{\alpha^{2d}} \sigma\sqrt{t}} \frac{\left\vert-\Phi_1(\gamma t)^2+\frac{1}{2} \Phi_3(\gamma t)\mathrm{e}^{-\gamma t} \right\vert}{\sqrt{\phi(\gamma t)}}  \left(\sup_{\theta\geq0}\theta \mathrm{e}^{-\frac{1-\alpha}{2} \theta^2}\right) \widehat{\mathrm{p}}^{(\alpha)}_t((q,p),(q',p')) .\numberthis\label{maj densite 2}
\end{align*}
Let us now study the behavior of $\frac{\vert-\Phi_1(\rho)^2+\frac{1}{2} \Phi_3(\rho)\mathrm{e}^{-\rho} \vert}{\sqrt{\phi(\rho)}}$ for $\rho\in\mathbb{R}$. We have that
\begin{equation*}
\frac{\left\vert-\Phi_1(\rho)^2+\frac{1}{2} \Phi_3(\rho)\mathrm{e}^{-\rho} \right\vert}{\sqrt{\phi(\rho)}}\left\{
\begin{aligned}
    &\underset{\rho\rightarrow\infty}{\sim}\frac{1}{\sqrt{6} \rho},\\
    &\underset{\rho\rightarrow0}{\longrightarrow}\frac{1}{2},\\
    &\underset{\rho\rightarrow-\infty}{\sim}\frac{\sqrt{\vert\rho\vert}}{\sqrt{6}} .
\end{aligned}
\right. 
\end{equation*}
Therefore there exists a universal constant $c>0$ such that for all $\rho\in\mathbb{R}$
$$\frac{\left\vert-\Phi_1(\rho)^2+\frac{1}{2}
    \Phi_3(\rho)\mathrm{e}^{-\rho} \right\vert}{\sqrt{\phi(\rho)}}\leq
c(1+\sqrt{\rho_-}), $$ where $\rho_-$ is the negative part of
$\rho$. As a result, it follows from \eqref{derivee log densite},
\eqref{maj densite 1} and \eqref{maj densite 2} that there exists a
constant $c_\alpha>0$ depending only on $\alpha \in (0,1)$ such that
for all $t>0$ and $x,y \in \mathbb{R}^{2d}$,
$$\left\vert\nabla_p\widehat{\mathrm{p}}_t((q,p),(q',p'))\right\vert\leq\frac{c_\alpha}{\sigma \sqrt{t}} (1+\sqrt{t \gamma_-}) \widehat{\mathrm{p}}^{(\alpha)}_t((q,p),(q',p')), $$ which concludes the proof of \eqref{maj gradient p}.
\end{proof}  

Consider now the proof of Proposition~\ref{existence densite}.
\begin{proof}[Proof of Proposition~\ref{existence densite}]
 For all $t>0$ and $x \in D$, it follows from~\eqref{eq:PtDPt} that the measure $\mathrm{P}^D_t(x,\cdot)$ is absolutely continuous with respect to the measure $\mathrm{P}_t(x,\cdot)$. Since, by Proposition~\ref{prop:kolmo-langevin}, the latter measure is absolutely continuous with respect to the Lebesgue measure, by the Radon-Nikodym theorem, we deduce that $\mathrm{P}^D_t(x,\cdot)$ possesses a density $\mathrm{q}^D_t(x,\cdot)$ with respect to the Lebesgue measure on $D$. We now study the joint measurability of the mapping $(t,x,y) \mapsto \mathrm{q}_t^D(x,y)$; more precisely, we construct a measurable function $(t,x,y) \mapsto \mathrm{p}_t^D(x,y)$ such that, for all $t>0$ and $x \in D$, $\mathrm{p}^D_t(x,y) = \mathrm{q}_t^D(x,y)$, $\mathrm{d}y$-almost everywhere on $D$.

For all $r>0$, it follows from Proposition~\ref{prop:kolmo-langevin} and Lemmas~\ref{couplage Lemma} and~\ref{cv indicatrices lemma} that the function $$\varphi_r : (t,x,y)\in\mathbb{R}_+^*\times D\times D\mapsto\frac{\mathrm{P}_t^D(x,\mathrm{B}(y,r)\cap D)}{\vert \mathrm{B}(y,r)\vert} = \frac{\mathbb{P}(\vert X^x_t-y\vert<r,\tau^x_\partial>t)}{\vert \mathrm{B}(y,r)\vert}$$
is continuous.  Let $(r_q)_{q\geq1}$ be a sequence of positive real numbers decreasing towards $0$. By definition, for any $t>0$ and $x \in D$, the density $ \mathrm{q}_t^D(x,\cdot )$ is integrable on $D$. As a result, the Lebesgue differentiation theorem states that almost every $y\in D$ is a Lebesgue point, hence $$\forall  t>0, \quad \forall  x\in D, \qquad\varphi_{r_q}(t,x,y)\underset{q\longrightarrow \infty}{{\longrightarrow}}  \mathrm{q}_t^D(x,y)\quad\text{$\mathrm{d}y$-almost everywhere on $D$.}$$
As a consequence, $\mathrm{q}_t^D(x,y)$ coincides, $\mathrm{d}y$-almost everywhere, with the measurable function
\begin{equation*}
  \mathrm{p}_t^D(x,y) := \limsup_{q\rightarrow\infty}\varphi_{r_q}(t,x,y),
\end{equation*}
which completes the proof.
\end{proof}
Last, we show the proof of Lemma \ref{lemma:majoration densite}.
\begin{proof}[Proof of Lemma \ref{lemma:majoration densite}]
Let $y_0=(q_0,p_0) \in \mathbb{R}^{2d}$, $M>0$ and $\alpha \in (0,1)$. Let $\delta_0 > 0$ be small enough for the assertion
  \begin{equation}\label{def beta}
    \forall s\in(0,\delta_0], \qquad  \left(\frac{M}{6}+\vert p_0\vert\right)\left\vert\mathrm{e}^{\gamma s}-1\right\vert\leq \frac{M}{12}
\end{equation}
to hold.

Let $(q',p') \in \mathrm{B}(y_0,M/6)$ and $(q,p)\in \mathbb{R}^{2d}$ satisfying $|p-p'| \geq M/3$. For $s \in (0,\delta_0]$, we consider the transition density $\widehat{\mathrm{p}}^{(\alpha)}_s((q,p),(q',p'))$ defined in \eqref{densite p^alpha}.

One has that  
\begin{align}
  &\widehat{\mathrm{p}}^{(\alpha)}_{s}((q,p),(q',p'))\nonumber\\
  &=\frac{\sqrt{\alpha^{2d}}}{ \sqrt{(2 \pi)^{2d} \left(\frac{\sigma^4 s^4}{12} \phi(\gamma s)\right)^d}} \mathrm{e}^{-\frac{\alpha}{2\sigma^2 s} \left\vert\gamma\delta q+\delta p\right\vert^2-\frac{6\alpha}{\sigma^2 s^3 \phi(\gamma s)} \left\vert\Phi_1(\gamma s) \delta q-\frac{s}{2} \Phi_3(\gamma s)\delta p\right\vert^2},\numberthis\label{expr densite tau_delta}
\end{align}
where 
\begin{align*}
\begin{pmatrix}
\delta q\\
\delta p
\end{pmatrix}= 
\begin{pmatrix}
q'-q-s\Phi_1(\gamma s) p   \\
p'-p \mathrm{e}^{-\gamma  s} 
\end{pmatrix}.
\end{align*} 
Let us start by introducing some notations. Let $m:=1+\vert\gamma\vert\delta_0+\sup_{\vert \rho\vert\leq\vert\gamma\vert\delta_0}\frac{\vert \rho\vert}{2}\Phi_3(\rho)$ and $a:=\frac{M\mathrm{e}^{-\vert\gamma\vert\delta_0}}{4m}$. Let us prove that necessarily
\begin{equation}\label{ineq arg densite}
    \left\vert\gamma\delta q+\delta p\right\vert\geq a\quad\text{or}\quad \left\vert\Phi_1(\gamma s) \delta q-\frac{s}{2} \Phi_3(\gamma s)\delta p\right\vert\geq as, 
\end{equation}
then reinjecting this statement onto the expression \eqref{expr densite tau_delta} of $\widehat{\mathrm{p}}^{(\alpha)}_{s}((q,p),(q',p'))$ we will be able to obtain \eqref{ineq lemma maj densite}.

Assume now that 
$$\left\vert\gamma\delta q+\delta p\right\vert< a\quad\text{and}\quad \left\vert\Phi_1(\gamma s) \delta q-\frac{s}{2} \Phi_3(\gamma s)\delta p\right\vert< as,$$
we will prove that $\Big\vert p'-p\Big\vert<M/3$, thus contradicting the initial assumption on $(q,p)$ and $(q',p')$.

Using the triangle inequality, since $\Phi_1(\rho)+\frac{\rho}{2}\Phi_3(\rho)=1$ for all $\rho\in\mathbb{R}$, one has that
\begin{align*}
    \Big\vert\delta q\Big\vert&=\left\vert\delta q\left(\Phi_1(\gamma s)+\frac{\gamma s}{2}\Phi_3(\gamma s)\right)\right\vert\\
    &\leq \left\vert\delta q\Phi_1(\gamma s)-\frac{s}{2}\Phi_3(\gamma s)\delta p \right\vert+\frac{s}{2}\Phi_3(\gamma s)\Big\vert\gamma\delta q+\delta p \Big\vert\\
    &< a\left(s+\frac{s}{2}\Phi_3(\gamma s)\right).
\end{align*}
As a result,
$$  \Big\vert\delta p\Big\vert\leq\Big\vert\gamma\delta q+\delta
                              p\Big\vert+\vert\gamma\vert\Big\vert\delta
                              q\Big\vert <a\left(1+\vert\gamma\vert s+\frac{\vert\gamma\vert s}{2}\Phi_3(\gamma s)\right).$$   
Since $s \leq \delta_0$, one obtains that $\Big\vert\delta p\Big\vert< a m$. Therefore $\Big\vert\delta p\Big\vert<\frac{M\mathrm{e}^{-\vert\gamma\vert\delta_0}}{4}$.
 Furthermore, by the triangle inequality, for $(q',p') \in \mathrm{B}(y_0,M/6)$,
\begin{align*}
    \Big\vert p'-p\Big\vert&\leq \mathrm{e}^{\gamma s}\underbrace{\Big\vert p'-p\mathrm{e}^{-\gamma s}\Big\vert}_{=\Big\vert\delta p\Big\vert}+\Big\vert p'-p_0\Big\vert\Big\vert\mathrm{e}^{\gamma s}-1\Big\vert+\Big\vert p_0\Big\vert\Big\vert\mathrm{e}^{\gamma s}-1\Big\vert\\
    &< \mathrm{e}^{\vert\gamma\vert\delta_0}\frac{M\mathrm{e}^{-\vert\gamma\vert\delta_0}}{4}+\frac{M}{6}\Big\vert\mathrm{e}^{\gamma s}-1\Big\vert+\Big\vert p_0\Big\vert\Big\vert\mathrm{e}^{\gamma s}-1\Big\vert\\
    &\leq \frac{M}{4}+\left(\frac{M}{6} + |p_0|\right)\left|\mathrm{e}^{\gamma s}-1\right|\leq \frac{M}{4}+\frac{M}{12}=\frac{M}{3},
\end{align*}
by \eqref{def beta} since $s \leq \delta_0$, hence \eqref{ineq arg densite}.

Reinjecting the inequality \eqref{ineq arg densite} into \eqref{expr densite tau_delta}, we get
\begin{equation*}
  \widehat{\mathrm{p}}^{(\alpha)}_{s}((q,p),(q',p')) \leq \frac{\sqrt{\alpha^{2d}}}{\sqrt{(2\pi)^{2d}(\frac{\sigma^4 s^4}{12}\phi(\gamma s))^d}}\exp\left(-\frac{1}{s}\min\left(\frac{a^2\alpha}{2\sigma^2},\frac{6 a^2\alpha}{\sigma^2 \phi(\gamma s)}\right)\right),
\end{equation*}
and using the fact that $\phi(\gamma s)$ is a positive and bounded continuous function for $s\in[-\vert\gamma\vert \delta_0,\vert\gamma\vert \delta_0]$, it follows that there exist some constants $C_0 \geq 0$ and $\mu>0$, which only depend on $\gamma$, $\sigma$, $M$, $\alpha$ and $\delta_0$, such that for any $s \in (0,\delta_0]$,
\begin{equation*}
  \widehat{\mathrm{p}}^{(\alpha)}_{s}((q,p),(q',p')) \leq C_0 \mathrm{e}^{-\mu/s},
\end{equation*}
which completes the proof.
\end{proof}
\bibliographystyle{plain}
\bibliography{biblio}

\end{document}